\documentclass[a4paper,11pt,openright,twoside]{book}
\usepackage[english]{babel}
\usepackage[utf8]{inputenc}
\usepackage[T1]{fontenc}
\usepackage{lmodern}
\usepackage{hyperref}
\usepackage{latexsym}
\usepackage{amsthm}
\usepackage{a4}
\usepackage{amsfonts}
\usepackage{bm}
\usepackage{amssymb}
\usepackage{amsmath}
\usepackage{color}
\usepackage{eucal}
\usepackage{amscd}
\usepackage{amsmath,mathtools}
\usepackage{array}
\usepackage{float}
\usepackage{setspace}
\usepackage{listings}
\usepackage{color} 
\usepackage{graphicx}
\usepackage{xifthen}
\usepackage{stmaryrd}
\usepackage{linguex}
\usepackage{textcomp}
\usepackage[backend=biber]{biblatex}
\usepackage{csquotes}
\usepackage{subfig}
\allowdisplaybreaks

\newcommand{\bigtau}{\scalebox{1.44}{$\tau \hspace{-0.06cm}$}}
\newcommand{\dlgraffa}{\{ \hspace{-0.1cm} \{}
\newcommand{\drgraffa}{\} \hspace{-0.1cm} \}}
\newcommand{\vertiii}[1]{{\left\vert\kern-0.25ex\left\vert\kern-0.25ex\left\vert #1 \right\vert\kern-0.25ex\right\vert\kern-0.25ex\right\vert}}

\newboolean{english}
\setboolean{english}{true} 
\newboolean{corelatore}
\setboolean{corelatore}{false} 

\title{Titolo}
\author{Armando Maria Monforte}

\begin{document}
\theoremstyle{plain} 
\newtheorem{thm}{Theorem}[chapter]
\newtheorem{lem}[thm]{Lemma}
\newtheorem{corol}[thm]{Corollary}
\newtheorem{prop}[thm]{Proposition}

\theoremstyle{definition}
\newtheorem{defin}[thm]{Definition} 

\theoremstyle{remark}
\newtheorem{rem}[thm]{Remark}
\newtheorem{example}[thm]{Example}

\thispagestyle{empty}
\space
\begin{center}
\textsc{\Large{Università degli studi di Pavia}\\
\normalsize{Dipartimento di Matematica\\
Corso di Laurea Magistrale in Matematica}}
\end{center}
\[\]
\begin{center}
	\includegraphics[width=0.35\textwidth]{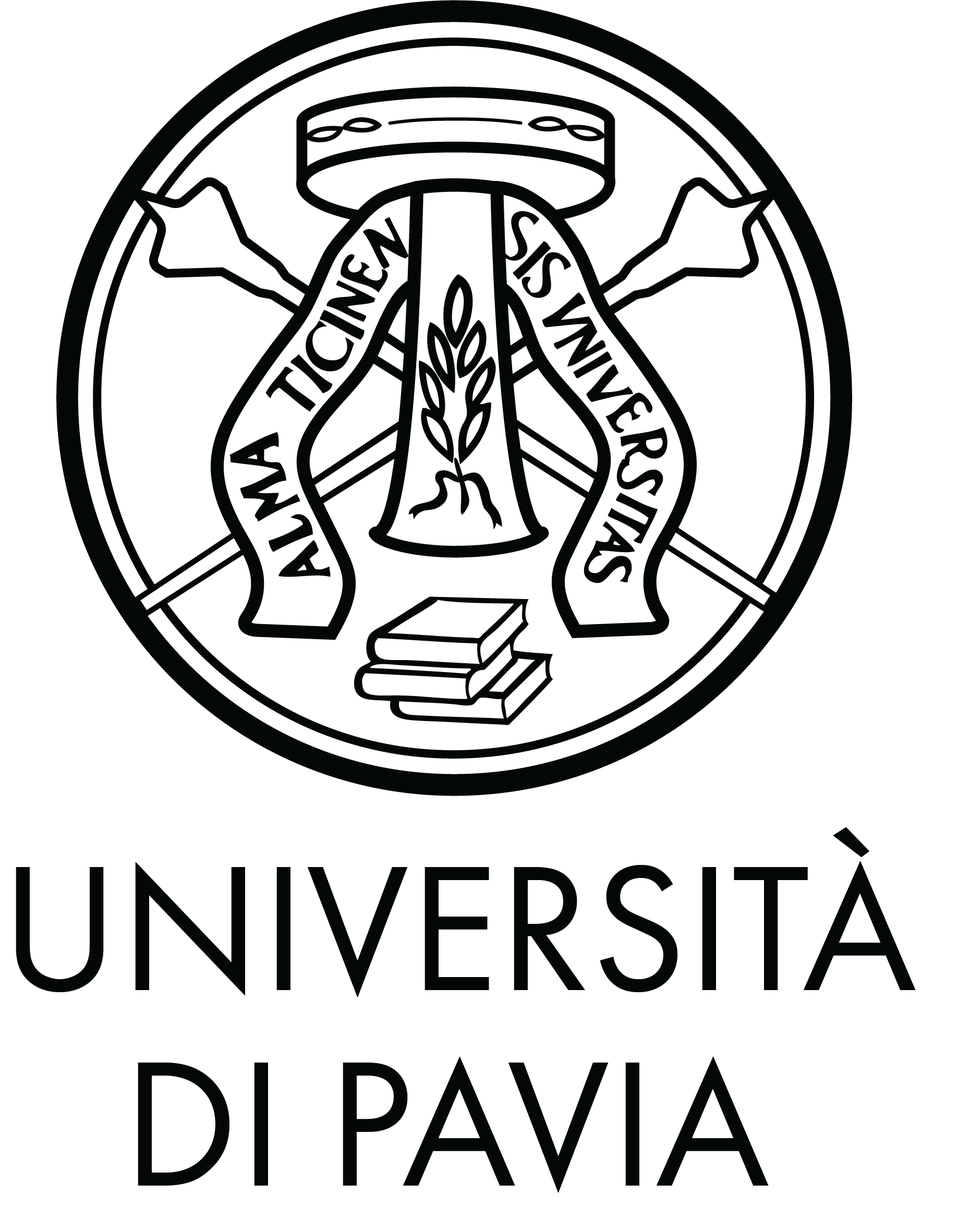}
\end{center}
\[\]
\space
\begin{center}
\ifthenelse{\boolean{english}} 
{
\smallskip

\textbf{\Large{PLANE WAVE DISCONTINUOUS GALERKIN METHODS FOR SCATTERING BY PERIODIC STRUCTURES}}
}
{}
\end{center}
\singlespace
\[\]
\begin{center}
\textbf{Tesi di Laurea Magistrale in Matematica}
\end{center}
\[\]\[\]\[\]
\begin{flushleft}
\ifthenelse{\boolean{english}}{Relatore (Supervisor):\\}{Relatore:\\}
\textbf{Prof. Andrea Moiola}
\ifthenelse{\boolean{corelatore}}{
\ifthenelse{\boolean{english}}{\\Correlatore (Co-Supervisor):\\}{\\Correlatore:\\}
\textbf{Nome Prof}
}
{}
\end{flushleft}
\[\]\[\]
\begin{flushright}
Tesi di Laurea di: \\
\textbf{Armando Maria Monforte}\\
Matricola 507047
\end{flushright}
\[\]\[\]\[\]
\begin{center}
Anno Accademico 2022-2023
\end{center}

\newpage\null\thispagestyle{empty}\newpage

\ifthenelse{\boolean{english}}
{\chapter*{Abstract}}{\chapter*{Riassunto}}
This thesis explores the application of Plane Wave Discontinuous Galerkin (PWDG) methods for the numerical simulation of electromagnetic scattering by periodic structures. Periodic structures play a pivotal role in various engineering and scientific applications, including antenna design, metamaterial characterization, and photonic crystal analysis. Understanding and accurately predicting the scattering behavior of electromagnetic waves from such structures is crucial in optimizing their performance and advancing technological advancements.

The thesis commences with an overview of the theoretical foundations of electromagnetic scattering by periodic structures. This theoretical dissertation serves as the basis for formulating the PWDG method within the context of wave equation. The DtN operator is presented and it is used to derive a suitable boundary condition.

The numerical implementation of PWDG methods is discussed in detail, emphasizing key aspects such as basis function selection and boundary conditions. The algorithm's efficiency is assessed through numerical experiments.

We then present the DtN-PWDG method, which is discussed in detail and is used to derive numerical solutions of the scattering problem. A comparison with the finite element method (FEM) is presented.

In conclusion, this thesis demonstrates that PWDG methods are a powerful tool for simulating electromagnetic scattering by periodic structures.

\chapter*{Riassunto}
Questa tesi esplora l'applicazione dei metodi Plane Wave Discontinuous Galerkin (PWDG) per la simulazione numerica dello scattering elettromagnetico da parte di strutture periodiche. Le strutture periodiche svolgono un ruolo fondamentale in varie applicazioni scientifiche e ingegneristiche, tra cui la progettazione di antenne, la caratterizzazione dei metamateriali e l'analisi dei cristalli fotonici. Comprendere e prevedere con precisione il comportamento di diffusione delle onde elettromagnetiche provenienti da tali strutture è fondamentale per ottimizzarne le prestazioni e far avanzare i progressi tecnologici.

La tesi inizia con una panoramica dei fondamenti teorici della diffusione elettromagnetica da parte di strutture periodiche. Questa analisi teorica serve come base per la formulazione del metodo PWDG nel contesto dell'equazione delle onde. Viene presentato l'operatore DtN che viene utilizzato per derivare una adeguata condizione al bordo.

L'implementazione numerica dei metodi PWDG viene discussa nel dettaglio, sottolineando aspetti chiave come la selezione delle funzioni di base e le condizioni al bordo. L'efficienza  dell'algoritmo viene valutata attraverso esperimenti numerici.

Presentiamo poi il metodo DtN-PWDG, che viene discusso in dettaglio e viene utilizzato per derivare soluzioni numeriche del problema di scattering. Vengono presentati confronti con il metodo degli elementi finiti (FEM).

In conclusione, questa tesi dimostra che i metodi PWDG sono un potente strumento per simulare lo scattering elettromagnetico su strutture periodiche.

\doublespacing
\tableofcontents
\singlespacing

\chapter{Introduction}
The study of electromagnetic scattering by periodic structures has been an area of significant interest in the field of electromagnetics and computational physics for many decades. As technology continues to advance, the demand for accurate and efficient numerical methods to analyze scattering phenomena from periodic structures becomes increasingly important. In this context, the Trefftz Discontinuous-Galerkin (TDG) method emerges as a promising approach to tackle the challenges posed by these complex problems.\\

This thesis explores the application and development of the Trefftz-DG methods for solving scattering problems involving periodic structures. The motivation behind this work arises from the desire to overcome the limitations of traditional numerical techniques, such as finite element methods, finite difference methods, and integral equation methods, when dealing with problems featuring periodicity.\\

The primary objective of this research is to present a comprehensive study of the TDG method's theoretical foundation and its practical implementation for scattering analysis. By leveraging the power of Trefftz functions, which are solutions to the governing equations of the problem, the TDG method provides a unique advantage in efficiently capturing the behavior of waves in periodic media.\\

The thesis starts with a review of the relevant literature on scattering problems, periodic surfaces, Trefftz methods, and discontinuous Galerkin methods. The mathematical formulation of the scattering problem is presented, including the governing equations, boundary conditions, and the discretization strategy based on TDG. The periodicity of the scattering surface leads to quasi-periodic solutions. For this reason, quasi-periodic boundary conditions are applied, together with a radiation condition on part of the boundary. \\

Next, we develop and implement numerical algorithms for solving the scattering problem: we implement the PWDG method using suitable finite element discretization techniques, incorporating quasi-periodic boundary conditions and appropriate basis functions. Special attention is given to the treatment of quasi-periodic boundary conditions and the construction of the basis functions using appropriate solutions of the governing equations.\\

To assess the performance and accuracy of the proposed TDG method, a series of numerical experiments are conducted. These experiments in two space dimensions involve various scattering scenarios by periodic structures, including different incident wave angles, frequencies, and material properties. \\

Finally, the conclusion summarizes the findings of this thesis and highlights the significance of the Trefftz Discontinuous-Galerkin method in the context of scattering by periodic structures. It also points out potential avenues for future research and developments in this area.\\

\chapter{Scattering by Periodic Structures} \label{chap2}
Scattering phenomena play a fundamental role in a wide range of scientific and engineering disciplines, including optics, acoustics, electromagnetics, and solid-state physics. One intriguing aspect of scattering is wave interaction with periodic structures, where the arrangement of constituents repeats periodically in space. Such periodic structures can be found in various natural and engineered systems, ranging from crystalline materials to photonic crystals, metamaterials, and diffraction gratings.

The study of scattering by periodic structures is crucial for understanding and manipulating wave propagation in materials and devices. It finds applications in diverse fields, including the design of optical filters, waveguides, sensors, and photonic devices with tailored properties. Moreover, scattering by periodic structures is not limited to electromagnetic waves but also extends to other types of waves, such as acoustic waves and matter waves, enabling advancements in fields like ultrasound imaging, sonic crystals, and atom optics.\\
Due to its multiple applications in engineering and technology, considerable attention has been devoted to the mathematical modeling and computational simulation of acoustic and electromagnetic wave diffraction by periodic or bi-periodic structures in unbounded domains.

\section{Functional Space Framework}
\subsection{General Notation}
First, we will briefly describe the functional spaces of interest for our problem, following \cite{Pinto}. We want to define the Quasi-periodic Sobolev space, which will be used to prove some theoretical results.\\

Let $B$ be a Banach space. We denote its norm as $\| \cdot \|_{B}$ and its dual space by $B'$, where we consider elements of $B$ as antilinear rather than linear forms over $B$. If $B$ is Hilbert, we shall denote the inner product between two elements $x,y \in B$ as $(x, y)_B$. For the special case $B = \mathbb{R}^2$ or $\mathbb{C}^2$, we write the inner product between $x, y \in B$ as $x \cdot \overline{y}$.\\
For $d = 1, 2$, let $\Omega \subset \mathbb{R}^d$ be an open domain and denote its boundary $\partial \Omega$.\\
We denote the set of continuous scalar functions in $\Omega$ with complex values as $\mathit{C}(\Omega)$ and define, for $n \in \mathbb{N}_0 $, the following spaces of continuous functions:
\begin{align*}
    &\mathit{C}^n(\Omega) := \{u \in \mathit{C}(\Omega) \hspace{0.15cm} | \hspace{0.15cm} \partial^\beta u \in \mathit{C}(\Omega), \hspace{0.15cm}\forall \beta \in \mathbb{N}^d, \hspace{0.15cm}\text{with}\hspace{0.15cm} |\beta| \leq n \}, \\
    &\mathit{C}^\infty(\Omega) := \{u \in \mathit{C}(\Omega) \hspace{0.15cm} | \hspace{0.15cm} \partial^\beta u \in \mathit{C}(\Omega), \hspace{0.15cm}\forall \beta \in \mathbb{N}^d \}, \\
    &\mathit{D}(\Omega) := \{u \in \mathit{C}^\infty(\Omega) \hspace{0.15cm} | \hspace{0.15cm} \text{supp}\hspace{0.05cm} u \subset \subset \Omega \},    
\end{align*}
where the multi-index $\beta = (\beta_1, \beta_2) \in \mathbb{N}^2$, with $|\beta| = \beta_1 + \beta_2$.\\

We say that a one-dimensional curve $\Gamma$ is of class $\mathit{C}^{r,1}$, for $r \in \mathbb{N}_0$, if it may be parametrized by a continuous function $z : (0, 2\pi) \to \Gamma$ so that $z$ has continuous derivatives up to order $r$ and its derivatives of order $r$ are Lipschitz continuous.\\

The space of antilinear distributions on $\Omega$ is referred to as $\mathit{D}'(\Omega)$, and its duality pairing with $\mathit{D}(\Omega)$ is written as
\begin{equation*}
    f(u) = \langle f \; , \; u \rangle_\Omega,
\end{equation*}
for any $f\in\mathit{D}'(\Omega)$  and $u\in \mathit{D}(\Omega)$.\\

For $s \geq 0$ and $p \geq 1$, $W^{s,p}(\Omega)$ denotes standard Sobolev spaces on $\Omega$. For $s \in \mathbb{R}$, we also introduce the spaces $H^s(\Omega)$ and $H^s_0(\Omega)$ as the usual Sobolev spaces of order $s$.

\subsection{Quasi-periodic Functions}
We now introduce the concept of quasi-periodic functions, which are fundamental for our problem formulation. 
\begin{defin}[Quasi-periodic function]
    A function $u:\mathbb{R}^2 \to \mathbb{C}$ is called quasi-periodic of period $T$ with parameter $\alpha>0$ if
    \begin{equation*}
        u(x_1+T,x_2)=e^{i\alpha T}u(x_1,x_2), \hspace{0.15cm} \forall \hspace{0.05cm} \mathbf{x} = (x_1,x_2) \in \mathbb{R}^2.
    \end{equation*}
    For any $n \in \mathbb{Z}$ and for all $\mathbf{x} \in \mathbb{R}^2$, it holds
    \begin{equation*}
        u(x_1+nT,x_2)=e^{in\alpha T}u(x_1,x_2), \hspace{0.15cm} \forall \hspace{0.05cm} \mathbf{x} = (x_1,x_2) \in \mathbb{R}^2.
    \end{equation*}
    This is equivalent to say that the function $x_1 \mapsto e^{-i\alpha x_1}u(x_1,x_2)$ is $T$-periodic for every $x_2$.
\end{defin} 
We define $\mathit{Q}_{\alpha,T}(\mathbb{R}^2)$ as the set of $\mathit{C}^\infty(\mathbb{R}^2)$-functions vanishing for large $|x_2|$ and that are $\alpha$-quasi-periodic with period $T$ in the first Cartesian component.\\\\
In order to properly define quasi-periodic distributions, we introduce the translation operator
\begin{defin}[Translation operator]
    Let $s\in\mathbb{R}$, $j\in\mathbb{N}$. We define the translation operator:
    \begin{equation*}
        \tau_s := \mathit{D}(\mathbb{R}^2) \to \mathit{D}(\mathbb{R}^2) : \psi(\mathbf{x}) \mapsto \psi(\mathbf{x}+s\mathbf{e}_1),
    \end{equation*}
    where $\{ \mathbf{e}_1, \mathbf{e}_2 \}$ is the canonical orthonormal basis of $\mathbb{R}^2$.\\
    It easily follows that $\tau_{s_1} \circ \tau_{s_2} = \tau_{s_1+s_2}$.
\end{defin} 
We say that $f \in L^1_{\text{loc}}(\mathbb{R}^2)$ is $\alpha$-quasi-periodic with period $T$ if, for every $\psi \in \mathit{D}(\mathbb{R}^2)$ and for all $n \in \mathbb{Z}$, 
\begin{equation*}
    \int_{\mathbb{R}^2} f(\mathbf{x}) \overline{\tau_{nT}\psi(\mathbf{x})} \hspace{0.05cm} d\mathbf{x} = \int_{\mathbb{R}^2} f(\mathbf{x}-nT\mathbf{e}_1) \overline{\psi(\mathbf{x})} \hspace{0.05cm} d\mathbf{x} = e^{-in\alpha T} \int_{\mathbb{R}^2} f(\mathbf{x}) \overline{\psi(\mathbf{x})} \hspace{0.05cm} d\mathbf{x}.
\end{equation*}
\begin{defin}[Quasi-periodic distribution]
    We say that a distribution $f \in \mathit{D}'(\mathbb{R}^2)$ is quasi-periodic with shift $\alpha$ and period $T$ if, for all $n \in \mathbb{Z}$ and $\psi \in \mathit{D}(\mathbb{R}^2)$, there holds
    \begin{equation*}
        \langle f, \tau_{nT}\psi \rangle_{\mathbb{R}^2} = e^{-in\alpha T} \langle f, \psi \rangle_{\mathbb{R}^2}.
    \end{equation*}
\end{defin} \noindent
We write $\mathit{Q}'_{\alpha,T}(\mathbb{R}^2)$ for the space of quasi-periodic distributions with shift $\alpha$ and period $T$. Furthermore, one can check that $\mathit{Q}'_{\alpha,T}(\mathbb{R}^2)$ corresponds to the space of antilinear maps from $\mathit{Q}_{\alpha,T}(\mathbb{R}^2)$ to the complex field. \\

From now on, we will restrict ourselves to quasi-periodic functions and distributions with period $2\pi$ in $x_1$, which we simply refer to as $\alpha$-quasi-periodic. Thus, we employ the equivalences $\mathit{Q}_{\alpha}(\mathbb{R}^2) \equiv \mathit{Q}_{\alpha,2\pi}(\mathbb{R}^2)$ and $\langle\cdot,\cdot\rangle_{\alpha}\equiv \langle\cdot,\cdot\rangle_{\alpha,2\pi} $. Furthermore, it is clear that elements of $\mathit{Q}_{\alpha}(\mathbb{R}^2)$ and $\mathit{Q}'_{\alpha}(\mathbb{R}^2)$ may formally be written as Fourier series.
\begin{prop}[Fourier expansion] \label{fourier}
    Let $\alpha_n = \alpha + n$ for all $n \in \mathbb{Z}$. Every $u \in \mathit{Q}_{\alpha}(\mathbb{R}^2)$ and $F\in\mathit{Q}'_{\alpha}(\mathbb{R}^2)$ may be represented as a Fourier series, i.e.
    \begin{align*}
        u(x_1, x_2) = \sum_{n \in \mathbb{Z}} u_n(x_2) e^{i \alpha_n x_1}, \\
        F(x_1,x_2) = \sum_{n \in \mathbb{Z}} F_n(x_2) e^{i \alpha_n x_1},
    \end{align*}
    where the coefficients $u_n$ and $F_n$ belong respectively to $\mathit{D}(\mathbb{R})$ and $\mathit{D}'(\mathbb{R})$ .
\end{prop} \noindent
The coefficients $u_n$ are defined as, for $n \in \mathbb{Z}$, as
\begin{equation} \label{fourier_coeff}
    u_n(x_2) = \frac{1}{2\pi} \int_0^{2 \pi} e^{-i\alpha_n x_1} u(x_1,x_2) \hspace{0.05cm} dx_1.
\end{equation}
To avoid redundancies, we limit the range of $\alpha$ to $[0, 1)$. Indeed, notice
that $\alpha$ and $\alpha + n$ define the same quasi-periodic space for any integer $n$.

\subsection{Quasi-periodic Sobolev Spaces}
Let $\mathcal{G}:= \{ \mathbf{x} \in \mathbb{R}^2 : 0 <x_1 <2\pi \}$. We define $\mathit{Q}_\alpha(\mathcal{G})$ as the set of restrictions to $\mathcal{G}$ of elements of $\mathit{Q}_\alpha(\mathbb{R}^2)$. For any open set $\mathcal{O} \subset \mathcal{G}$, we say that $u \in \mathit{Q}_\alpha(\mathcal{O})$ if $u = U_{|\mathcal{O}}$ for some $U\in \mathit{Q}_\alpha(\mathcal{G})$ such that $\overline{\text{supp}U}^{\mathcal{G}} \subset \mathcal{O}$. We define the restriction of $F\in\mathit{Q}'_\alpha(\mathbb{R}^2)$ to $\mathcal{O}$, denoted $F_{|\mathcal{O}}$, as an antilinear map acting on $u\in\mathit{Q}_\alpha(\mathcal{O})$ in the following way:
\begin{equation*}
    \langle F_{|\mathcal{O}} , u \rangle_{\mathcal{O},\alpha} := \langle F,U \rangle _{\alpha},
\end{equation*}
where $U$ is an extension of $u$ to $\mathcal{G}$. This can be easily extended to an element of $\mathit{Q}_\alpha(\mathbb{R}^2)$. Furthermore, we introduce the space of these restrictions as $\mathit{Q}'_\alpha(\mathcal{O})$.
\begin{defin}[Quasi-periodic Sobolev spaces] \label{sobolev}
    Let $s \in \mathbb{R}$ and $\alpha_n = \alpha +n$ for all $n \in \mathbb{Z}$; recall that $u_n$ denotes the $n$-th Fourier coefficient of $u$ as in Proposition \ref{fourier}. We introduce the quasi-periodic Bessel potential
    \begin{equation*}
        \mathcal{J}_\alpha^s(u) (\xi_1,\xi_2) := \sum_{n\in \mathbb{Z}} (1+\alpha_n^2+|\xi_2|^2)^{\frac{s}{2}} \hat{u}_n(\xi_2) e^{i\alpha_n \xi_1}, \hspace{0.3cm} \forall \xi_1,\xi_2 \in \mathcal{G},
    \end{equation*}
    where $\hat{u}_n(\xi_2)$ are the Fourier transforms of $u_n(x_2)$ in distributional sense. We define the quasi-periodic Sobolev space $H^s_\alpha(\mathcal{G})$ as
    \begin{equation*}
        H^s_\alpha(\mathcal{G}) := \biggl\{  u\in\mathit{Q}'_\alpha(\mathbb{R}^2) : u=\sum_{n\in\mathbb{Z}} u_ne^{i\alpha_nx_1}, u_n \in H^s(\mathbb{R}) \hspace{0.2cm} \text{and} \hspace{0.2cm} \mathcal{J}_\alpha^s(u) \in L^2_\alpha(\mathcal{G}) \biggl\},
    \end{equation*}
    with
    \begin{equation*}
        L^2_\alpha(\mathcal{G}) \equiv H^0_\alpha(\mathcal{G}) := \biggl\{ u\in\mathit{Q}'_\alpha(\mathbb{R}^2) : \sum_{n\in\mathbb{Z}} \| u_n \|^2_{L^2(\mathbb{R})} < \infty \biggl\}.
    \end{equation*}
\end{defin} \noindent
We shall identify elements of $H^s_\alpha(\mathcal{G})$ with their restrictions to $\mathcal{G}$.\\
$L^2_\alpha(\mathcal{G})$ is a Hilbert space with inner product given by
\begin{equation*}
    (u,v)_{L^2_\alpha(\mathcal{G})} := \sum_{n\in\mathbb{Z}} (u_n,v_n)_{L^2(\mathbb{R})}.
\end{equation*}
We consider the following norm for every $u$ in $H^s_\alpha(\mathcal{G})$:
\begin{equation*}
    \|u \|_{H^s_\alpha(\mathcal{G})} := \| \mathcal{J}_\alpha^s(u) \|_{L^2_\alpha(\mathcal{G})},
\end{equation*}
from where it follows directly that the quasi-periodic Bessel potential, as an operator, is an isometric isomorphism between $H^s_\alpha(\mathcal{G})$ and $L^2_\alpha(\mathcal{G})$. 
\begin{prop}
    Let $s \in \mathbb{R}$. Then, $\mathit{Q}_\alpha(\mathcal{G})$ is dense in $H^s_\alpha(\mathcal{G})$ and $H^s_\alpha(\mathcal{G})$ is a Hilbert space with inner product and norm respectively defined as
    \begin{equation*}
        (u,v)_{H^s_\alpha(\mathcal{G})} := (\mathcal{J}_\alpha^s(u),\mathcal{J}_\alpha^s(v))_{L^2_\alpha(\mathcal{G})}, \hspace{0.5cm} \| u \|_{H^s_\alpha(\mathcal{G})} := (u,u)^{1/2}_{H^s_\alpha(\mathcal{G})}.
    \end{equation*}
\end{prop} \noindent
We can characterize the dual of $H^s_\alpha(\mathcal{G})$ by an isometric isomorphic space to $H^{-s}_\alpha(\mathcal{G})$ and we have that, for $s>0$, $H^s_\alpha(\mathcal{G}) \subseteq L^2_\alpha(\mathcal{G}) \subseteq H^{-s}_\alpha(\mathcal{G})$.
We now define the quasi-periodic Sobolev space on a subset $\mathcal{O} \subset \mathcal{G}$.
\begin{defin}
    Let $s \in \mathbb{R}$. For an open subset $\mathcal{O} \subset \mathcal{G}$, we define $H^s_\alpha(\mathcal{O})$ as the space of restrictions of elements of $H^s_\alpha(\mathcal{G})$ to $\mathcal{O}$, i.e.
    \begin{equation*}
        H^s_\alpha(\mathcal{O}) := \{ u \in \mathit{Q}'_\alpha(\mathcal{O}) : u = U_{|\mathcal{O}} \hspace{0.2cm} \text{and} \hspace{0.2cm} U \in H^s_\alpha(\mathcal{G}) \}.
    \end{equation*}
\end{defin} \noindent
It is possible to verify that $H^s_\alpha(\mathcal{O})$ is a Hilbert space.\\

Now we prove that the spaces defined as the closure of $\mathcal{Q}_\alpha(\mathcal{G})$ for positive $s\in\mathbb{Z}$, i.e. $\overline{\mathcal{Q}_\alpha(\mathcal{G})}^{\| \cdot \|_{H^s(\mathcal{G})}}$, are equivalent to the spaces $H^s_\alpha(\mathcal{G})$ in Definition \ref{sobolev}. For spaces of non-integer order, the result follows by direct application of interpolation theory.
\begin{thm}
    For any $0 \leq s < \infty$, the norms $\| \cdot \|_{H^s_\alpha(\mathcal{G})}$ and $\| \cdot \|_{H^s(\mathcal{G})}$ are equivalent in the space $H^s_\alpha(\mathcal{G})$. Hence, $H^s_\alpha(\mathcal{G})$ may be equivalent defined as $\overline{\mathcal{Q}_\alpha(\mathcal{G})}^{\| \cdot \|_{H^s(\mathcal{G})}}$.
\end{thm}
\begin{proof}
    Take $s\in\mathbb{N}_0$, then $H^s(\mathcal{G})=W^{s,2}(\mathcal{G})$. For $u\in\mathcal{Q}_\alpha(\mathcal{G})$, it holds
    \begin{align*}
        \| u \|^2_{H^s(\mathcal{G})} & = \sum_{|\beta| \leq s} \left\| D^{\beta}u \right\|^2_{L^2(\mathcal{G})} = \sum_{|\beta| \leq s} \left\| D^\beta \sum_{n\in \mathbb{Z}} u_n(x_2)e^{i\alpha_n x_1} \right\|^2_{L^2(\mathcal{G})} \\
        & = \sum_{\beta_1+\beta_2 \leq s} \left\| \sum_{n\in \mathbb{Z}} D^{\beta_2}u_n(x_2)(i\alpha_n)^{\beta_1}e^{i\alpha_n x_1} \right\|^2_{L^2(\mathcal{G})} \\
        & = \sum_{\beta_1+\beta_2 \leq s} \sum_{n\in \mathbb{Z}} (\alpha_n)^{2\beta_1} \left\| D^{\beta_2}u_n(x_2) \right\|^2_{L^2(\mathbb{R})} \\    
        & \simeq \sum_{n\in \mathbb{Z}} \sum_{\beta_1+\beta_2 \leq s} (\alpha_n)^{2\beta_1} \left\| \xi_2^{\beta_2}\hat{u}_n(\xi_2) \right\|^2_{L^2(\mathbb{R})} \\
        & = \sum_{n\in \mathbb{Z}} \sum_{\beta_1+\beta_2 \leq s}  \left\| (\alpha_n)^{\beta_1} \xi_2^{\beta_2}\hat{u}_n(\xi_2) \right\|^2_{L^2(\mathbb{R})} \\
        & \simeq \sum_{n\in \mathbb{Z}} \left\| (1+\alpha_n^2+\xi_2^2)^{\frac{s}{2}} \hat{u}_n(\xi_2) \right\|^2_{L^2(\mathbb{R})} = \| u \|^2_{H^s_\alpha(\mathcal{G})}.
    \end{align*}
    The extension to $u\in H^s_\alpha(\mathcal{G})$ follows by the density of $\mathcal{Q}_\alpha(\mathcal{G})$ in $H^s_\alpha(\mathcal{G})$.
\end{proof}
\begin{corol}
    Let $\mathcal{O}$ be an open set such that $\overline{\mathcal{O}} \subset \mathcal{G}$; then 
    \begin{equation*}
        H^s_\alpha(\mathcal{O}) = H^s(\mathcal{O}), \hspace{0.5cm} \forall s \in \mathbb{R}.
    \end{equation*}
\end{corol}
\begin{defin}
    We introduce the Sobolev space of quasi-periodic functions
    \begin{equation*}
        \Tilde{H}^s_\alpha(\mathcal{O}) := \overline{\mathcal{Q}_\alpha(\mathcal{O})}^{\| \cdot \|_{H^s_\alpha(\mathcal{G})}}.
    \end{equation*}
\end{defin} \noindent
For $s=1$ this is equivalent to the space of functions which are zero at the boundary of $\mathcal{O}$.\\

Finally, we define the quasi-periodic Sobolev spaces on one-dimensional boundaries, to be used when introducing the Dirichlet to Neumann operator.
\begin{defin}
    Let $0\leq s < \infty$. We define $H^s[0,2\pi]$ as
    \begin{equation*}
        H^s[0,2\pi] := \left\{ \phi \in L^2(0,2\pi) \hspace{0.15cm} \biggl| \hspace{0.15cm} \sum_{n\in \mathbb{Z}} (1+n^2)^s |\phi_n|^2 < \infty \right\},
    \end{equation*}
    where $\{ \phi_n \}_{n\in\mathbb{Z}}$ are the Fourier coefficients for $\phi \in L^2(0,2\pi)$.
\end{defin}
\begin{thm} \label{sobolev_line}
    For $0\leq s < \infty$, $H^s[0,2\pi]$ is a Hilbert space with inner product and induced norm given by
    \begin{equation*}
        (u,v)_{H^s[0,2\pi]} := \sum_{n\in \mathbb{Z}} (1+n^2)^s u_n \overline{v_n}, \hspace{0.8cm} \| u \|_{H^s[0,2\pi]} := (u,u)^{1/2}_{H^s[0,2\pi]}. 
    \end{equation*}
\end{thm}
\begin{defin}
    Let $0\leq s < \infty$. We define $H^s_\alpha[0,2\pi]$ as
    \begin{equation*}
        H^s_\alpha[0,2\pi] := \left\{ \phi \in L^2(0,2\pi) \hspace{0.15cm} | \hspace{0.15cm} e^{-i\alpha x} \phi(x) \in H^s[0,2\pi] \right\}.
    \end{equation*}
\end{defin}
\begin{thm}
    For $0\leq s < \infty$, $H^s_\alpha[0,2\pi]$ is a Hilbert space with inner product and induced norm given by
    \begin{equation*}
        (u,v)_{H^s_\alpha[0,2\pi]} := \sum_{n\in \mathbb{Z}} (1+\alpha_n^2)^s u_{n,\alpha} \overline{v_{n,\alpha}}, \hspace{0.8cm} \| u \|_{H^s_\alpha[0,2\pi]} := (u,u)^{1/2}_{H^s_\alpha[0,2\pi]} ,
    \end{equation*}
    where
    \begin{equation*}
        u_{n,\alpha} = \frac{1}{2\pi} \int_0^{2\pi} e^{-inx}e^{-i\alpha x} u(x) \hspace{0.05cm} dx.
    \end{equation*}
\end{thm}
\begin{defin}
     For $0\leq s < \infty$, we define $H^{-s}[0,2\pi]$ as the dual space of $H^s[0,2\pi]$ and $H^{-s}_\alpha[0,2\pi]$ as the dual space of $H^s_\alpha[0,2\pi]$.
\end{defin} \noindent
We now assume $\Gamma \subset \mathcal{G}$ to be a single period of a periodic curve parametrized by a Lipschitz continuous function $z$ so that
\begin{equation*}
    \Gamma := \{ z(t), t \in (0,2\pi) \},
\end{equation*}
where $z$ may be continuously extended to all $\mathbb{R}$, with
\begin{equation*}
    z(t)=(z_1(t),z_2(t)), \hspace{0.7cm} z_1(t+2\pi n) = z_1(t) + 2\pi n, \hspace{0.7cm} z_2(t+2\pi n) = z_2(t),
\end{equation*}
for all $n \in \mathbb{Z}$ and $t \in [0,2\pi)$.\\
We introduce quasi-periodic Sobolev spaces on arbitrary curves.
\begin{defin}
    Let $0 \leq s < r$, with $r \in \mathbb{N}$ such that $\Gamma$ is a periodic curve of class $C^{r-1,1}$. We define the quasi-periodic Sobolev space of order $s$ over $\Gamma$ as
    \begin{equation*}
        H^s_\alpha(\Gamma) := \{ u \in L^2(\Gamma) \hspace{0.15cm} | \hspace{0.15cm} (u \circ z) (t) \in H^s_\alpha[0,2\pi]  \},
    \end{equation*}
    and equip it with the inner product
    \begin{equation*}
        (u,v)_{H^s_\alpha(\Gamma)} := (u\circ z, v\circ z)_{H^s_\alpha[0,2\pi] },
    \end{equation*}
    where $z : (0, 2\pi) \to \Gamma$ is a parametrization of $\Gamma$.
\end{defin}
\begin{prop}
    Let $0 \leq s < \infty$. The space $H^s_\alpha(\Gamma)$, along with the inner product $(\cdot,\cdot)_{H^s_\alpha(\Gamma)}$ is a Hilbert space. Moreover, these spaces are independent of the chosen parametrization of $\Gamma$.
\end{prop}
\begin{defin}
    For $0 \leq s < \infty$, we define $H^{-s}_\alpha(\Gamma)$ as the completion of $H^0_\alpha(\Gamma) \equiv L^2_\alpha(\Gamma)$ with respect to the norm:
    \begin{equation*}
        \| u \|_{H^{-s}_\alpha(\Gamma)} := \biggl\| (u \circ z) \hspace{0.1cm} \| \stackrel{.}{z} \|_{\mathbb{R}^2} \hspace{0.1cm} \biggl\|_{H^{-s}[0,2\pi]},
    \end{equation*}
    with $\stackrel{.}{z} (t) := (\stackrel{.}{z}_1 (t), \stackrel{.}{z}_2 (t))$ and $z$ as before.
\end{defin}

\section{Model problem} \label{sec2} 
We now present the problem of scattering of a plane wave by a periodic surface. The diffraction grating problem is referred to as wave scattering by periodic structures. The grating problem has been studied in depth with different approaches in many articles and books such as \cite{Bonnet}, \cite{CIVILETTI2020112478} or \cite{Kirsch}.\\

First, following \cite{CIVILETTI2020112478}, we will derive the strong formulation and provide an existence result, then we will focus on the weak formulation and describe how to compute a numerical solution of the problem with a Finite Element Method.\\

We consider linear optics with an $ e^{-i \omega t}$ dependence on time $t$, where $i = \sqrt{-1}$ and $\omega$ is the angular frequency of light. This assumption agrees with the notation in \cite{CIVILETTI2020112478}, \cite{HUTTUNEN200227} and \cite{Melenk}, while in \cite{Survey}, \cite{PWDG} and \cite{KAPITA2018208} the time-dependence is assumed to be $ e^{i \omega t}$. The main difference between these two notations is that with the first one a plane wave with the expression $e^{ik \mathbf{x} \cdot \mathbf{d}}$ effectively propagates in the direction corresponding to the vector $\mathbf{d}$, while with the second notation the wave propagates in the opposite direction $-\mathbf{d}$. This difference in the notation leads to opposite signs in the Sommerfeld radiation condition and also an opposite impedance boundary condition. We choose to follow the convention of Civiletti \cite{CIVILETTI2020112478} since this is more coherent with the plane wave propagation direction and since the Sommerfeld radiation condition in this case is more standard.\\

Under the assumption of $ e^{-i \omega t}$ dependence on time $t$, from Maxwell’s equations one can show that the electric field phasor $\mathbf{E}$ solves
\begin{equation} \label{max}
    \Delta \mathbf{E} = -\omega^2 \mu_0 \varepsilon_0 \varepsilon \mathbf{E} - \nabla \left( \mathbf{E}\cdot\frac{\nabla \varepsilon}{\varepsilon} \right),
\end{equation}
where $\varepsilon=\varepsilon(x_1,x_2)$ is the spatially dependent relative permittivity, and $\varepsilon_0$ and $\mu_0$ are respectively the permittivity and permeability of vacuum. We also assume that $\mu = \mu_0$ everywhere. The domain under consideration is assumed to be invariant in the $\mathbf{e}_3 = (0, 0, 1)$ direction, so the electric field phasor is invariant in the $\mathbf{e}_3$ direction, i.e. $\mathbf{E}= \mathbf{E}(x_1,x_2)$.\\
For $s$-polarized light, we also have that $\mathbf{E} = (0, 0, E_3)$, and so the last term on the right hand side of (\ref{max}) is zero. The wavenumber in air is denoted by $k = \omega/c_0$ and the speed of light in air is $c_0 = 1/\sqrt{\varepsilon_0 \mu_0}$. We obtain the vector Helmholtz equation
\begin{equation}
    \Delta \mathbf{E} +k^2\varepsilon \mathbf{E}=0,
\end{equation}
where $E_1 = E_2 = 0$. So we see that this reduces to a scalar Helmholtz equation. A similar result holds for the $p$-polarization case for the magnetic field phasor $\mathbf{H}$, but that case is more difficult because a $1/\varepsilon$ term appears in the principal part of the differential operator.\\
\begin{figure}[]
    \centering
    \includegraphics[width=.6\textwidth]{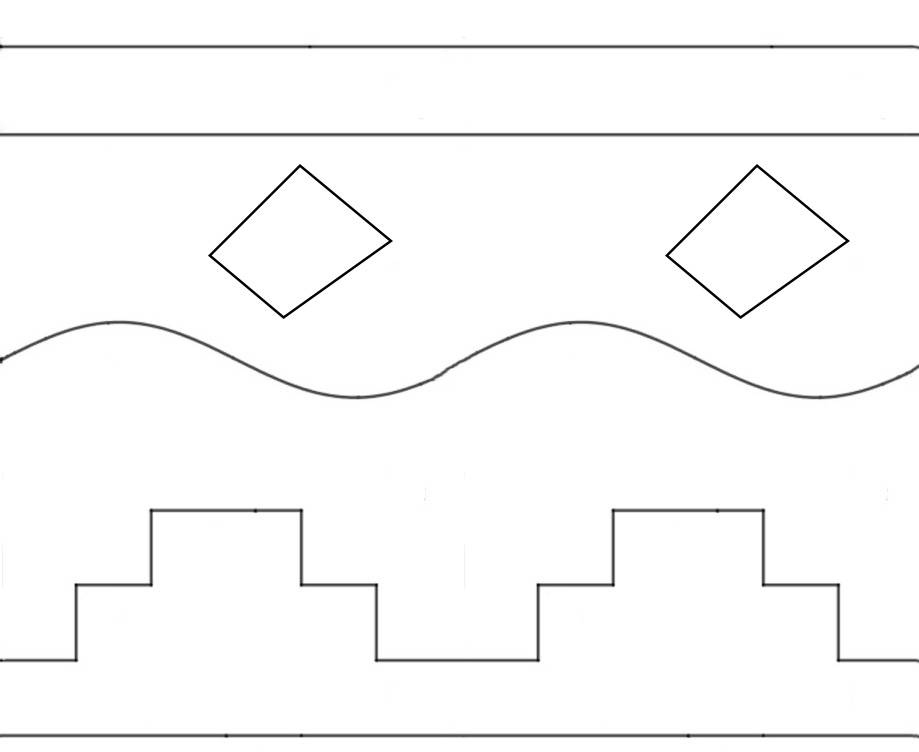}
    \caption{Geometry of the region $\Omega_0$}
    \label{fig:scatt}
\end{figure}\\

We now study the problem of scattering of a plane wave by a periodic structure: we consider a domain $\Omega_0 \subset \mathbb{R}^2$, which is periodic in the $x_1$ direction with period $L$: $\Omega_0$ is divided into subdomains by different interfaces, the profiles of which are picewise linear of Lipschitz functions with the same period $L$. These interfaces can cross all the domain or can also be scattering objects located inside a domain period. The domain division corresponds to its composition of different materials, which can be penetrable or impenetrable by waves. In our model we assume that the lower interface is a straight line impenetrable by waves, which gives us a Dirichlet boundary condition. A possible scattering region is represented in Figure \ref{fig:scatt}. Here we can see many interfaces $\Gamma_j$ inside the region, that are periodic in the $x_1$ direction with period $L$. The relative permittivity $\varepsilon$ is assumed to be $L$-periodic in $x_1$ and invariant in $x_3$.
\begin{figure}
    \centering
    \includegraphics[width=.5\textwidth]{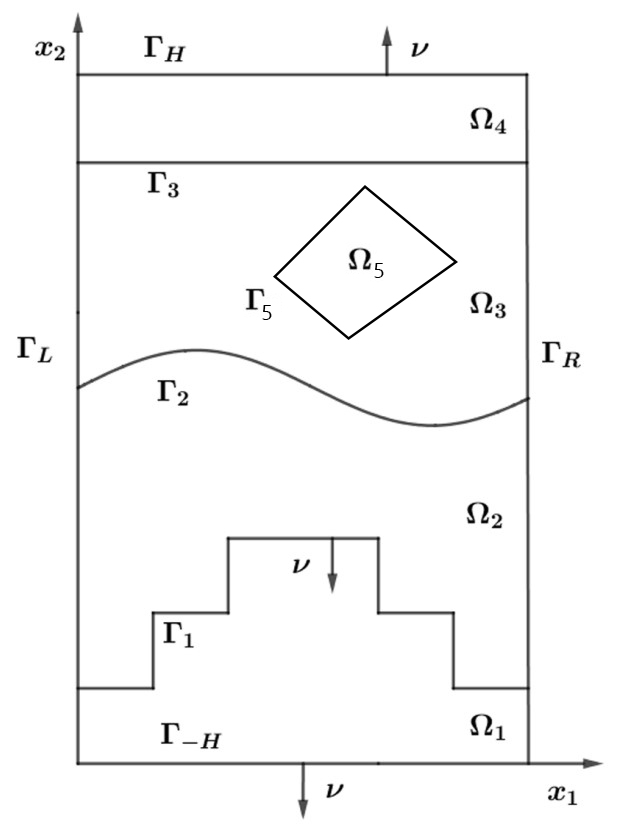}
    \caption{The truncated domain $\Omega$}
    \label{fig:region}
\end{figure}\\

The mathematical formulation is a Helmholtz equation with a periodically variable relative permittivity $\varepsilon$.  By a simple change of variables we can assume the period to be $2\pi$. Due to the periodicity of the scattering region and the relative permittivity, we can reduce our study to a cell of the scatterer, which is the domain $\mathcal{G}:= \{ \mathbf{x} \in \mathbb{R}^2 : 0 <x_1 <2\pi \}$ defined before. Then it will be necessary to truncate the cell on the $x_2$ axis, getting a bounded 2D domain $\Omega = \{x \in \mathbb{R}^2 : 0 < x_1 < 2\pi, -H < x_2 < H \},$ where $H > 0$ is an arbitrary positive constant. We will refer to the two straight lines that delimit the domain $\Omega$ form above and below as $\Gamma_H$ and $\Gamma_{-H}$. Formally we have that
\begin{equation}
    \Gamma_{\pm H} := \{ \mathbf{x \in \mathbb{R}^2} \hspace{0.3cm} : \hspace{0.3cm} x_1 \in [0, 2\pi], \hspace{0.5cm} x_2 = \pm H \}.
\end{equation}
In Figure \ref{fig:region} we can see a possible scattering region, obtained by truncating the region in Figure \ref{fig:scatt}.\\

We will consider the scattering problem when an $s$-polarized plane wave with electric field phasor polarized in the $\mathbf{e}_3$ direction is incident on $\Omega$ with incidence angle $\theta \in \left(-\pi,0 \right)$. The third component of the incident electric field phasor can be stated as
\begin{equation*}
    u^i(\mathbf{x})=u^i(x_1,x_2)=\exp\{ik(x_1 \cos \theta + x_2 \sin \theta)\}.
\end{equation*}

Our aim is to compute the total wave on a two-dimensional domain $\Omega$. The wave solves the Helmholtz equation in $\Omega$ with quasi-periodic boundary conditions
\begin{equation} \label{bound}
    \begin{cases} \Delta u + k^2\varepsilon u=0, & \hspace{0.5cm} \text{in} \hspace{0.1cm} \Omega, \\
    u(2\pi,x_2) = e^{-i\alpha_0 2\pi}u(0,x_2) & \hspace{0.5cm} \forall \hspace{0.05cm} x_2, \\
    \frac{\partial}{\partial x_2}u(2\pi,x_2) = e^{-i\alpha_0 2\pi}\frac{\partial}{\partial x_2}u(0,x_2) & \hspace{0.5cm} \forall \hspace{0.05cm} x_2 \\
    u = 0, & \hspace{0.5cm} \text{on} \hspace{0.1cm} \Gamma_{-H}, \end{cases}
\end{equation}
where $\alpha_0 = k \cos \theta$. We need to insert the periodicity condition to achieve the uniqueness of the solution. With this conditions, we can solve the problem in the cell and then extend the solution with quasi-periodicity.

We assume that $\Omega$ is divided in different subregions; the lines that divide $\Omega$ could either be piecewise linear or smooth functions and could either go from the left to the right boundary or $\Omega$ or close on themselves. We call this separating interfaces $\Gamma_j$. In our experiments we will only consider the piecewise linear case since we do not want to deal with integrals on curved edges. If we have a total number of $I$ lines dividing the domain, $\Omega$ is separated into $I+1$ subdomains, that we call $\Omega_j$, for $j=1,...,I+1$. 

Furthermore, we have the following assumptions on $\varepsilon$. First, $\varepsilon \in C^{(1,1)}(\Omega_k)$ for all $k = 1, 2, ... , I+1$. Also, $\varepsilon$ is allowed to be complex valued in $\Omega$, and either $\{\Re(\varepsilon) > 0, \Im(\varepsilon) \geq 0 \}$ or $\{\Re(\varepsilon) \leq 0, \Im(\varepsilon) > 0 \}$ in $\Omega$. A standard assumption from the literature is that $\varepsilon$ is constant in every $\Omega_k$ subdomain of $\Omega$, but in the general case we will consider a relative permittivity which changes smoothly in the domain. Typically, we take the relative permittivity in the upper half space $\Omega_H^+ := \{ \mathbf{x} \in \mathbb{R}^2 : x_2 > H \}$ to be $\varepsilon^+ = 1$. Thus, the half space above $\Omega$ is air.

On each interface $\Gamma_j$, we choose the unit normal to point downwards or outwards. By $\llbracket \phi \rrbracket_{\Gamma_j}$ we denote the jump of a function $\phi$ across the interface $\Gamma_j$. Thus,
\begin{equation*}
    \llbracket \phi \rrbracket_{\Gamma_j} := \phi |_{\Gamma_j}^+ - \phi |_{\Gamma_j}^-
\end{equation*}
where $\phi |_{\Gamma_k}^+$ is the limit taken from outside $\Omega_{j}$ and $\phi |_{\Gamma_j}^-$ is the limit taken from inside $\Omega_j$, for $1 \leq j \leq I$.

\section{Variational formulation}
We now derive the variational formulation of the scattering problem. First, we need to specify a radiation condition on $\Gamma_H$, as we want that the scattered field "goes upward" on $\Gamma_H$. To specify this condition we will make use of the Dirichlet-to-Neumann operator.

We assume that on the upper domain $u$ can be represented as a linear combination of upward propagating waves and evanescent waves. Since $u$ is $\alpha_0$-quasi-periodic in $\Omega$, we can write
\begin{equation}
    u(\mathbf{x})=\sum_{n\in\mathbb{Z}} u_n(x_2) e^{i\alpha_n x_1},
\end{equation}
for $\mathbf{x}\in \Omega$, where $\alpha_n = \alpha_0 +n$ and the Fourier coefficients $u_n(x_2)$ defined as in (\ref{fourier_coeff}). We define, for $a\geq H$, $\Gamma_a := \{ \mathbf{x} \in \mathbb{R}^2 : x_2=a \}$; we also define $\Omega_H^+ := \{ \mathbf{x} \in \mathbb{R}^2 : 0 \leq x_1 \leq 2\pi, \hspace{0.1cm} x_2 > H \}$. 

We can write the total field $u$ as the sum of incident and scattered field, so we have
\begin{equation*}
    u = u^i+u^s.
\end{equation*}
We know that $u^i$ solves the Helmholtz equation in $\Omega_H^+$ and, since $u^s$ in $C^\infty$ for $x_2>H$ and it is quasi-periodic, we can expand it in terms of Fourier coefficients $u^s(\mathbf{x})=\sum_{n\in\mathbb{Z}} u^s_n(x_2) e^{i\alpha_n x_1}$, with
\begin{equation}
    u^s_n(x_2) = \frac{1}{2\pi} \int_0^{2 \pi} e^{-i\alpha_n x_1} u^s(x_1,x_2) \hspace{0.05cm} dx_1.
\end{equation}
We have that $u^s$ satisfies $\Delta u^s + k^2 \varepsilon^+ u^s = 0$ for $x_2 > H$, where $\varepsilon^+$ is the relative permittivity in the region above the domain $\Omega$, which we usually assume to be equal to $1$; the Fourier coefficients of $u^s$ must solve the following differential equation:
\begin{equation}
    -\frac{\partial^2 u^s_n}{\partial x_2^2} - (k^2 \varepsilon^+ - \alpha^2_n)u^s_n = 0, \hspace{0.8cm} \text{if} \hspace{0.2cm} x_2 > H.
\end{equation}
Since each one of these equations has two independent solutions, we are forced to choose between them as follows:
\begin{itemize}
    \item[(i)] if $(k^2 \varepsilon^+ - \alpha^2_n)<0$, we select the decaying solution
    \begin{equation*}
        u^s_n(x_2) = u_n^s(H) e^{-\sqrt{\alpha^2_n-k^2 \varepsilon^+}(x_2-H)},
    \end{equation*}
    \item[(ii)] if $(k^2 \varepsilon^+ - \alpha^2_n)=0$, we choose the constant
    \begin{equation*}
        u^s_n(x_2) = u_n^s(H),
    \end{equation*}
    \item[(iii)] if $(k^2 \varepsilon^+ - \alpha^2_n)>0$, we opt for the solution corresponding to an outgoing wave
    \begin{equation*}
        u^s_n(x_2) = u_n^s(H) e^{i\sqrt{k^2 \varepsilon^+-\alpha^2_n}(x_2-H)}.
    \end{equation*}
\end{itemize}
Then the scattered field has the following form
\begin{equation} \label{representation}
    u^s(\mathbf{x})=\sum_{n\in\mathbb{Z}} u^s_n(H) e^{i\beta_n (x_2-H)} e^{i\alpha_n x_1},
\end{equation}
where we assume that $\alpha_n^2 \neq k^2 \varepsilon^+ $ for every $n$ and 
\begin{equation} \label{beta}
    \beta_n = \begin{cases} \sqrt{k^2\varepsilon^+-\alpha_n^2} & \alpha_n^2 < k^2\varepsilon^+, \\ i\sqrt{\alpha_n^2-k^2\varepsilon^+} & \alpha_n^2 > k^2\varepsilon^+. \end{cases} 
\end{equation}

We are now able to introduce an appropriate radiation condition to be satisfied by the scattered field 
\begin{defin}[Radiation Condition]
    We say that $u\in H^1_{\alpha_0}(\Omega)$ satisfies the radiation conditions at infinity if there exists $h > 0$ such that, for all $x_2 \geq h$, there holds
    \begin{equation} 
        u^s(\mathbf{x}) = \sum_{n\in\mathbb{Z}} u_n e^{i\beta_n (x_2-h)} e^{i\alpha_n x_1},
    \end{equation}
    with $u_n \in \mathbb{C}$ for all $n\in \mathbb{Z}$ and $\beta_n$ as in (\ref{beta}).
\end{defin}

We now want to impose this radiation condition on $\Gamma_H$; to this purpose, we introduce the Dirichlet-to-Neumann operator.\\ Taking the normal derivative of $u^s$ on $\Gamma_H$, we have that
\begin{equation*}
    \frac{\partial u^s}{\partial x_2}\biggl|_{\Gamma_H} = i\sum_{n\in\mathbb{Z}} u^s_n(H) \beta_n e^{i\alpha_nx_1}.
\end{equation*}
We want our numerical solution to reproduce this behaviour, so we define the Dirichlet-to-Neumann (DtN) operator
\begin{align} \label{DtN}
    & T:H^{1/2}_{\alpha_0}(\Gamma_H) \to H^{-1/2}_{\alpha_0}(\Gamma_H), \\ \nonumber
    & (T\phi)(x_1) = i \sum_{n\in\mathbb{Z}} \phi_n \beta_n e^{i\alpha_nx_1}, \hspace{0.5cm} \text{for} \hspace{0.2cm} \phi \in H^{1/2}(\Gamma_H).
\end{align}
This operator maps the Dirichlet trace of $\phi \in H^{1/2}_{\alpha_0}(\Gamma_H)$ into its Neumann trace.

The following property holds for the DtN operator:
\begin{lem}
    The operator $T:H^{1/2}_{\alpha_0}(\Gamma_H) \to H^{-1/2}_{\alpha_0}(\Gamma_H)$ is continuous.
\end{lem}
\begin{proof}
    By straightforward computation (recall the definition of the $H^{s}_{\alpha_0}(\Gamma_H)$ in Theorem \ref{sobolev_line}), it holds
    \begin{equation*}
        \| T\phi \|^2_{H^{-1/2}_{\alpha_0}(\Gamma_H)} = \sum_{n \in \mathbb{Z}} (1+n^2)^{-1/2} |\beta_n|^2 |\phi_n|^2 \leq \sup_{n \in \mathbb{Z}} \frac{|\beta_n|^2}{1+n^2} \| \phi \|^2_{H^{1/2}_{\alpha_0}(\Gamma_H)}.
    \end{equation*}
    By definition of $\beta_n$, 
    \begin{equation*}
        \frac{|\beta_n|^2}{1+n^2} = \frac{k^2\varepsilon^+ - \alpha_n^2}{1+n^2} = \frac{k^2\varepsilon^+ - n^2 -2n\alpha_0 - \alpha_0^2}{1+n^2},
    \end{equation*}
    which is bounded in $n$.
\end{proof}

Using the Dirichlet to Neumann operator $T$, the problem can be restricted to the domain $\Omega$, i.e., given an incident plane wave $u^i$, seek an $\alpha_0$-quasi-periodic solution $u$ such that
\begin{equation} \label{bvp}
    \begin{cases}
        \Delta u + k^2 \varepsilon u = 0 & \hspace{0.5cm} \text{in} \hspace{0.1cm} \Omega, \\
        u = 0 & \hspace{0.5cm} \text{on} \hspace{0.1cm} \Gamma_{-H}, \\
        \frac{\partial(u-u^i)}{\partial \nu} - T(u-u^i) = 0 & \hspace{0.5cm} \text{on} \hspace{0.1cm} \Gamma_{H},
    \end{cases}
\end{equation}
where $\nu$ is the outward pointing normal on $\Gamma_H$.\\

We are now ready to derive the variational formulation of (\ref{bvp}). We consider the space
\begin{equation}
    V=\{ v \in H^1_{\alpha_0}(\Omega) : v=0 \hspace{0.2cm}\text{on} \hspace{0.2cm} \Gamma_{-H}\},
\end{equation}
where $H^1_{\alpha_0}(\Omega)$ is the $\alpha_0$-quasi-periodic Sobolev space defined in Definition \ref{sobolev}. Let $u \in V$ be a distributional solution of the Helmholtz problem (\ref{bvp}). Multiplying both sides of the Helmholtz equation by a test function $v\in V$ and integrating by parts, we get
\begin{equation*}
    \int_\Omega \left( \nabla u \cdot \nabla \overline{v} - k^2\varepsilon u\overline{v} \right) - \int_{\partial \Omega} \overline{v} \nabla u \cdot \nu = 0.
\end{equation*}

The integral over the lower boundary $\Gamma_{-H}$ is equal to zero since we have Dirichlet boundary conditions, while the integral over the left and right boundaries cancel for quasi-periodicity. This yields
\begin{equation}
    \int_\Omega \left( \nabla u \cdot \nabla \overline{v} - k^2\varepsilon u\overline{v} \right) d\mathbf{x} - \int_{\Gamma_H} \overline{v} \nabla u \cdot \nu \hspace{0.1cm} ds = 0.
\end{equation}
Now we use the Dirichlet-to-Neumann operator to replace the normal derivative of the scattered field. Replacing $u$ by $u^i+u^s$, using $Tu^s$ for $ \nabla u^s \cdot \nu$ and replacing $u^s$ again by $u - u^i$ yields
\begin{equation}
    \int_\Omega \left( \nabla u \cdot \nabla \overline{v} - k^2\varepsilon u\overline{v} \right) d\mathbf{x} - \int_{\Gamma_H} Tu \hspace{0.1cm} \overline{v} \hspace{0.1cm} ds = \int_{\Gamma_H} \left( \nabla u^i \cdot \nu -Tu^i \right) \overline{v} \hspace{0.1cm} ds.
\end{equation}
The left hand side of the equation above defines the sesquilinear form
\begin{equation} \label{bilinear}
     b_\varepsilon(u,v) := \int_\Omega \left( \nabla u \cdot \nabla \overline{v} - k^2\varepsilon u\overline{v} \right) d\mathbf{x} - \int_{\Gamma_H}  Tu \hspace{0.1cm} \overline{v} \hspace{0.1cm} ds.
\end{equation}
For the right-hand side of the variational problem, we observe that on $\Gamma_H$
\begin{equation*}
    \frac{\partial u^i}{\partial \nu} - T u^i = 2i \beta_0 u^i,
\end{equation*}
where $\beta_0 = k \sin \theta$ and $\nu$ is the outer normal to $\Omega$. It follows that the weak formulation of the problem is: find $u \in V$ such that
\begin{equation} \label{weak}
    b_\varepsilon(u,v) = \int_{\Gamma_H} 2i\beta_0 u^i \hspace{0.03cm} \overline{v} \hspace{0.1cm} ds \hspace{0.5cm} \forall \hspace{0.03cm} v \in V.
\end{equation}

We already proved that the DtN operator $T$ is a bounded mapping from $H^{1/2}_{\alpha_0}(\Gamma_H)$ into $H^{-1/2}_{\alpha_0}(\Gamma_H)$. Also, $T$ is bijective if and only if $\beta_n \neq 0$ for all $n \in \mathbb{Z}$. From now on we interpret the integral $\int_{\Gamma_H}  Tu \hspace{0.1cm} \overline{v} \hspace{0.1cm} ds$ for $u, v \in V$ as the dual bracket $\langle Tu,v \rangle_{1/2}$ (since $Tu \in H^{-1/2}_{\alpha_0}(\Gamma_H)$ and $v_{| \Gamma_H} \in H^{1/2}_{\alpha_0}(\Gamma_H)$ by the trace theorem).

We can split the bilinear form a from (\ref{bilinear}) into two parts $a = a_0 + a_K$ with
\begin{align*}
    & a_0(u,v) := \int_\Omega \left( \nabla u \cdot \nabla \overline{v} + u\overline{v} \right) d\mathbf{x} - \int_{\Gamma_H} T_0u  \hspace{0.1cm}  \overline{v} \hspace{0.1cm} ds, \\
    & a_K(u,v) := -(k^2\varepsilon+1)\int_\Omega  u\overline{v} \hspace{0.1cm} d\mathbf{x} + \int_{\Gamma_H} (T_0-T)u \hspace{0.1cm} \overline{v} \hspace{0.1cm} ds,
\end{align*}
where
\begin{equation*}
    T_0 \phi := - \sum_{n\in\mathbb{Z}} |\alpha_0+n| \psi_n e^{i\alpha_n \cdot}, \hspace{0.5cm} \text{for} \hspace{0.2cm} \psi = \sum_{n\in\mathbb{Z}} \psi_n e^{i\alpha_n \cdot}.
\end{equation*} 
Then $-T_0$ is nonnegative and, therefore, $a_0$ is coercive in $V$. By the theorem of Lax-Milgram there exists an isomorphism $L$ from $V$ onto itself with $a_0(v, w) = (Lv, w)_{H^1(\Omega)}$ for all $v,w \in V$. By the Riesz theorem there exists $r\in V$ and a bounded operator $K$ from $V$ into itself with $a_K(v, w) = (Kv, w)_{H^1(\Omega)}$ for all $v,w \in V$ and $\int_{\Gamma_H} 2i\beta_0 u^i \hspace{0.03cm} \overline{v} \hspace{0.1cm} ds = (r,v)_{H^1(\Omega)}$ for $v \in V$. Furthermore, the operator $T- T_0$ is compact since $\beta_n = i \hspace{0.05cm} |n+\alpha_0| + \mathcal{O}(\frac{1}{|n|})$ for $|n|\to \infty$. This implies that also $K$ is compact (see \cite{Kirsch}).

The variational equation (\ref{weak}) can then be written in the form $Lu + Ku = 0$ with isomorphism $L$ and compact $K$. It is easily seen that any solution $u\in V$ of $Lu + Ku = 0$ is a classical solution of the scattering problem if it is extended by the solution of the Dirichlet problem with boundary data $u$ into the region $(0, 2\pi) \times (H, \infty)$. It follows that, under appropriate hypothesis on $\varepsilon$, $L+K$ is one-to-one. Fredholm's alternative then again yields the existence of a unique solution of the variational equation (\ref{weak}) in $V$. This construction can be made for other boundary conditions as well.

\section{Well-posedness and Stability} \label{regularity}
There are various formulations of the periodic grating problem. For example, one can solve this problem for the inhomogeneous Helmholtz equation $\Delta u + k^2 \varepsilon u = f$, with $f \in L^2(\Omega)$, as it is done in \cite{CIVILETTI2020112478}. This difference is given by the choice to approximate the scattered field $u^s$ instead of the total one. In our problem, we choose to solve the homogeneous Helmholtz problem since the PWDG method is suitable only for this type of equation.  Another difference in the scattering problem can be found in the boundary conditions: a common choice is to consider a so-called transparent domain, which is simply a truncated vertical stripe with a Dirichlet to Neumann boundary condition on both the upper and the lower truncation of the domain; obviously, the definition of the lower DtN operator will change, as we want the total wave to point downward. Under $\Gamma_{-H}$ the domain is assumed to have constant relative permittivity $\varepsilon^-$, which can be different from $\varepsilon^+$.

Another substantial difference in the formulations is in the assumptions on the relative permittivity $\varepsilon$. In some cases this is considered to be a function that changes in a continuous way, while in other works the value of $\varepsilon$ is constant in every subdomain of $\Omega$, as we assume in our numerical tests.

The boundary value problem in \cite{CIVILETTI2020112478} is the following: given an incident plane wave $u^i$, seek an $\alpha_0$-quasi-periodic solution $u$ such that
\begin{equation} \label{bvp_civil}
    \begin{cases}
        \Delta u + k^2 \varepsilon u = f & \hspace{0.5cm} \text{in} \hspace{0.1cm} \Omega, \\
        \frac{\partial u}{\partial \nu} - T^- u = 0  & \hspace{0.5cm} \text{on} \hspace{0.1cm} \Gamma_{-H}, \\
        \frac{\partial u}{\partial \nu} - T^+ u = 0 & \hspace{0.5cm} \text{on} \hspace{0.1cm} \Gamma_{H},
    \end{cases}
\end{equation}
where $T^+, T^-$ are the two DtN operators defined respectively on $\Gamma_H$ and $\Gamma_{-H}$.

A fundamental property that we want assume on $\varepsilon$ is the so-called \textit{non trapping condition}. This property guarantees that the reflected wave does not remain "trapped" by the scatterer and ensures the uniqueness of the solution.

In \cite{CIVILETTI2020112478}, under the assumption $\Re(\varepsilon)>0$ and $\Im(\varepsilon)=0$, the following theorem is proved. We recall that these results are referred to a different boundary value problem.
\begin{thm} \label{thm:stima}
    Assume that $\varepsilon \in C^{(1,1)}(\Omega_j)$ for all $ j= 1, 2, ... , I + 1$, and the non-trapping conditions
    \begin{equation} 
        \frac{\partial \varepsilon}{\partial x_2} \geq 0, \hspace{0.5cm} \llbracket \varepsilon \rrbracket_{\Gamma_j} > 0, \hspace{0.5cm} \text{and} \hspace{0.3cm} \Re(\varepsilon^+-\varepsilon) \geq 0 \hspace{0.3cm} \text{on} \hspace{0.3cm} \Gamma_H
    \end{equation}
    hold for all $j = 1, 2, ... , I + 1$. Further, assume that $\Re(\varepsilon^\pm) \geq 0$ and $\Im(\varepsilon^\pm) \geq 0$ and $\varepsilon$ is real in $\Omega$. Then, for $f\in L^2(\Omega)$, there exists a unique solution $u \in H^1_\alpha(\Omega)$ of the variational problem (\ref{bvp_civil}). 
\end{thm}

\begin{rem}
    When the non-trapping condition is not verified, we are not guaranteed to find a unique solution of (\ref{bvp_civil}). In Section 5 of \cite{Bonnet} many non uniqueness examples are presented. In our analysis we will assume that this uniqueness property is satisfied.
\end{rem}

\begin{rem}
    Note that in the classical case where $\varepsilon$ is piecewise constant, the first part of the non-trapping conditions is always verified, since $\frac{\partial \varepsilon}{\partial x_2} = 0$.\\
\end{rem}

For our boundary value problem of interest (\ref{bvp}) with Dirichlet boundary conditions on $\Gamma_{-H}$, we will say that $\varepsilon$ verifies the \textit{non-trapping conditions} if the problem is well-posed and it holds a stability estimate for the solution $u$, i.e. if it exists a constant $C > 0$ such that
\begin{equation} \label{notrap}
    \| u \|_{H^1(\Omega)} \leq C \| u^i \|_{H^{1/2}(\Gamma_H)},
\end{equation}
for every possible quasi-periodic incident wave $u^i$.

In the following results, we will assume $\varepsilon$ is non-trapping, although we don't know exactly what properties it has to verify.

\section{Truncated Boundary Value Problem}
In practical computations, one needs to truncate the infinite series of the DtN operator (\ref{DtN}) to obtain an approximate mapping written as a finite sum
\begin{equation} \label{DtN_trunc}
    (T_M \phi)(x_1) = i \sum_{n=-M}^M \phi_n \beta_n e^{i\alpha_nx_1}, \hspace{0.5cm} 
\end{equation}
for all $\phi \in H^{1/2}_{\alpha_0}(\Gamma_H) $. Here $M$ is usually an integer greater than $k^2\varepsilon^+$ if $k^2\varepsilon^+$ is real.

The boundary value problem (\ref{bvp}) is replaced by the following modified boundary value problem with a truncated DtN map: Find $u^M \in H^1_{\alpha_0}(\Omega)$ such that
\begin{equation} \label{bvp_tr}
    \begin{cases}
        \Delta u^M + k^2 \varepsilon u^M = 0 & \hspace{0.5cm} \text{in} \hspace{0.1cm} \Omega, \\
        u^M = 0 & \hspace{0.5cm} \text{on} \hspace{0.1cm} \Gamma_{-H}, \\
        \frac{\partial(u^M-u^i)}{\partial \nu} - T_M(u^M-u^i) = 0 & \hspace{0.5cm} \text{on} \hspace{0.1cm} \Gamma_{H}.
    \end{cases}
\end{equation}

\begin{rem}
    It is not guaranteed that the good properties of (\ref{bvp}), such as the uniqueness of the solution, are maintained by the truncated problem. In our analysis, we will assume that the number of Fourier modes $M$ that we use to approximate the DtN operator is large enough to guarantee these properties also for (\ref{bvp_tr}).
\end{rem}

\begin{rem}
    Lemma 4.5 of \cite{FEMQP} proves a good approximation property of the truncated DtN operator $T_M$ for the full operator $T$: in particular, it shows  a lower bound on the number of Fourier modes that guarantees that the error derived from the approximation of the DtN operator and the numerical scheme is small enough.
\end{rem}

\section{FEM with DtN boundary conditions}
The finite element method for the Helmoltz equation and the diffraction grating problem with DtN boundary conditions has been studied in various works, such as \cite{Bao} or \cite{FEMQP}.\\
Here we propose a simple FEM-P1 method, which is mainly used to provide a benchmark to test the TDG developed in Section \ref{sec_tdg}.\\

Let $\bigtau_h$ be a regular triangulation on the domain $\Omega$. Every triangle $K \in \bigtau_h$ is considered a closed set. To define a finite element space whose functions are quasi-periodic in the $x_1$ direction, we require that if $(0,x_2)$ is a node on the left boundary, then $(2\pi, x_2)$ also is a node on the right boundary, and vice versa. Let $V_h \subset V$ denote a conforming finite element space, that is,
\begin{equation*}
    V_h := \{ v_h \in C(\overline{\Omega}): v_h |_K \in P_p(K) \hspace{0.2cm} \forall K \in \bigtau_h,\hspace{0.2cm} v_h(0,x_2) = e^{-i\alpha_0 2\pi}v_h(2\pi,x_2), \hspace{0.2cm} \forall x_2 \}
\end{equation*}
where $p$ is a positive integer and $P_p(K)$ is the set of polynomials of degrees $p$ on $K$.\\
The finite element approximation to problem (\ref{weak}) reads as follows: Find $u_h \in V_h$ such that
\begin{equation} \label{discrete}
    b_\varepsilon^M(u_h^M,v_h) =  \int_{\Gamma_H} 2i\beta_0 u^i \hspace{0.03cm} \overline{v_h} \hspace{0.1cm} dx_1, \hspace{0.5cm} \forall \hspace{0.03cm} v_h \in V_h.
\end{equation}
where the sesquilinear form $b_\varepsilon^M(\cdot,\cdot)$ is defined as
\begin{equation}
    b^M_\varepsilon(u,v) := \int_\Omega \left( \nabla u \cdot \nabla v - k^2\varepsilon u\overline{v} \right) dx - \int_{\Gamma_H} T_M u \hspace{0.05cm} \overline{v} \hspace{0.1cm} dx_1,
\end{equation}
where $T_M$ is the truncated Dirichlet to Neumann operator defined in (\ref{DtN_trunc}).

\subsection{Numerical implementation}
We implement the discrete problem in MATLAB: we make use of polynomials of degree one. The basis functions are the usual hat functions with degrees of freedom associated to vertices of the triangulation. The problem $(\ref{discrete})$ can be written as a linear system $A \mathbf{u} = \mathbf{b}$ 
The matrix $A$ can be written as the sum of three parts, $A=A_S+k^2\varepsilon A_M + A_{DtN}$, the first two being respectively the usual stiffness and mass matrix of the finite element method, with the only difference that when we compute the matrix entries for the basis function associated to the right side of $\Omega$, we have to multiply for the quasi-periodicity constant $e^{-i\alpha_0 2\pi}$ (this is done to ensure the quasi-periodicity of the numerical solution). The matrix $ A_{DtN}$ is associated to the basis functions defined on the vertices on $\Gamma_H$ and contains the value of the integrals $\int_{\Gamma_H} T_Mu_h \hspace{0.05cm} \overline{v_h} \hspace{0.05cm} dx_1$. Note that the function $T_M u_h$ is defined on all $\Gamma_H$, while $u_h$ is non-zero only on a part of the boundary, so every basis function associated to the vertices on $\Gamma_H$ communicates with every other basis function on $\Gamma_H$.

For what concerns the right hand side of the system, the only nonzero entries of the vector $\mathbf{b}$ are the ones associated to vertices on $\Gamma_H$.

For the computation of the line integrals $\int_{\Gamma_H}  T_Mu_h \hspace{0.05cm} \overline{v_h} \hspace{0.05cm} dx_1$ and the right-hand side $\int_{\Gamma_H} 2i\beta_0 u^i \hspace{0.03cm} \overline{v_h} \hspace{0.1cm} dx_1$, we make use of a Gauss-Legendre quadrature rule. In our experiments we choose $M=100$, so that we use $2M+1=201$ Fourier modes to approximate the DtN operator.

Here and in the following, to create the domain triangulation, we will use the \textit{distmesh} MATLAB package provided in \cite{distmesh}.

\subsection{Numerical Experiments}
To test our method, we consider a simple model problem. We divide the domain $\Omega=[0,2\pi]\times [-H,H]$ in two different regions with different relative permittivity $\varepsilon_1, \varepsilon_2$ by the straight line $x_2=0$. We consider $\varepsilon_1 = 1$ in the region $\{ x_2 > 0 \}$ and $\varepsilon_2 > 1$ in $\{ x_2 < 0 \}$. A plane wave is incident with an angle $\theta \in (-\pi,0)$ on the straight line.\\
In this simple case, we can derive the exact solution of the scattering problem: we know that the solution $u$ will have two different expressions, depending on the region of $\Omega$.
We will have the following expression:
\begin{equation} \label{sol}
    \begin{cases} u(x_1,x_2) = e^{ik(x_1 \cos \theta + x_2 \sin \theta)} + Re^{ik(x_1 \cos \theta - x_2 \sin \theta)} & x_2>0\\
     u(x_1,x_2) = T_1e^{ik(x_1 \cos \theta - x_2 \sqrt{\varepsilon_2-\cos^2\theta})} + T_2e^{ik(x_1 \cos \theta + x_2 \sqrt{\varepsilon_2-\cos^2\theta})} & x_2<0 \end{cases}
\end{equation}
In the first region $\{ x_2>0 \}$, the exact solution is the sum of the incident wave and the wave reflected by the line $x_2=0$, while in the second region $\{ x_2<0 \}$ the wave angle changes due to the different value of $\varepsilon_2$. To find the coefficients $R,T_1,T_2$ we impose the continuity of the function and its normal derivative and the Dirichlet condition on $\Gamma_{-H}$. We refer as $u_+$ for the expression of the solution $u$ in the upper region $\{ x_2>0 \}$, and as $u_-$ for its expression in $\{ x_2 < 0 \}$.
\begin{equation*} 
    \begin{cases}
        u_{+}(x_1,0) = u_{-}(x_1,0), \\
        \frac{\partial}{\partial x_2}u_{+}(x_1,0) = \frac{\partial}{\partial x_2}u_{-}(x_1,0), \\
        u_{+}(x_1,-H) = 0.
    \end{cases} \Longrightarrow \hspace{0.3cm}
    \begin{cases}
        1 + R = T_1 + T_2 \\
        (1 - R) \sin \theta = (T_2 - T_1) \sqrt{\varepsilon_2 - \cos^2 \theta} \\
        T_1e^{ikH \sqrt{\varepsilon_2-\cos^2\theta}} + T_2e^{-ikH\sqrt{\varepsilon_2-\cos^2\theta}} = 0
    \end{cases}
\end{equation*}
This gives us a linear system that we can solve in MATLAB.
\begin{example}
    We first consider a problem with $\varepsilon_1=1$ and $\varepsilon_2=(1.27+0.1i)^2$. The incident angle is $\theta = -\pi / 3$ and the wavenumber is $k=5$. In figures \ref{fig:QP_P1} and \ref{fig:QP_P1_err} we can see the results of the numerical experiments. We decrease the value of the mesh width $h$ and notice that the $L^2$ error against the exact solution decreases with a rate of approximately $h^2$.
\end{example}
\begin{figure}[]
    \centering
    \includegraphics[width=.85\textwidth]{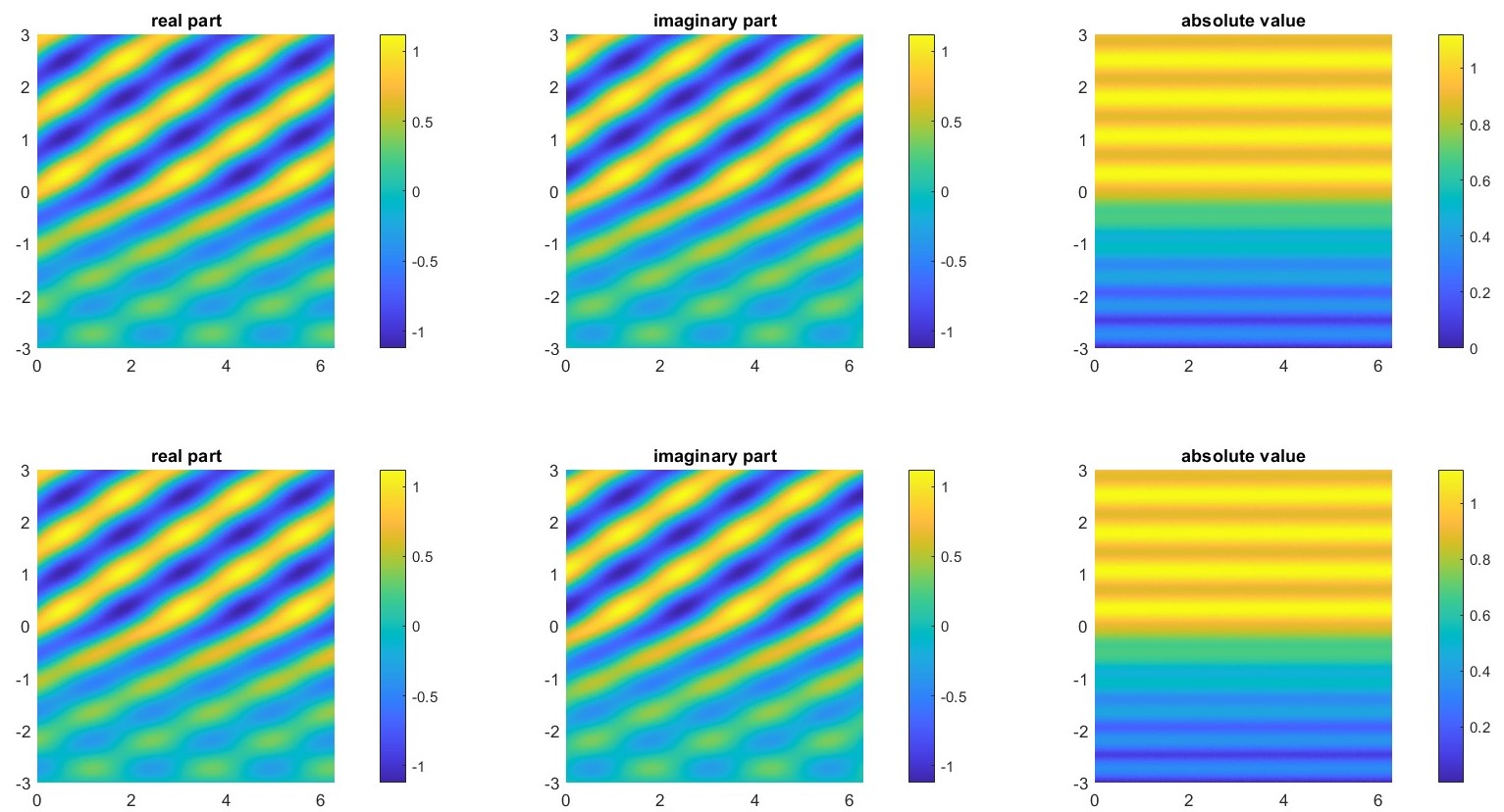}
    \caption{Plots of the numerical solution (top row) against the exact solution (bottom row)}
    \label{fig:QP_P1}
\end{figure}
\begin{figure}[]
    \centering
    \includegraphics[width=.65\textwidth]{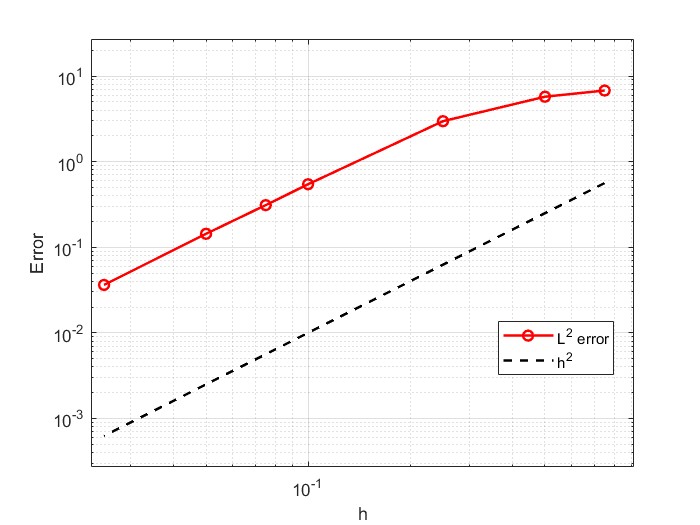}
    \caption{The errors in $L^2(\Omega)$ norm for different values of $h$}
    \label{fig:QP_P1_err}
\end{figure}
\begin{example} \label{examp}
    Another simple problem we can consider is based on the same domain $\Omega$ with different values of relative permittivity $\varepsilon_1=1$ and $\varepsilon_2=(1.27+0.2i)^2$, incident wave angle $\theta = -\pi / 4$ and wavenumber $k=5$. In figures \ref{fig:QP_P1_p4} and \ref{fig:QP_P1_p4_err} we can see the results of the numerical experiments. Here we can see the effect of the stronger absorption value of $\varepsilon_2$, as we have that the absolute value of the field rapidly decreases in the lower half of the domain. In the error plot we observe the same convergence rate as before.
\end{example}
\begin{figure}[]
    \centering
    \includegraphics[width=.85\textwidth]{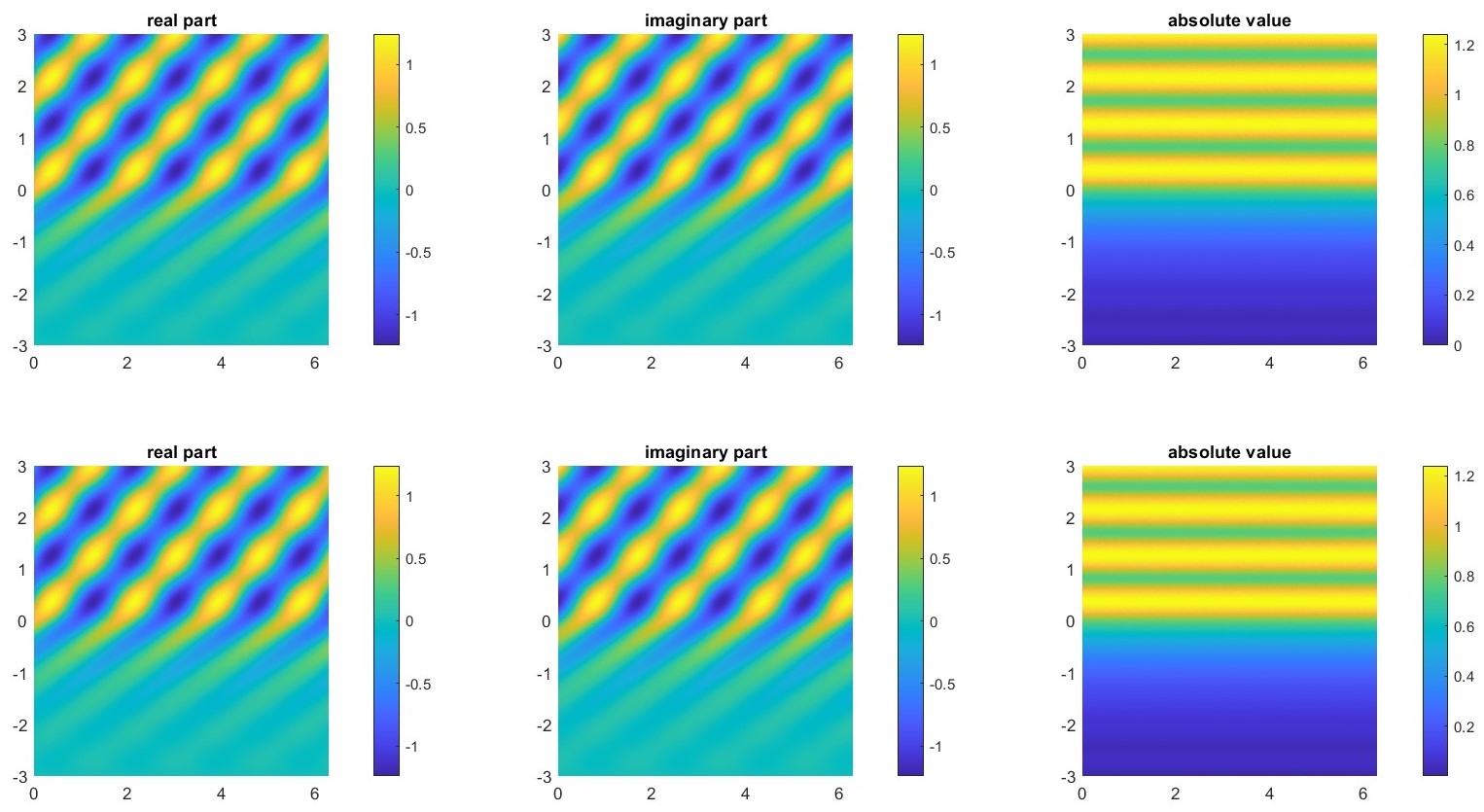}
    \caption{Plots of the numerical solution (top row) against the exact solution (bottom row)}
    \label{fig:QP_P1_p4}
\end{figure}
\begin{figure}[]
\centering
\subfloat[]
{\includegraphics[width=.41\textwidth]{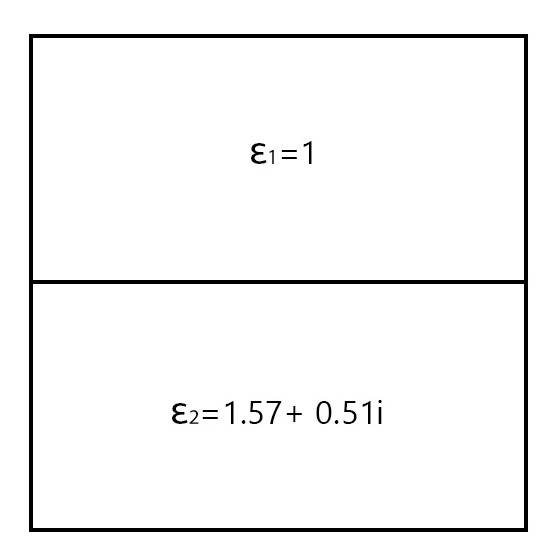}} \quad
\subfloat[]
{\includegraphics[width=.56\textwidth]{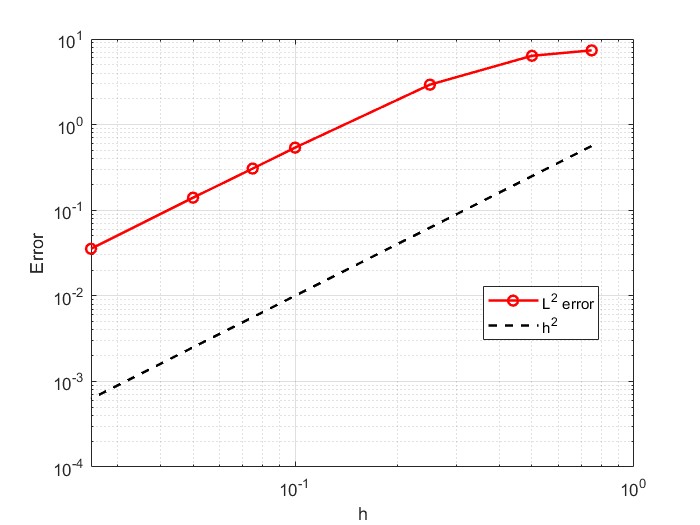}} \quad
\caption{The domain for Example \ref{examp} and the errors in $L^2(\Omega)$ norm for different values of $h$ }
\label{fig:QP_P1_p4_err}
\end{figure}

\begin{example} \label{ex:quad}
    We consider the domain $\Omega= [0,2\pi] \times [-2,2]$ displayed in Figure \ref{fig:quad}, divided by a piecewise constant line in two regions with different relative permittivity $\varepsilon=1$ and $\varepsilon_2 = (1.27+0.25i)^2$. A plane wave is incident with angle $\theta= -\pi / 4$ and wavenumber $k=2$. In this case we can't refer to an exact solution.
    
    In Figure \ref{fig:quad_plot} we display the plot of the numerical solution for decreasing values of $h$. The results are similar to the one obtained in \cite{Bao}.
\end{example}
\begin{figure}[]
    \centering
    \includegraphics[width=.51\textwidth]{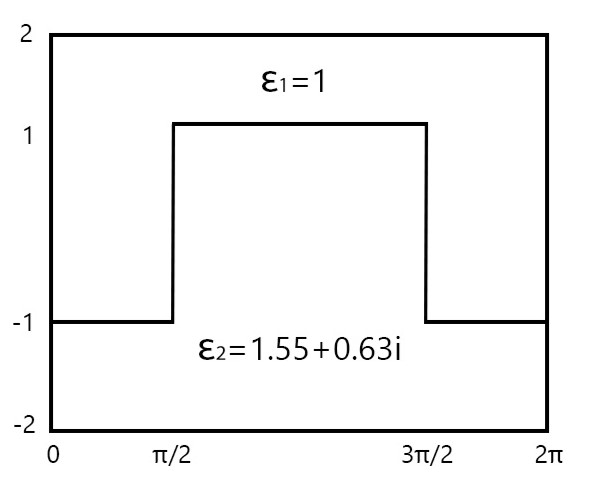}
    \caption{The domain of Example \ref{ex:quad}}
    \label{fig:quad}
\end{figure}
\begin{figure}[]
    \centering
    \includegraphics[width=.99\textwidth]{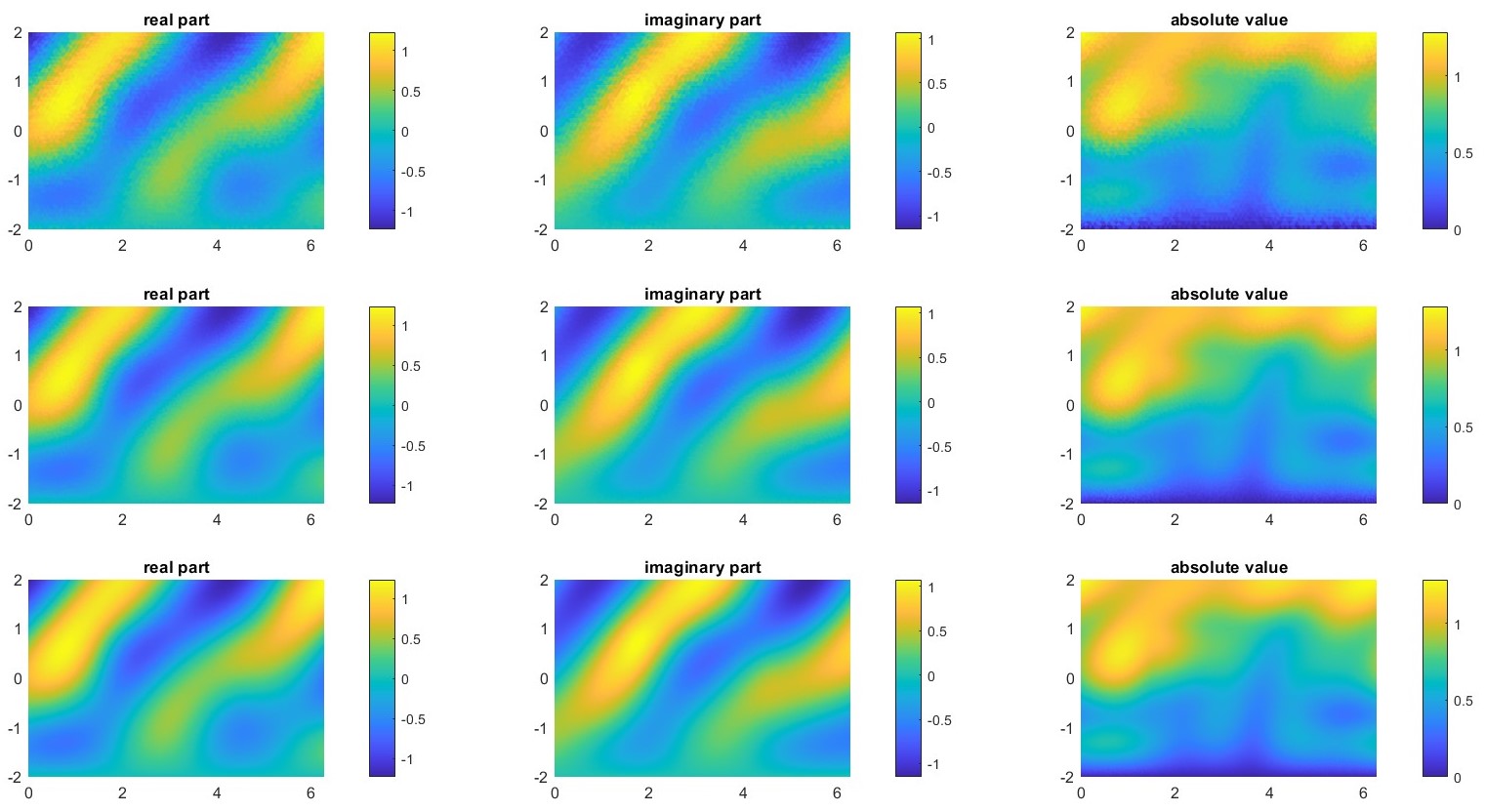}
    \caption{The numerical solution of Example \ref{ex:quad} for $h=0.1$, $h=0.05$ and $h=0.025$ respectively}
    \label{fig:quad_plot}
\end{figure}

\chapter{Trefftz Discontinuous Galerkin Methods}

In this Chapter we describe the Trefftz-DG methods for impedance boundary value problems on bounded domains, following \cite{Survey} and \cite{PWDG}. For the extension to the case with piecewise constant relative permittivity we follow \cite{Hoearth}. We will then extend this method to quasi-periodic problem in unbounded domains in Chapter \ref{ch4}.

\section{Trefftz Methods}
A Trefftz method is a volume-oriented discretisation scheme, for which all trial and test functions are solutions of the PDE, when restricted to any element of a given mesh. The name comes from the work of Erich Trefftz, dating back to 1926, where this idea was applied to the Laplace equation.\\
Trefftz methods are related to both finite element (FEM) and boundary element methods (BEM). With the former they have in common that they provide a discretisation in the volume. With the latter they share some characteristics such as the need of integration on lower-dimensional manifolds only.\\
Trefftz methods possess several distinctive features that set them apart from traditional numerical methods: they provide an analytical representation of the solution in terms of Trefftz functions, allowing for a direct assessment of the global behavior of the field and offer the potential for high accuracy even with a relatively small number of degrees of freedom, leading to significant computational savings.

There are many different variations and extensions of Trefftz methods: the Trefftz-discontinuous Galerkin (TDG), the ultra weak variational formulation (UWVF), the discontinuous enrichment method (DEM), the variational theory of complex rays (VTCR) and the wave based method (WBM). We will only focus on the first method, and in particular we will explore the aspects of Plane Wave Discontinuous Galerkin (PWDG) methods.\\

\section{Trefftz-DG method for impedance BVPs} \label{dg}
Trefftz Discontinuous Galerkin (TDG) methods are a class of Trefftz-type methods for the spatial discretization of boundary value problems. We rely on a simple model boundary value problem (BVP) for the Helmholtz equation and derive the general DG formulation. Let $\Omega \subset \mathbb{R}^n$, $n=2,3$, be a bounded, Lipschitz, connected domain, with $\partial\Omega=\Gamma_D \cup \Gamma_R$, where $\Gamma_D$ and $\Gamma_R$ are disjoint components of $\partial\Omega$. We assume $\Gamma_R$ to be non-empty, while $\Gamma_D$ could also be empty. We denote by $\mathbf{n}$ the outward-pointing unit normal vector field on $\partial\Omega$. We consider the Robin–Dirichlet BVP:
\begin{equation} \label{eqn: Helmholtz}
    \begin{cases}
        \hspace{0.1cm}-\Delta u - \kappa^2u = 0 &  \text{in} \hspace{0.1cm}\Omega, \\
        \hspace{0.1cm}u=g_D & \text{on} \hspace{0.1cm} \Gamma_D, \\
        \hspace{0.1cm}\nabla u \cdot n - i\kappa u = g_R  & \text{on} \hspace{0.1cm} \Gamma_R.
    \end{cases}
\end{equation}
Here $\kappa = k \sqrt{\varepsilon} \in \mathbb{C}$, where $k>0$ is the wave number, that is assumed to be constant, and $\varepsilon$ is the relative permittivity, which value may vary inside of $\Omega$. The functions $g_D$ and $g_R$ are the Dirichlet and impedance boundary data and $i$ is the imaginary unit. We assume that $\Omega$, $\varepsilon$, $g_D$ and $g_R$ are such that $u\in H^{3/2+s}(\Omega)$, for some $s > 0$. Typically, $\Gamma_D$ represents the boundary of the scatterer, and $\Gamma_R$ stands for an artificial truncation of the unbounded region where waves propagate.

Note that one can consider other boundary value problems, such as Neumann boundary conditions on $\partial \Omega$ or mixed conditions as in \cite{Survey}. For now, we will only consider the model problem (\ref{eqn: Helmholtz}).

We introduce a finite element partition $\bigtau_h = \{ K \}$ of $\Omega$; this partition has to be suited to the composition of the domain $\Omega$, in the sense that $\varepsilon$ has to be constant inside a single element $K$. With this assumption, we can consider $\kappa$ as constant inside any element and we have to deal with discontinuities only on faces shared by two different elements. We write $\mathbf{n}_K$ for the outward-pointing unit normal vector on $\partial K$, and $h$ for the mesh width of $\{ \bigtau_h \}$. We denote by $ F_h = \bigcup_{K \in \bigtau_h } \partial K $ and $F_h^I = F_h \setminus \partial \Omega $ the skeleton of the mesh and its inner part.

Given two elements of the partition $K_1, K_2$, a piecewise-smooth function $v$ and vector field $\tau$ on $\bigtau_h$, we introduce  on $\partial K_1 \cap \partial K_2$ the averages and the normal jumps:
\begin{align} \label{jumps}
    \{ \hspace{-0.1cm} \{ v \} \hspace{-0.1cm} \} & := \frac{1}{2}\left( v_{|K_1} + v_{|K_2} \right), & \{ \hspace{-0.1cm} \{ \tau \} \hspace{-0.1cm} \} & := \frac{1}{2}\left( \tau_{|K_1} + \tau_{|K_2} \right), \\ 
    \llbracket v \rrbracket_N & := v_{|K_1}\mathbf{n}_{K_1} + v_{|K_2}\mathbf{n}_{K_2}, & \llbracket \tau \rrbracket_N & := \tau_{|K_1}\cdot \mathbf{n}_{K_1} + \tau_{|K_2}\cdot \mathbf{n}_{K_2}. \nonumber
\end{align}

Furthermore, we denote by $\nabla_h$ the elementwise application of $\nabla$ and write $\partial_n = \mathbf{n} \cdot \nabla_h$ and $\partial_{n_K} = \mathbf{n}_K \cdot \nabla_h$ for the normal derivatives on $\partial \Omega$ and $\partial K$ respectively.

For $s>0$, we define the two main spaces where Trefftz functions live, the broken Sobolev space $H^s(\bigtau_h)$ and the Trefftz space $T(\bigtau_h)$:
\begin{align} \label{broken_sob}
    & H^s(\bigtau_h) := \{ v \in L^2(\Omega) : v_{| K}\in H^s(K) \hspace{0.2cm} \forall K \in \bigtau_h \}, \\ 
    & T(\bigtau_h) := \{ v \in H^1(\bigtau_h) : -\Delta v - \kappa^2v = 0 \hspace{0.1cm} \text{in} \hspace{0.1cm} K \hspace{0.1cm} \text{and} \hspace{0.1cm} \partial_{\mathbf{n}_K}v \in L^2(\partial K) \hspace{0.2cm}\forall K \in \bigtau_h \}. \nonumber
\end{align}
The discrete Trefftz space $V_p(\bigtau_h)$ is a finite-dimensional subspace of $T(\bigtau_h)$ and
can be represented as $V_p(\bigtau_h) = \bigoplus_{K \in \bigtau_h}V_{p_K}(K)$, where $V_{p_K}(K)$ is a $p_K$-dimensional subspace of $T(\bigtau_h)$ of functions supported in $K$.

We now derive the general DG formulation, following \cite{PWDG}. In order to derive the method, we start by writing problem (\ref{eqn: Helmholtz}) as a first order system:
\begin{equation} \label{eqn: Helm first ord}
    \begin{cases}
-i\kappa\mathbf{\sigma} = \nabla u & \text{in} \hspace{0.1cm}\Omega,\\
-i\kappa u-\nabla\cdot\mathbf{\sigma}=0 & \text{in} \hspace{0.1cm}\Omega,\\
u=g_D & \text{on} \hspace{0.1cm} \Gamma_D, \\
-i\kappa \mathbf{\sigma}\cdot\mathbf{n}- i \kappa u=g_R & \text{on} \hspace{0.1cm} \Gamma_R.
\end{cases}
\end{equation} 
By multiplying the first and second equation of (\ref{eqn: Helm first ord}) by smooth test functions $\tau$ and $v$, respectively, and integrating by parts on each $K \in \bigtau_h$, we obtain
\begin{equation} \label{eqn: var}
   \begin{split}
       & \int_K -i \kappa \hspace{0.05cm} \sigma \cdot \overline{\tau} \hspace{0.1cm} dV + \int_K u \overline{\nabla\cdot\tau} \hspace{0.1cm} dV - \int_{\partial K} u \overline{\tau\cdot \mathbf{n}} \hspace{0.1cm} dS = 0 \hspace{0.3cm} \forall \tau \in \mathbf{H}(\text{div};K), \\
        & \int_K -i \kappa \hspace{0.05cm} u \hspace{0.05cm} \overline{v} \hspace{0.1cm} dV + \int_K \sigma \cdot \overline{\nabla v} \hspace{0.1cm} dV - \int_{\partial K} \sigma \cdot \mathbf{n}\overline{v} \hspace{0.1cm} dS = 0 \hspace{0.3cm} \forall v\in H^1(K).
   \end{split} 
\end{equation}
We now replace $u, v$ by $u_p, v_p \in V_p(\bigtau_h)$ and $i \kappa \sigma, \tau$ by $i \kappa \sigma_p, \tau_p \in V_p(\bigtau_h)^d$. Then, we approximate the traces of $u$ and $i \kappa \sigma$ across inter-element boundaries by the so-called \textit{numerical fluxes} denoted by $\hat{u}_p$ and $\widehat{i \kappa \sigma}_p$, which are defined below.

Replacing the elements in (\ref{eqn: var}) with their numerical approximations, we get
\begin{equation} \label{eqn: approx}
   \begin{split}
       & \int_K -i\kappa \sigma_p \cdot \overline{\tau_p} \hspace{0.1cm} dV + \int_K u_p \overline{\nabla\cdot\tau_p} \hspace{0.1cm} dV - \int_{\partial K} \hat{u}_p \overline{\tau_p\cdot \mathbf{n}} \hspace{0.1cm} dS = 0 \hspace{0.3cm} \forall \tau_p \in V_p(K)^d, \\
        & \int_K -i\kappa u_p  \hspace{0.05cm}\overline{v_p} \hspace{0.1cm} dV + \int_K \sigma_p \cdot \overline{\nabla v_p} \hspace{0.1cm} dV - \int_{\partial K} \frac{1}{i \kappa} \widehat{i \kappa \sigma}_p \cdot \mathbf{n}\overline{v}_p \hspace{0.1cm} dS = 0 \hspace{0.3cm} \forall v_p \in V_p(K).
   \end{split} 
\end{equation}

Integrating again by parts the first equation of (\ref{eqn: approx}), we obtain:
\begin{equation}
    \int_K i\kappa \sigma_p \cdot \overline{\tau_p} \hspace{0.1cm} dV = - \int_K \nabla u_p \cdot \overline{\tau_p} \hspace{0.1cm} dV + \int_{\partial K} (u_p - \hat{u}_p)\overline{\tau_p \cdot \mathbf{n}} \hspace{0.1cm} dS.
\end{equation}

Now we assume that $\nabla_hV_p(\bigtau_h) \subset V_p(\bigtau_h)^d$, in this way we can take $\tau_p=\nabla v_p$ in each element; note that with the PWDG method we are going to use in our model, this property is easily verified since the gradient of a plane wave is still a plane wave. Inserting the resulting expression into the second equation of (\ref{eqn: approx}), we obtain:
\begin{equation} \label{1.5}
    \int_K \left(\nabla u_p \cdot \overline{\nabla v_p} - \kappa^2u_p\overline{v_p}\right) dV - \int_{\partial K} (u_p - \hat{u}_p)\overline{\nabla v_p \cdot \mathbf{n}} \hspace{0.1cm} dS + \int_{\partial K} \widehat{i \kappa\sigma}_p \cdot \mathbf{n} \overline{v_p}  \hspace{0.1cm} dS = 0.
\end{equation}
We integrate  by parts once more the first term in (\ref{1.5}) and derive:
\begin{equation}
    \int_K  u_p \hspace{0.05cm} \overline{\left(-\Delta v_p - \kappa^2 v_p\right)} \hspace{0.1cm}dV + \int_{\partial K} \hat{u}_p\overline{\nabla v_p \cdot \mathbf{n}} \hspace{0.1cm}dS + \int_{\partial K} \widehat{i \kappa \sigma}_p \cdot \mathbf{n} \overline{v_p}  \hspace{0.1cm}dS = 0.
\end{equation}
(the boundary term appearing in this integration by parts cancels out with a boundary term already present in (\ref{1.5}))

The first term in this equation vanishes, since $v_p \in V_p(\bigtau_h)$, so we simply have:
\begin{equation} \label{1.7}
    \int_{\partial K} \hat{u}_p\overline{\nabla v_p \cdot \mathbf{n}} \hspace{0.1cm} dS + \int_{\partial K} \widehat{i \kappa \sigma}_p \cdot \mathbf{n} \overline{v_p}  \hspace{0.1cm} dS = 0
\end{equation}

We now give the explicit expression of the numerical fluxes $\hat{u}_p$ and $\widehat{ i \kappa \sigma}_p$.\\
On the internal faces $F\in F_h^I$ we define
\begin{equation} \label{flux_interno}
    \begin{cases}
        \hat{u}_p = \{ \hspace{-0.1cm} \{  u_p \} \hspace{-0.1cm} \}  + \beta \frac{1}{i\xi} \llbracket \nabla_h u_p  \rrbracket_N, & \\
        \widehat{i \kappa \sigma}_p = - \{ \hspace{-0.1cm} \{ \nabla_h u_p \} \hspace{-0.1cm} \}  - i \xi \alpha \llbracket u_p  \rrbracket_N, &      
    \end{cases}
\end{equation}
where $\xi$ is defined on an internal face $F$ as
\begin{equation} \label{xi}
    \xi = \frac{1}{2} \biggl[ \Re (\kappa_1) + \Re (\kappa_2) \biggl],
\end{equation}
with $\kappa_1=k\sqrt{\varepsilon_1}, \kappa_2=k\sqrt{\varepsilon_2}$ being the value of $\kappa$ on the two elements sharing the internal face $F$. The definition of $\xi$ follows from the formulation in Section 3.3 of \cite{Hoearth}.

On faces in $\Gamma_R$ we choose
\begin{equation} \label{flux_Robin}
    \begin{cases}
        \hat{u}_p = u_p -\delta\left( -\frac{1}{i\kappa} \nabla_h u_p \cdot \mathbf{n} + u_p + \frac{1}{i\kappa} g_R  \right), & \\
        \widehat{i \kappa \sigma}_p = -\nabla_h u_p -(1-\delta)\left( - \nabla_h u_p + i \kappa u_p \mathbf{n} +  g_R \mathbf{n} \right). &        
    \end{cases}
\end{equation}

Finally, on faces in $\Gamma_D$ we opt for
\begin{equation} \label{flux_dirichlet}
    \begin{cases}
        \hat{u}_p = g_D, & \\
        \widehat{i \kappa \sigma}_p = -\nabla_h u_p -i \kappa \alpha (u_p-g_D) \mathbf{n}. &        
    \end{cases}
\end{equation}

The parameters $\alpha$, $\beta$ and $\delta$ are called \textit{flux parameters}; the numerical fluxes also take into account the inhomogeneous boundary conditions.

Now we add (\ref{1.7}) over all elements $K \in \bigtau_h$ and obtain:
\begin{equation*}
    \int_{F_h^I} \left( \hat{u}_p \llbracket \overline{\nabla_h v_p} \rrbracket_N + \widehat{i \kappa \sigma}_p \cdot \llbracket \overline{v_p} \rrbracket_N \right) dS + \int_{F_h^B}  \left( \hat{u}_p\overline{\nabla_h v_p \cdot \mathbf{n}} + \widehat{i \kappa \sigma}_p \cdot \mathbf{n}\overline{v_p}\right) dS = 0.
\end{equation*}

Finally, we insert the above defined numerical fluxes and derive the DG method as follows: find $u_p \in V_p(\bigtau_h)$ such that, for all $v_p \in V_p(\bigtau_h)$,
\begin{equation} \label{bilin}
    A_h(u_p, v_p) =L_h(v_p),
\end{equation}
where
\begin{align} \label{forma_bilineare_PWDG}
    & A_h(u, v) :=\\ 
    & \int_{F_h^I} \left( \dlgraffa u \drgraffa \hspace{0.05cm} \llbracket \overline{\nabla_h v} \rrbracket_N  -i\xi^{-1} \beta \llbracket \nabla_h u\rrbracket_N \llbracket \overline{\nabla_h v} \rrbracket_N  -  \{ \hspace{-0.1cm} \{ \nabla_h u \} \hspace{-0.1cm} \} \cdot \llbracket \overline{v} \rrbracket_N  - i\xi \hspace{0.05cm} \alpha \hspace{0.05cm} \llbracket  u\rrbracket_N \cdot \llbracket \overline{v} \rrbracket_N  \right) dS  \nonumber \\ \nonumber
    & + \int_{\Gamma_R} \left( (1-\delta) \hspace{0.05cm} u \hspace{0.05cm} \overline{\nabla_h v \cdot \mathbf{n}} - \delta \hspace{0.05cm}i\kappa^{-1} \nabla_h u \cdot \mathbf{n}\overline{\nabla_h v \cdot \mathbf{n}} -\delta \nabla_h u\cdot \mathbf{n}\overline{v} - i\kappa (1-\delta) u \overline{v}  \right) dS \\ \nonumber
    & + \int_{\Gamma_D} \left( -\nabla_h u \cdot \mathbf{n} \hspace{0.05cm} \overline{v} - i\kappa \hspace{0.05cm}\alpha \hspace{0.05cm} u\hspace{0.05cm} \overline{v} \right) dS, \nonumber
\end{align}
and
\begin{equation} \label{termine_noto_PWDG}
    L_h(v) :=  \int_{\Gamma_R} g_R \left( -\delta i\kappa^{-1} \overline{\nabla_h v \cdot \mathbf{n}}+ (1-\delta) \overline{v} \hspace{0.1cm} \right) dS + \int_{\Gamma_D} g_D \left( -\alpha i\kappa \overline{v} - \overline{\nabla_h v \cdot \mathbf{n}} \right) dS.
\end{equation}

\section{The PWDG Method}
Many different basis functions can be chosen to span the discrete Trefftz space $V_p(\bigtau_h)$; with the Plane Wave Discontinuous Galerkin (PWDG) method we choose to use plane waves as basis functions.

Consider an element $K \in \bigtau_h$; we denote by $V_p(K)$ the plane wave space on $K$ spanned by $p$ plane waves, $p \in \mathbb{N}$:
\begin{equation} \label{pwspace}
    V_p(K) = \{ \hspace{0.1cm} v \in L^2(K) :  v(\mathbf{x}) = \sum_{j=1}^p \alpha_j \exp\{i\kappa\mathbf{d_j} \cdot \mathbf{x}\}, \hspace{0.3cm} \alpha_j \in \mathbb{C} \hspace{0.1cm} \},
\end{equation}
where $\{ \mathbf{d}_j \}_{j=1}^p \subset \mathbb{R}^2$, with $|\mathbf{d_j}|=1$ are different directions. To obtain isotropic approximations, in 2D, uniformly-spaced directions on the unit circle can be chosen as $\mathbf{d}_j = \left( \cos(\frac{2\pi j}{p}), \sin(\frac{2\pi j}{p}) \right) $.

Note that in this case we are choosing the same number of directions in every element $K$, so $p_K=p$ for every $K \in \bigtau_h$.

Then, we define:
\begin{equation} \label{pwglobal}
    V_p(\bigtau_h)=\{ \hspace{0.1cm} v \in L^2(\Omega) : v_{|K} \in V_p(K), \hspace{0.2cm}\forall K \in \bigtau_h \hspace{0.1cm} \}.
\end{equation}

The variational formulation (\ref{bilin}) derived above is still valid for the PWDG method. The main advantage of PWDG is that, using the properties of plane waves functions, we can easily find a closed formula to compute all the boundary integrals, reducing the errors caused by numerical approximation.\\

\begin{prop}
    The PWDG formulation is consistent.
\end{prop}
\begin{proof}
    If $u \in H^2(\Omega)$ solves (\ref{eqn: Helmholtz}), then it holds that $\llbracket u \rrbracket_N= 0$ and $\llbracket \nabla_h u \rrbracket_N= 0$ on $F_h^I$, while $\dlgraffa u \drgraffa = u$ and $\dlgraffa \nabla_ h \drgraffa = \nabla_h u$. Moreover, $u=g_D$ on $\Gamma_D$ and $\nabla_h u \cdot n - i\kappa u = 0$ on $\Gamma_R$.\\
    From the definition of the numerical fluxes $\hat{u}_p$ and $\widehat{i \kappa \sigma}_p$ we have that on all the faces, when we substitute $u$ in the expression of the fluxes, we have that $\hat{u}_p = u$ and $ \hspace{0.05cm}\widehat{i \kappa \sigma}_p= -\nabla_h u$. This gives us that, for every $v_p \in V_p (\bigtau_h)$,
    \begin{equation*}
        A_h(u,v_p)=L_h(v_p),
    \end{equation*}
    that is the consistency of the method.
\end{proof}

\subsection{Matrix formulation in the case of $\varepsilon$ constant} \label{matr_form}
In this section, we assume $\varepsilon$ constant and equal to $1$; under these assumptions we have that $\kappa=k$.

We already introduced the plane wave (\ref{pwspace}) space on $K$ and the global space (\ref{pwglobal}). The PWDG formulation is: find $u_p \in V_p(\bigtau_h)$ such that, for all $v_p \in V_p(\bigtau_h)$,
\begin{equation*}
    A_h(u_p, v_p) =L_h(v_p).
\end{equation*}
Consider $p\in\mathbb{N}$ fixed; since $u_p,v_p \in V_p(\bigtau_h)$, we can write 
\begin{align*}
   & u_p(\mathbf{x})_{|K} = \sum_{l=1}^p u_{l}^K \varphi_{l}^K(\mathbf{x}), \hspace{0.7cm} v_p(\mathbf{x})_{|K} = \sum_{j=1}^p v_{j}^K \varphi_{j}^K(\mathbf{x}),
\end{align*}
where $u_{j}^K,v_{j}^K \in \mathbb{C}$, $\hspace{0.2cm} \forall j=1,...,p$, $\hspace{0.2cm} \forall K\in \bigtau_h$ and 
\begin{equation*}
    \varphi_{j}^K(\mathbf{x})=\begin{cases}
    \exp\{ik\mathbf{d_j} \cdot \mathbf{x}\} & \mathbf{x} \in K \\
    0 & \mathbf{x} \not \in K
\end{cases}, \hspace{0.4cm} \forall j=1,...,p, \hspace{0.2cm} \forall K\in \bigtau_h.
\end{equation*}\\
Now, we can write $V_p(\bigtau_h) = \text{span} \{ \varphi_1, \varphi_2, ... , \varphi_N \}$, where $N$ is the total number of degrees of freedom, i.e. $N=pT$, where $T$ is the number of elements $K\in \bigtau_h$.

Since $A_h(\cdot,\cdot)$ is a bilinear form and $L_h(\cdot)$ is linear, the variational problem is equivalent to: \begin{center}
    find $u_p \in V_p(\bigtau_h)$ such that $A_h(u_p, \varphi_l) =L_h(\varphi_l)$, for all $l = 1,...,N$.
\end{center}

If we write $u_p = \sum_{j=1}^N u_j\varphi_j$, we get that
\begin{equation*}
    A_h(u_p, \varphi_l) = A_h\left(\sum_{j=1}^N u_j\varphi_j, \varphi_l\right) = \sum_{j=1}^N u_j A_h(\varphi_j, \varphi_l) = L_h(\varphi_l), \hspace{0.2cm} \forall l=1,...,N,
\end{equation*}
which can be written as a linear sistem $A\mathbf{u}=\mathbf{L}$, where
\begin{align*}
    \mathbf{u}_l = u_l, & & A_{l,j} = A_h(\varphi_j, \varphi_l), & & \mathbf{L}_l = L_h(\varphi_l).
\end{align*}

We note that if two basis functions $\varphi_l$ and $\varphi_j$ are associated to the same element $K$ or to two adjacent elements $K, K'$, the corresponding matrix entry is non-zero; in all the other cases the entry of the matrix is equal to zero. 
First of all, we have to distinguish between internal and boundary elements. Then, we have to consider two different situations, the first being when $\varphi_l, \varphi_j$ are in the same element $K$, the second when they are in two adjacent elements $K, K'$.

Let's now focus on the case of $\varphi_l, \varphi_j$ on the same internal element $K$:
\begin{flalign*}
    A_{j,l}=A_h(\varphi_l, \varphi_j)=\int_{\partial K} & \big[ \hspace{0.05cm} \{ \hspace{-0.1cm} \{ \varphi_l \} \hspace{-0.1cm} \} \llbracket \overline{\nabla \varphi_j} \rrbracket_N  -ik^{-1} \left( \beta \llbracket \nabla \varphi_l\rrbracket_N \llbracket \overline{\nabla \varphi_j} \rrbracket_N \right) \\ 
     & - \{ \hspace{-0.1cm} \{ \nabla_h \varphi_l \} \hspace{-0.1cm} \} \cdot \llbracket \overline{\varphi_j} \rrbracket_N  - ik \left( \alpha \llbracket  \varphi_l\rrbracket_N \cdot \llbracket \overline{\varphi_j} \rrbracket_N \right) \big] \hspace{0.1cm}dS,
\end{flalign*}
(note that the boundary terms are not present since $K$ is an internal element).\\
We first observe that in this case, since $\varphi_l$ and $\varphi_j$ are associated to the same element, the normal jumps don't change sign because the exterior normal $\mathbf{n}$ of $K$ does not change. Replacing the explicit expressions of the averages and the normal jumps as defined below, and recalling that $\varphi_l, \varphi_j$ are equal to zero outside $K$, we obtain:
\begin{flalign*}
    A_{j,l} = \int_{\partial K} & \big[ \hspace{0.05cm} \frac{1}{2} \hspace{0.05cm} \varphi_l \hspace{0.05cm} \overline{\nabla \varphi_j}\cdot \mathbf{n}  - \beta ik^{-1} \left( \nabla \varphi_l \cdot \mathbf{n} \overline{\nabla \varphi_j}\cdot \mathbf{n} \right)\\ 
    &- \frac{1}{2}  \nabla \varphi_l \cdot \overline{\varphi_j}\mathbf{n} - \alpha ik \left( \varphi_l\mathbf{n} \cdot \overline{\varphi_j}\mathbf{n} \right) \big] \hspace{0.1cm}dS.
\end{flalign*}
We observe that, on the element $K$, 
\begin{align*}
    & \varphi_l(\mathbf{x}) = \exp\{ ik\mathbf{x}\cdot\mathbf{d}_l \}, & & \nabla\varphi_l(\mathbf{x}) = ik\mathbf{d}_l \exp\{ ik\mathbf{x}\cdot\mathbf{d}_l \}=ik\mathbf{d}_l \hspace{0.05cm} \varphi_l(\mathbf{x}), \\
    & \overline{\varphi_l}(\mathbf{x}) = \exp\{ -ik\mathbf{x}\cdot\mathbf{d}_l \}, & & \overline{\nabla\varphi_l}(\mathbf{x}) = -ik\mathbf{d}_l\exp\{ -ik\mathbf{x}\cdot\mathbf{d}_l \}=-ik\mathbf{d}_l \hspace{0.05cm} \overline{\varphi_l}(\mathbf{x}) .
\end{align*} 
So the equation above becomes:
\begin{align*}
    A_{j,l} & = \int_{\partial K} e^{ ik\mathbf{x}\cdot(\mathbf{d}_l-\mathbf{d}_j)} \left[ -\frac{1}{2}(ik\mathbf{d}_j\cdot\mathbf{n}+ik\mathbf{d}_l\cdot\mathbf{n})-\beta ik\mathbf{d}_l\cdot\mathbf{n}\hspace{0.05cm}\mathbf{d}_j\cdot\mathbf{n} -\alpha ik \right] dS \\
    & = \sum_{F\in\mathcal{F}} \int_{F} e^{ ik\mathbf{x}\cdot(\mathbf{d}_l-\mathbf{d}_j)} \left[ -\frac{1}{2}(ik\mathbf{d}_j\cdot\mathbf{n}+ik\mathbf{d}_l\cdot\mathbf{n})-\beta ik\mathbf{d}_l\cdot\mathbf{n}\hspace{0.05cm}\mathbf{d}_j\cdot\mathbf{n} -\alpha ik \right] dS \\
    & = \sum_{F\in\mathcal{F}}  ik \left[ -\frac{1}{2}(\mathbf{d}_j\cdot\mathbf{n}+\mathbf{d}_l\cdot\mathbf{n})-\beta \hspace{0.05cm} \mathbf{d}_l\cdot\mathbf{n}\hspace{0.05cm}\mathbf{d}_j\cdot\mathbf{n} -\alpha \hspace{0.05cm} \right] \int_{F} e^{ ik\mathbf{x}\cdot(\mathbf{d}_l-\mathbf{d}_j)} \hspace{0.05cm} dS,
\end{align*}
where $\mathcal{F}$ is the set containing all the faces of $K$.

Another simple case is when $\varphi_l, \varphi_j$ are associated to the same element $K$, but some of its faces are part of the boundary. In this case the matrix entry is:
\begin{flalign*}
    A_{j,l}=&\int_{\partial K \cap F_h^I} \big[ \hspace{0.05cm} \{ \hspace{-0.1cm} \{ \varphi_l \} \hspace{-0.1cm} \} \llbracket \overline{\nabla \varphi_j} \rrbracket_N  - ik^{-1} \left( \beta \llbracket \nabla \varphi_l\rrbracket_N \llbracket \overline{\nabla \varphi_j} \rrbracket_N \right) \\ 
     & - \{ \hspace{-0.1cm} \{ \nabla_h \varphi_l \} \hspace{-0.1cm} \} \cdot \llbracket \overline{\varphi_j} \rrbracket_N  - ik \left( \alpha \llbracket  \varphi_l\rrbracket_N \cdot \llbracket \overline{\varphi_j} \rrbracket_N \right) \big] \hspace{0.1cm}dS \\
     & + \int_{\partial K \cap \Gamma_R} \big[ \hspace{0.05cm}  (1-\delta)\varphi_l \overline{\nabla \varphi_J \cdot \mathbf{n}} - ik^{-1}(\delta\nabla\varphi_l\cdot\mathbf{n}\overline{\nabla \varphi_J \cdot \mathbf{n}}) \\
     &  - \delta\nabla\varphi_l\cdot\mathbf{n}\overline{\varphi_j} - ik(1-\delta)\varphi_l\overline{\varphi_j}  \hspace{0.05cm}\big] \hspace{0.1cm}dS \\
     & + \int_{\partial K \cap \Gamma_D} \left[ -\nabla_h \varphi_l \cdot \mathbf{n} \hspace{0.05cm} \overline{\varphi_j} - ik \hspace{0.05cm}\alpha \hspace{0.05cm} \varphi_l\hspace{0.05cm} \overline{\varphi_j} \right]  \hspace{0.1cm} dS
\end{flalign*}
(note the first part is the same as before, except that we integrate only over the internal faces of $K$).\\
Replacing the explicit expressions of $\varphi_l, \varphi_j$ on $K$ and computing the averages and the normal jumps we obtain:
\begin{align*}
    A_{j,l} & = \sum_{F\in\mathcal{F}^I} \int_{F} e^{ ik\mathbf{x}\cdot(\mathbf{d}_l-\mathbf{d}_j)} \left[ -\frac{1}{2}(ik\mathbf{d}_j\cdot\mathbf{n}+ik\mathbf{d}_l\cdot\mathbf{n}) - \beta ik\mathbf{d}_l\cdot\mathbf{n}\hspace{0.05cm}\mathbf{d}_j\cdot\mathbf{n} - \alpha ik \right] dS\\ 
    & + \sum_{F\in\mathcal{F}^R} \int_{F} e^{ ik\mathbf{x}\cdot(\mathbf{d}_l-\mathbf{d}_j)}  \left[ ik(1-\delta)(-1-\mathbf{d}_j\cdot\mathbf{n})+ik\delta\mathbf{d}_l\cdot\mathbf{n} (-\mathbf{d}_j\cdot\mathbf{n}-1) \right] dS \\
    & + \sum_{F\in\mathcal{F}^D} \int_{F} e^{ ik\mathbf{x}\cdot(\mathbf{d}_l-\mathbf{d}_j)}  \left[ -ik \mathbf{d}_l \cdot \mathbf{n} -\alpha i k \right] dS \\
    & = \sum_{F\in\mathcal{F}^I} ik\left[ -\frac{1}{2}(\mathbf{d}_j\cdot\mathbf{n}+\mathbf{d}_l\cdot\mathbf{n})-\beta \mathbf{d}_l\cdot\mathbf{n}\hspace{0.05cm}\mathbf{d}_j\cdot\mathbf{n} -\alpha \right] \int_{F} e^{ ik\mathbf{x}\cdot(\mathbf{d}_l-\mathbf{d}_j)} \hspace{0.05cm} dS\\ 
    & + \sum_{F\in\mathcal{F}^R}  ik\left[ (1-\delta)(-1-\mathbf{d}_j\cdot\mathbf{n})+\delta\mathbf{d}_l\cdot\mathbf{n} (-\mathbf{d}_j\cdot\mathbf{n}-1) \right] \int_{F} e^{ ik\mathbf{x}\cdot(\mathbf{d}_l-\mathbf{d}_j)} \hspace{0.05cm} dS \\
    & + \sum_{F\in \mathcal{F}^D}  ik\left[ -\alpha - \mathbf{d}_l \cdot \mathbf{n} \right] \int_{F} e^{ ik\mathbf{x}\cdot(\mathbf{d}_l-\mathbf{d}_j)} \hspace{0.05cm} dS,
\end{align*}
where $\mathcal{F}^I = \mathcal{F} \cap F_h^I$, $\mathcal{F}^R = \mathcal{F} \cap \Gamma_R$ and $\mathcal{F}^D = \mathcal{F} \cap \Gamma_D$.

The last case is when $\varphi_l, \varphi_j$ are associated to two adjacent elements $K, K'$. We have to integrate only over the face $F$ common to the two elements. We observe that this face can't be a boundary face, since it is adjacent to two different elements, so the matrix entry becomes:
\begin{flalign*}
    A_{j,l}=\int_{F} & \big[ \hspace{0.05cm} \{ \hspace{-0.1cm} \{ \varphi_l \} \hspace{-0.1cm} \} \llbracket \overline{\nabla \varphi_j} \rrbracket_N  - ik^{-1} \left( \beta \llbracket \nabla \varphi_l\rrbracket_N \llbracket \overline{\nabla \varphi_j} \rrbracket_N \right) \\ 
     & - \{ \hspace{-0.1cm} \{ \nabla_h \varphi_l \} \hspace{-0.1cm} \} \cdot \llbracket \overline{\varphi_j} \rrbracket_N  - ik \left( \alpha \llbracket  \varphi_l\rrbracket_N \cdot \llbracket \overline{\varphi_j} \rrbracket_N \right) \big] \hspace{0.1cm}dS.
\end{flalign*}
We have that $\varphi_l=0$ outside of $K$ and $\varphi_j=0$ outside of $K'$, so computing the averages is not a problem since they are equivalent to the function multiplied by $\frac{1}{2}$. When we compute the normal jumps we have to be careful, because the outer normal vectors of $K$ and $K'$ are not the same. If we denote with $\mathbf{n}$ the outer normal of $K$ on $F$ and with $\mathbf{n}'$ the outer normal of $K'$, we have that $\mathbf{n}'=-\mathbf{n}$. So, when we compute  the normal jumps on $F$ for $\varphi_j$ and $\nabla\varphi_j$ a negative sign will appear. Since the normal jumps of $\varphi_j$ or $\nabla\varphi_j$ appear exactly once for every term, the matrix entry is simply equal to the case of the internal face we analyzed before, with the difference of a minus sign in front:
\begin{align*}
    A_{j,l} & = - ik \left[ -\frac{1}{2}(\mathbf{d}_j\cdot\mathbf{n}+\mathbf{d}_l\cdot\mathbf{n})-\beta \mathbf{d}_l\cdot\mathbf{n}\hspace{0.05cm}\mathbf{d}_j\cdot\mathbf{n} -\alpha \right] \int_{F} e^{ ik\mathbf{x}\cdot(\mathbf{d}_l-\mathbf{d}_j)} \hspace{0.05cm} dS.
\end{align*}
\begin{figure}[t]
    \centering
    \includegraphics[width=.35\textwidth]{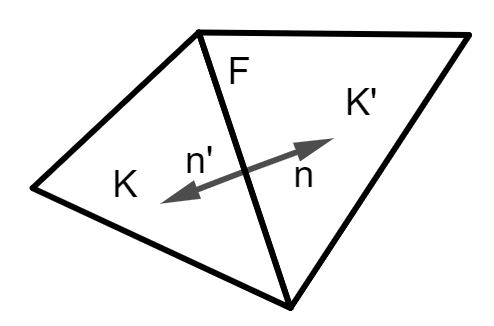}
    \caption{Two adjacent elements of the mesh and the shared face}
    \label{fig:triangoli}
\end{figure}

We can observe that in each of the three cases, the matrix entry is the sum of terms of the form 
\begin{equation*}
    ik \Tilde{c} \int_{F} e^{ ik\mathbf{x}\cdot\mathbf{d}} \hspace{0.05cm} dS,
\end{equation*}
where $\Tilde{c}$ is a real constant and $\mathbf{d} = \mathbf{d}_l-\mathbf{d}_j$. Until now we considered a general element $K$ with its face $F$; let's now focus on the 2D case, so that $F$ is a segment. The integral on $F$ can be computed exactly with a closed formula.\\
Let the edge $F$ be parameterized by 
\begin{equation*}
    \gamma : [0, 1] \to F \subset \mathbb{R}^2, \hspace{0.4cm} t \mapsto \mathbf{a} + t \cdot (\mathbf{b} - \mathbf{a}),
\end{equation*}
where $\mathbf{a},\mathbf{b} \in \mathbb{R}^2$ are the end points of $F$. Then
\begin{equation*}
    \int_{F} e^{ ik\mathbf{x}\cdot\mathbf{d}} \hspace{0.05cm} dS = \| \mathbf{b} - \mathbf{a} \| \int_{0}^{1} e^{ ik( \mathbf{a} + t \cdot (\mathbf{b} - \mathbf{a}))\cdot\mathbf{d}} \hspace{0.05cm} dt = \| \mathbf{b} - \mathbf{a} \| e^{ik \mathbf{a} \cdot \mathbf{d}}\int_{0}^{1} e^{ ik( \mathbf{b} - \mathbf{a})\cdot\mathbf{d} \hspace{0.05cm} t} \hspace{0.05cm} dt.
\end{equation*}
Using
\begin{equation*}
    \psi(z) := \int_0^1 e^{zt} dt = \begin{cases} \frac{e^z-1}{z} & z\neq 0, \\ 1 & z=0, \end{cases}
\end{equation*}
we can write
\begin{equation*}
    \int_{F} e^{ ik\mathbf{x}\cdot\mathbf{d}} \hspace{0.05cm} dS = \| \mathbf{b} - \mathbf{a} \| \hspace{0.05cm}\psi( ik( \mathbf{b} - \mathbf{a})\cdot\mathbf{d}).
\end{equation*}
Note that the value $ik( \mathbf{b} - \mathbf{a})\cdot\mathbf{d}$ becomes equal to zero either when the direction $\mathbf{d}=\mathbf{d}_l-\mathbf{d}_j=\mathbf{0}$ or when $( \mathbf{b} - \mathbf{a})\cdot\mathbf{d} =0$, which means that the direction of the plane wave is orthogonal to the edge of the element.

For the right-hand side of the linear system, we observe that we only have an integral on the exterior boundary, so this term is zero for basis function associated to internal elements. For the boundary elements we have that:
\begin{align*}
    \mathbf{L}_j & =L_h(\varphi_j) = \int_{\partial K \cap \Gamma_R} g_R \left( -\delta ik^{-1} \overline{\nabla \varphi_j \cdot \mathbf{n}}+ (1-\delta) \overline{\varphi_j} \hspace{0.1cm} \right) dS\\
    & + \int_{\partial K \cap \Gamma_D} g_D \left( -\alpha ik \hspace{0.1cm} \overline{\varphi_j} - \overline{\nabla \varphi_j \cdot \mathbf{n}} \right) dS\\
    & = \sum_{F\in\mathcal{F}^R} \int_{F} g_R(\mathbf{x}) \hspace{0.1cm}e^{ -ik\mathbf{x}\cdot\mathbf{d}_j} \left[ -\delta \hspace{0.05cm} \mathbf{d}_j \cdot \mathbf{n} + (1-\delta) \right] dS \\
    & + \sum_{F\in\mathcal{F}^D} \int_{F} g_D(\mathbf{x}) \hspace{0.1cm}e^{ -ik\mathbf{x}\cdot\mathbf{d}_j} \left[ -\alpha \hspace{0.1cm} ik + ik\mathbf{d}_j \cdot \mathbf{n} \right] dS \\
    & = \sum_{F\in\mathcal{F}^R} \left[ \hspace{0.05cm} \delta \hspace{0.05cm}( -\mathbf{d}_j \cdot \mathbf{n}- 1) + 1 \right] \int_{F} e^{ -ik\mathbf{x}\cdot\mathbf{d}_j} g_R(\mathbf{x}) dS \\
    & + \sum_{F\in\mathcal{F}^D} ik \left[ -\alpha + \mathbf{d}_j \cdot \mathbf{n} \right] \int_{F} e^{ -ik\mathbf{x}\cdot\mathbf{d}_j} g_D(\mathbf{x}) dS.
\end{align*}
In this case we are not guaranteed to find a closed formula for the computation of $\int_{F} e^{ -ik\mathbf{x}\cdot\mathbf{d}_j} g_R(\mathbf{x}) dS$ and $\int_{F} e^{ -ik\mathbf{x}\cdot\mathbf{d}_j} g_D(\mathbf{x}) dS$, so we will need to approximate it using an appropriate quadrature rule. In the case of an incident plane wave we can derive a closed formula to compute this integral.

\section{Error analysis}
We develop our a priori error analysis under the following additional assumptions:
\begin{enumerate}
    \item $\Omega$ is convex;
    \item each element $K$ of $\bigtau_h$ is a convex Lipschitz domain;
    \item there exists a constant $\rho \in (0, 1)$ such that each element $K \in \bigtau_h$ contains a ball of radius $\rho h_K$ (shape regularity);
    \item there exists a constant $\mu \in (0, 1)$ such that, for each $K \in \bigtau_h$, $h_K \geq \mu \hspace{0.05cm} h$ (quasi-uniformity);
    \item $\alpha, \beta$, and $\delta$ are real, strictly positive, and independent of $p, h$, and $k$, with $0<\delta\leq 1/2$.
\end{enumerate}
We consider the broken Sobolev space $H^s(\bigtau_h)$ and the Trefftz space $T(\bigtau_h)$ defined in (\ref{broken_sob}) and define on the Trefftz space the following mesh- and flux-dependent seminorms:
\begin{align*}
    \vertiii{w}_{\text{TDG}}^2 & := \left\| \xi^{-\frac{1}{2}}\beta^{\frac{1}{2}} \llbracket \nabla_h w \rrbracket_N \right\|^2_{L^2(F_h^I)} + \left\| \xi^{\frac{1}{2}}\alpha^{\frac{1}{2}} \llbracket w \rrbracket_N \right\|^2_{L^2(F_h^I)} \\
    & + \left\| \kappa^{-\frac{1}{2}} \delta^{\frac{1}{2}} \nabla_h w \cdot \mathbf{n} \right\|^2_{L^2(\Gamma_R)} + \left\| \kappa^{\frac{1}{2}} (1-\delta)^{\frac{1}{2}} w \right\|^2_{L^2(\Gamma_R)} + \left\| \kappa^{\frac{1}{2}} \alpha^{\frac{1}{2}} w \right\|^2_{L^2(\Gamma_D)}, \\
    \vertiii{w}_{\text{TDG}^+}^2 & := \vertiii{w}_{\text{TDG}}^2 + \left\| \xi^{\frac{1}{2}} \beta^{-\frac{1}{2}} \dlgraffa w \drgraffa \right\|^2_{L^2(F_h^I)} + \left\| \xi^{-\frac{1}{2}} \alpha^{-\frac{1}{2}} \dlgraffa \nabla_h w \drgraffa \right\|^2_{L^2(F_h^I)} \\
    &  + \left\| \kappa^{\frac{1}{2}} \delta^{-\frac{1}{2}} w \right\|^2_{L^2(\Gamma_R)} + \left\| \kappa^{-\frac{1}{2}} \alpha^{-\frac{1}{2}}\nabla_h w \cdot \mathbf{n} \right\|^2_{L^2(\Gamma_D)}.
\end{align*}
It holds that these seminorms are indeed norms in the Trefftz space $T(\bigtau_h)$.
\begin{prop} \label{prop:norm}
    The seminorms $\vertiii{\cdot}_{\text{TDG}}$ and $\vertiii{\cdot}_{\text{TDG}^+}$ are norms in the Trefftz space $T(\bigtau_h)$.
\end{prop}
\begin{proof}
    Let $w\in T (\bigtau_h)$ be such that $\vertiii{w}_{\text{TDG}}^2=0$. Then $w \in H^2(\Omega)$ and satisfies $\Delta w + \kappa^2 w = 0$ in $\Omega$, $w = 0$ on $\partial \Omega$, and $\nabla w \cdot \mathbf{n} = 0$ on $\Gamma_R$, which implies $w=0$ on $\Gamma_D$ and $\nabla w \cdot \mathbf{n}- i\kappa w = 0$ on $\Gamma_R$. The uniqueness of the solution of problem (\ref{eqn: Helmholtz}), which is given in Theorem 2.1 of \cite{loc_ref}, gives $w = 0$.
\end{proof}
\begin{prop} \label{prop:imag}
    For every $w\in T (\bigtau_h)$ it hold that 
    \begin{equation*}
        \Im \hspace{0.05cm} [ A_h(w,w) ] = -\vertiii{w}_{\text{TDG}}^2\hspace{0.05cm}.
    \end{equation*}
\end{prop}
\begin{proof}
    Provided that $u, v \in T (\bigtau_h)$, local integration by parts permits us to rewrite the bilinear form (\ref{forma_bilineare_PWDG}) as
    \begin{align*}
        A_h(u, v) & = (\nabla_h u, \nabla_h v)_{L^2(\Omega)} - \int_{F_h^I} \llbracket u \rrbracket_N \cdot \dlgraffa \overline{\nabla_h v}\drgraffa \hspace{0.1cm} dS - \int_{F_h^I} \dlgraffa \nabla_h u \drgraffa \cdot \llbracket \overline{v} \rrbracket_N \hspace{0.1cm} dS \\
        & - \int_{\Gamma_R} \delta \hspace{0.05cm} u \hspace{0.05cm} \overline{\nabla_h v \cdot \mathbf{n}} \hspace{0.1cm} dS - \int_{\Gamma_R} \delta \hspace{0.05cm} \nabla_h u \cdot \mathbf{n} \hspace{0.05cm} \overline{ v } \hspace{0.1cm} dS \\
        & - i \int_{F_h^I} \xi^{-1} \beta \llbracket \nabla_h u \rrbracket_N \hspace{0.05cm} \llbracket \overline{\nabla_h v} \rrbracket_N \hspace{0.1cm} dS - i \int_{\Gamma_R} \kappa^{-1} \delta \hspace{0.05cm} \nabla_h u \cdot \mathbf{n} \hspace{0.05cm} \overline{\nabla_h v \cdot \mathbf{n}} \hspace{0.1cm} dS \\
        & - i \int_{F_h^I} \xi \alpha \hspace{0.05cm} \llbracket u \rrbracket_N \cdot \llbracket \overline{v} \rrbracket_N \hspace{0.1cm} dS - i \int_{\Gamma_R} \kappa (1-\delta) \hspace{0.05cm} u \hspace{0.05cm} \overline{ v } \hspace{0.1cm} dS \\
        & - \int_{\Gamma_D} \hspace{0.05cm} \nabla_h u \cdot \mathbf{n} \hspace{0.05cm} \overline{ v } \hspace{0.1cm} dS -\int_{\Gamma_D} \hspace{0.05cm} u \hspace{0.05cm} \overline{ \nabla_h v \cdot \mathbf{n} } \hspace{0.1cm} dS \\
        & - i \int_{\Gamma_D} \kappa \alpha  \hspace{0.05cm} u \hspace{0.05cm} \overline{ v } \hspace{0.1cm} dS  - \hspace{0.05cm}(\kappa u, \kappa v)_{L^2(\Omega)}.
    \end{align*}
    
    Therefore, from the expression above we have
    \begin{align*}
        A_h(w,w) & = \| \nabla_h w \|^2_{L^2(\Omega)} - \Re \left[ \int_{F_h^I} \llbracket w \rrbracket_N \cdot \dlgraffa \overline{\nabla_h w}\drgraffa \hspace{0.1cm} dS + \int_{\Gamma_R} \delta \hspace{0.05cm} w \hspace{0.05cm} \overline{\nabla_h w \cdot \mathbf{n}} \hspace{0.1cm} dS \right] \\
        & - i \| \xi^{-\frac{1}{2}} \beta^{1/2} \llbracket \nabla_h w \rrbracket_N \|^2_{L^2(F_h^I)} - i \| \kappa^{-\frac{1}{2}} \delta^{1/2} \nabla_h w \cdot \mathbf{n} \|^2_{L^2(\Gamma_R)} \\
        & - i\| \xi^{\frac{1}{2}} \alpha^{1/2} \llbracket w \rrbracket_N \|^2_{L^2(F_h^I)} - i \| \kappa^{\frac{1}{2}} (1-\delta)^{1/2} w \|^2_{L^2(\Gamma_R)}\\
        & - \Re \left[ \int_{\Gamma_D}  \hspace{0.05cm} w \hspace{0.05cm} \overline{\nabla_h w \cdot \mathbf{n}} \hspace{0.1cm} dS \right] - i\| \kappa^{\frac{1}{2}}\alpha^{1/2} w\|^2_{L^2(\Gamma_D)} - \| \kappa w \|^2_{L^2(\Omega)},
    \end{align*}
    from which, by taking the imaginary part, we get the result.
\end{proof}
\begin{rem}
    We can observe that the well posedness of the PWDG method follows directly from Propositions \ref{prop:norm} and \ref{prop:imag}. If $A_h(u_p, v_p) = 0$ for all $v_p \in V_p(\bigtau_h)$, then we have $A_h(u_p, u_p) = 0$, and so $\vertiii{u_p}_{\text{TDG}} = 0$, which implies $u_p = 0$, since $\vertiii{\cdot}_{\text{TDG}}$ is a norm on the Trefftz space.
\end{rem}
We also have a continuity property of the TDG sesquilinear form.
\begin{prop}
    There exists a constant $C > 0$ independent of $h, p$, and $k$ such that, for all $v, w \in H^2(\bigtau_h)$,
    \begin{equation*}
        |A_h(v,w)| \leq C \vertiii{v}_{\text{TDG}^+} \hspace{0.1cm} \vertiii{w}_{\text{TDG}} \hspace{0.05cm}.
    \end{equation*}
\end{prop}
\begin{proof}
    This result follows from the definition of $A_h(\cdot, \cdot)$, $(1 - \delta)^{-1/2} \leq \delta^{-1/2}$, and repeated applications of the (weighted) Cauchy–Schwarz inequality.
\end{proof}
In the next proposition, we prove quasi-optimality of the PWDG method in the $\vertiii{\cdot}_{\text{TDG}}$-norm.
\begin{prop}
    Let u be the analytical solution to (\ref{eqn: Helmholtz}) and let up be the PWDG solution. Then, there exists a constant $C > 0$ independent of $h, p$, and $k$ such that
    \begin{equation*}
        \vertiii{u-u_p}_{\text{TDG}} \leq C \inf_{v_p \in V_p(\bigtau_h)} \vertiii{u-v_p}_{\text{TDG}^+}\hspace{0.05cm}.
    \end{equation*}
\end{prop}
\begin{proof}
    We apply the triangle inequality and write, for all  $v_p \in V_p(\bigtau_h)$,
    \begin{equation*}
        \vertiii{u-u_p}_{\text{TDG}} \leq \vertiii{u-v_p}_{\text{TDG}} + \vertiii{v_p-u_p}_{\text{TDG}}\hspace{0.05cm}.
    \end{equation*}
    Since $u_p - v_p \in T(\bigtau_h)$, Proposition \ref{prop:imag} gives
    \begin{equation*}
        \vertiii{u_p-v_p}_{\text{TDG}}^2 = -\Im \hspace{0.05cm} [ A_h(u_p-v_p,u_p-v_p) ].
    \end{equation*}
    From Galerkin orthogonality and continuity of $A_h(\cdot, \cdot)$, we have
    \begin{equation*}
        \vertiii{u_p-v_p}_{\text{TDG}}^2 \leq \vertiii{u-v_p}_{\text{TDG}^+} \hspace{0.1cm} \vertiii{u-v_p}_{\text{TDG}}\hspace{0.05cm},
    \end{equation*}
    which gives us the result.
\end{proof}
It is possible to bound the $L^2$-norm of any Trefftz function by using a duality argument and the convexity of the domain $\Omega$. We recall that $k$ is assumed to be constant.
\begin{lem} \label{lem:L2}
    Under the assumptions of the duality argument in \cite{PWDG}, there exists a constant $C > 0$ independent of $h, p$, and $k$ such that, for any $w \in T (\bigtau_h)$,
    \begin{equation*}
        \| w \|_{L^2(\Omega)} \leq C \operatorname{diam}(\Omega)\left( k^{-1/2} h^{-1/2} + k^{1/2} h^{1/2} \right) \vertiii{w}_{\text{TDG}}\hspace{0.05cm}.
    \end{equation*}
\end{lem}
We apply Lemma \ref{lem:L2} to $u-u_p \in T (\bigtau_h)$ and bound the $L^2$-norm of the error by its $\vertiii{\cdot}_{\text{TDG}}$-norm.
\begin{corol}
    Under the assumptions of the duality argument in \cite{PWDG}, let $u$ be the analytical solution to (\ref{eqn: Helmholtz}) and let $u_p$ be the PWDG solution. Then, there exists a constant $C > 0$ independent of $h, p$, and $k$ such that
    \begin{equation*}
        \| u-u_p \|_{L^2(\Omega)} \leq C \operatorname{diam}(\Omega)\left( k^{-1/2} h^{-1/2} + k^{1/2} h^{1/2} \right) \vertiii{u-u_p}_{\text{TDG}}\hspace{0.05cm}.
    \end{equation*}
\end{corol} \noindent

\section{Convergence Experiments}
In this section, we numerically investigate the $p$-convergence of the PWDG method for regular and singular solutions of the Helmholtz equation in two dimensions. We consider the case of only impedance boundary conditions, so we assume $\Gamma_R = \partial \Omega$ and $\Gamma_D = \emptyset$, as we try to reproduce the results in \cite{PWDG}. We also assume that $\varepsilon$ is constant and equal to $1$ in all $\Omega$, so we have $\kappa=k$.

For $p$-convergence, we assume the mesh $\bigtau_h$ to be fixed, and we vary only the number of of plane wave directions $p$.

We consider a square domain $\Omega = [0, 1] \times [-0.5, 0.5]$, partitioned by a mesh consisting of eight triangles, so that $h=1 / \sqrt{2}$. We fix $k=10$, such that an entire wavelength $\lambda = 2\pi/\omega \simeq 0.628$ is completely contained in $\Omega$.

\subsection{Circular waves}
\begin{figure}[h]
\centering
\subfloat[][\emph{The triangulation on $\Omega$}]
{\includegraphics[width=.4\textwidth]{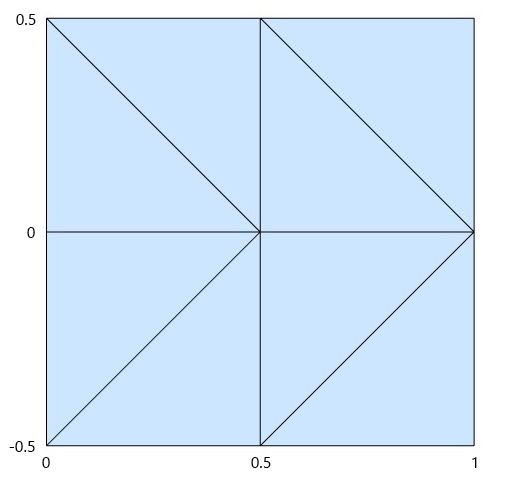}} \quad
\subfloat[][\emph{$\xi=1$}]
{\includegraphics[width=.4\textwidth]{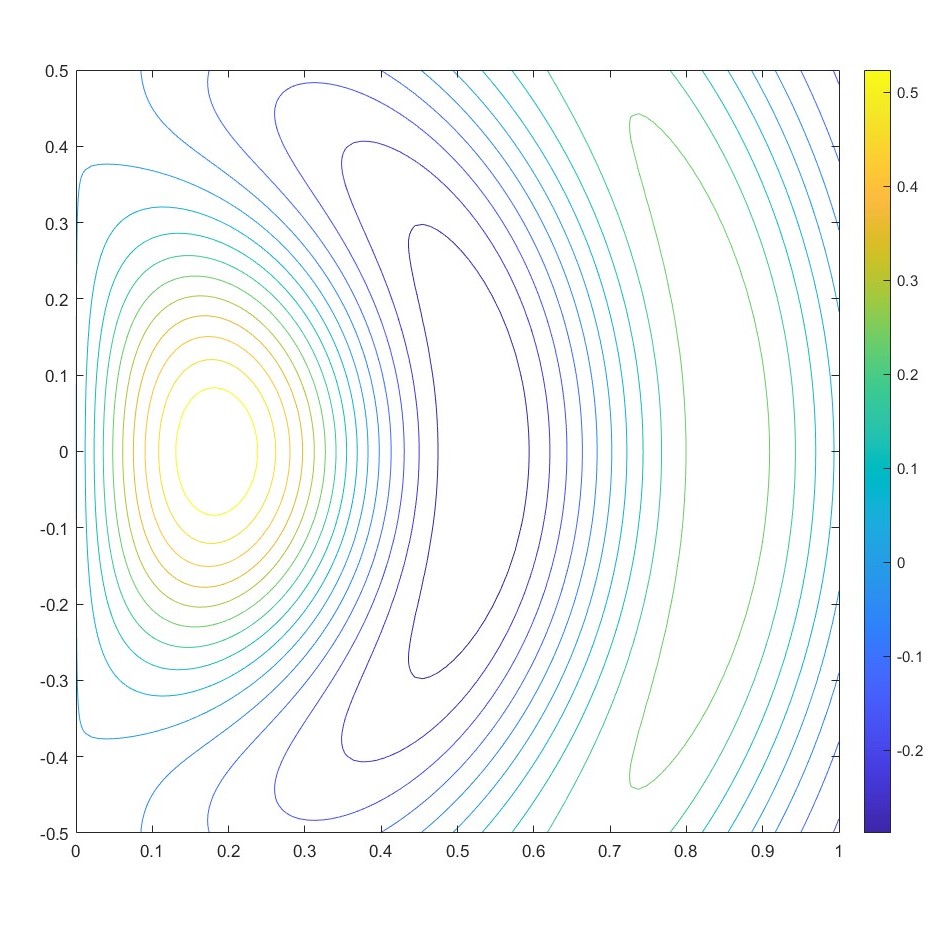}} \quad
\subfloat[][\emph{$\xi=2/3$}]
{\includegraphics[width=.4\textwidth]{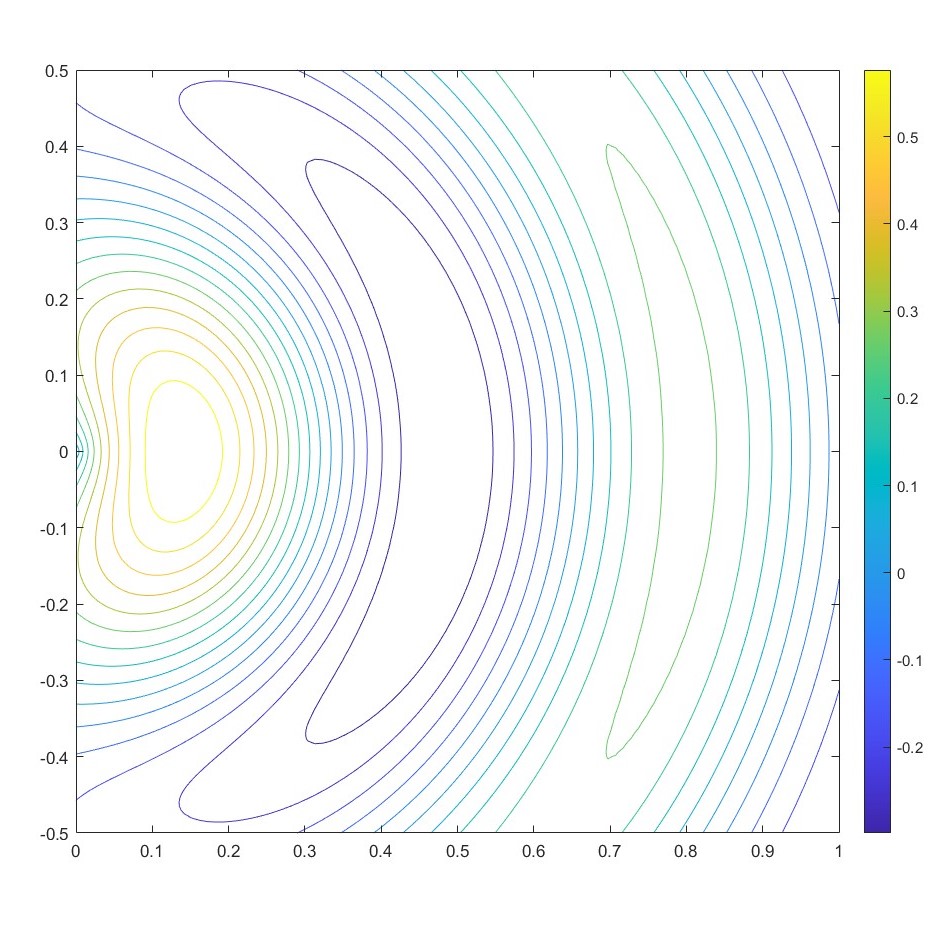}} \quad
\subfloat[][\emph{$\xi=3/2$}]
{\includegraphics[width=.4\textwidth]{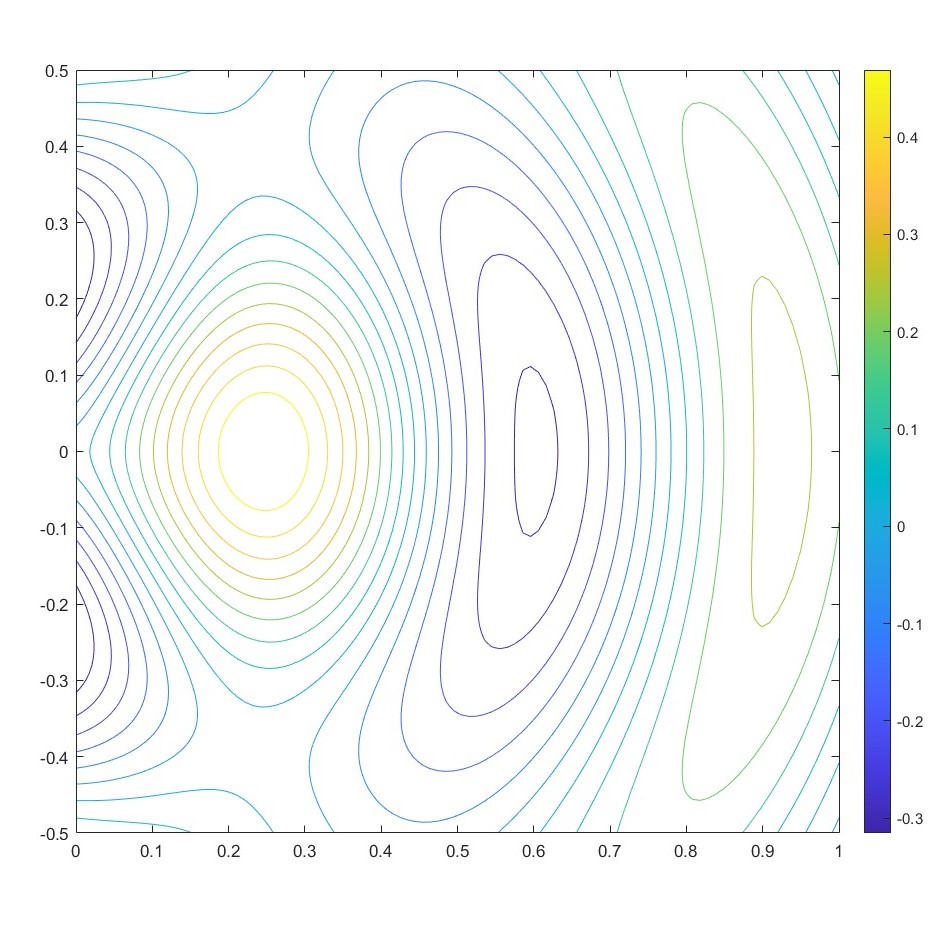}} \quad
\caption{The domain triangulation and the analytical solutions for the different values of $\xi$ and $k = 10$}
\label{fig:contours}
\end{figure} \noindent
\noindent We choose the impedance boundary conditions in such a way that the analytical solutions are the circular waves given, in polar coordinates $\mathbf{x} = (r \cos \theta, r \sin \theta)$, by
\begin{equation*}
    u(\mathbf{x}) = J_\xi (kr) \cos(\xi \theta),
\end{equation*}
where, for $\xi \geq 0$, $J_\xi$ denotes the Bessel function of the first kind and order $\xi$. If $\xi \in \mathbb{N}$, $u$ can be analytically extended to a Helmholtz solution in $\mathbb{R}^2$, while, if $\xi \not\in \mathbb{N}$, its derivatives have a singularity at the origin. In figure \ref{fig:contours} we display the profiles of the analytical solutions corresponding to these three cases.
\begin{figure}[h]
\centering
\subfloat[][\emph{$\xi=1$}]
{\includegraphics[width=.48\textwidth]{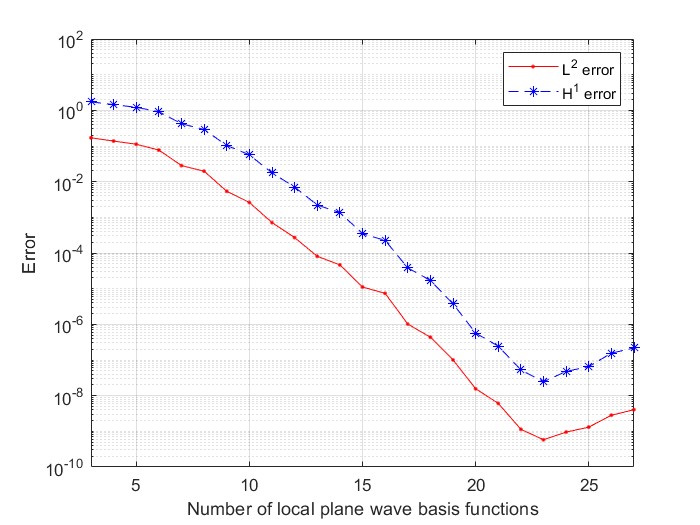}} \quad
\subfloat[][\emph{$\xi=2/3$}]
{\includegraphics[width=.48\textwidth]{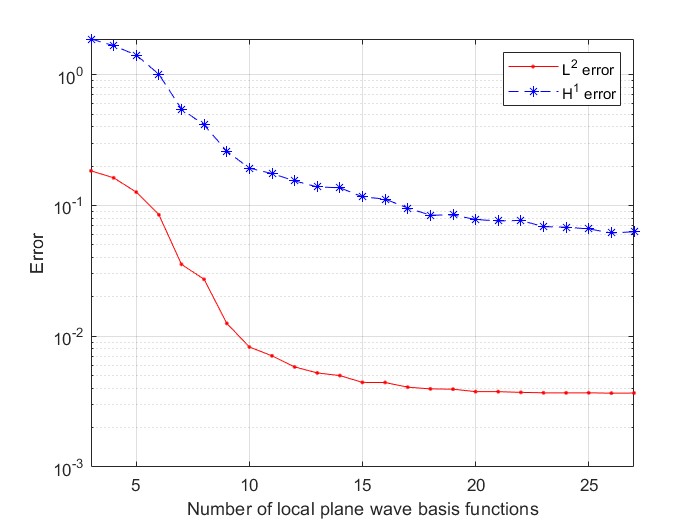}} \quad
\subfloat[][\emph{$\xi=3/2$}]
{\includegraphics[width=.48\textwidth]{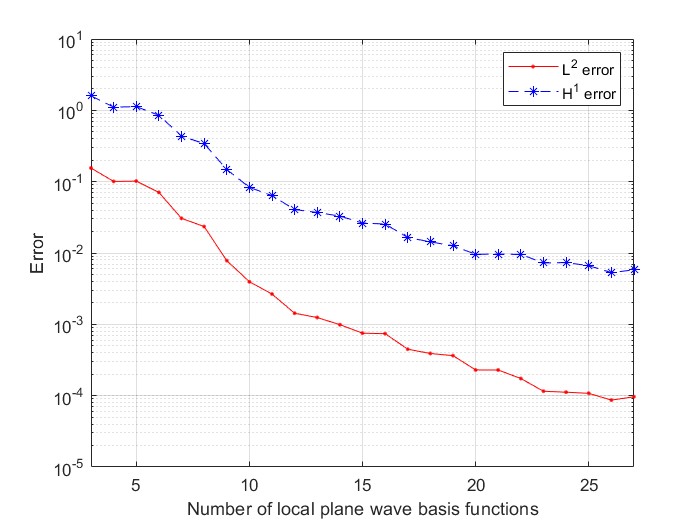}} \quad
\caption{The errors in $L^2$ and $H^1$ norm for the different values of $\xi$ and $p \in \{3,...,27\}$}
\label{fig:errors}
\end{figure}
\\
We compute the numerical solutions in the regular case $\xi = 1$ and in the singular cases $\xi = 2/3$ and $\xi = 3/2$ with $k=10$ in all the experiments. The system matrix is computed exactly using the formulas presented above. To compute the right hand side of the linear system, we use a Gauss-Legendre quadrature rule with $10$ quadrature points on each mesh edge.\\
We consider as numerical fluxes those corresponding to the original ultra weak variational formulation (UWVF) of Cessenat and Després \cite{UWVF}, which are $\alpha=\beta=\delta=\frac{1}{2}$ and we compute the $L^2$ and $H^1$ error of the numerical solution against the exact one. To evaluate the integrals over the triangles of the mesh, we use a Duffy quadrature rule, which is presented in \cite{Duffy}.

In figure \ref{fig:errors} we display the error of the numerical solution against the exact one for different values of the number of plane waves $p$, from $3$ to $27$. The results are comparable to those in \cite{PWDG}.

These plots highlight three different regimes for increasing $p$: a first preasymptotic region with slow convergence, a second region of faster convergence, and finally a third region with a sudden stalling of convergence, due to the impact of round-off and ill-conditioning of the linear system.

\newpage
\subsection{Approximation of a plane wave}
\begin{figure}[h]
    \centering
    \includegraphics[width=.55\textwidth]{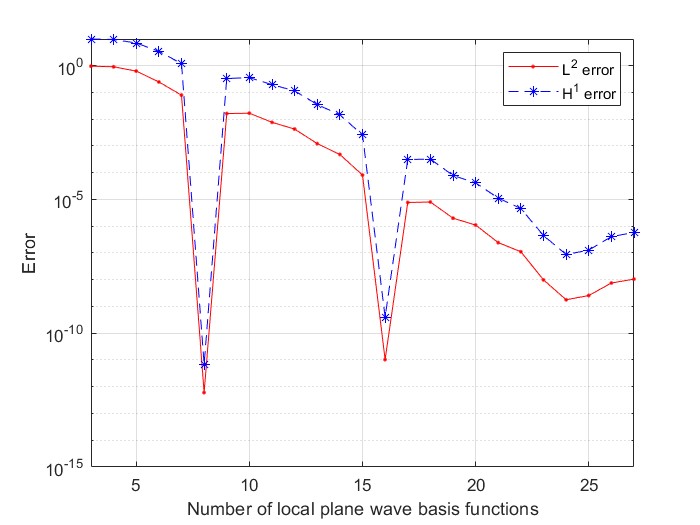}
    \caption{The errors in $L^2$ and $H^1$ norm for $p \in \{3,...,27\}$}
    \label{fig:planewave}
\end{figure}
\noindent Another simple convergence test we can perform with impedance boundary condition is such that the exact solution is a plane wave with direction $\mathbf{d}$ and $k=10$. We choose the direction of the wave in the way that $\mathbf{d}$ is contained in the ensemble of directions only for certain values of $p$. We expect that for this particular values, the numerical solution would be quasioptimal, as the exact solution is part of the discrete space.

In the numerical test we choose $\mathbf{d} = \left( \frac{\sqrt{2}}{2}, \frac{\sqrt{2}}{2} \right)$, so that the exact solution is part of the discrete space only if $p$ is a multiple of $8$. In Figure \ref{fig:planewave} we display the errors obtained in the numerical test. As we expected, for some values of $p$ the error decreases drastically, as we can observe in the figure for $p=8$, since for this values the exact solution belongs to the plane wave space. Increasing the number of plane wave directions we observe again the impact of round-off and ill-conditioning.

\section{PWDG method with different values of $\varepsilon$}
We now test the PWDG method in a region where $\varepsilon$ is not constant. The equation is the following
\begin{equation}
    \begin{cases}
        \hspace{0.1cm}-\Delta u - \kappa^2 u = 0 &  \text{in} \hspace{0.1cm}\Omega, \\
        \hspace{0.1cm}\nabla u \cdot n - i\kappa u = g_R  & \text{on} \hspace{0.1cm} \partial \Omega,
    \end{cases}
\end{equation}
with $\Omega=[0,2\pi]\times [-H,H]$, with $H>0$, $\kappa=k\sqrt{\varepsilon}$, with $k$ constant and $\varepsilon=\varepsilon(\mathbf{x})$ piecewise constant in $\Omega$. Since the value of $\varepsilon$ changes between the regions, also the basis functions of the method are going to be different.

As usual we introduce a finite element partition $\bigtau_h = \{ K \}$ of $\Omega$, which has to be suited to the conformation of the scatterer, i.e. the relative permittivity $\varepsilon$ has to be constant on each element of the mesh. Consider an element $K \in \bigtau_h$; we denote by $V^n_p(K)$ the plane wave space on K spanned by $p$ plane waves, $p \in \mathbb{N}$:
\begin{align*}
    V^n_p(K) & = \{ \hspace{0.1cm} v \in L^2(K) :  v(\mathbf{x}) = \sum_{j=1}^p \alpha_j \exp\{ik \sqrt{\varepsilon_n}\mathbf{d_j} \cdot \mathbf{x}\}, \hspace{0.3cm} \alpha_j \in \mathbb{C} \hspace{0.1cm} \},
\end{align*}
where $\{ \mathbf{d}_j \}_{j=1}^p$, with $|\mathbf{d_j}|=1$ are different directions and $\kappa_n=k \sqrt{\varepsilon_n}$ is the value of the relative permittivity in the region where $K$ is located.

As already seen, we can write the discrete problem as a linear system. We still have three different cases, with the first two (the case of two functions associated to the same element) remaining the same, with the only difference being that we have to change the values of all $k$ appearing in the matrix entries.

In the matrix entries corresponding to two basis functions associated to adjacent elements, we have to be careful because the wave numbers could be different. Here $\varphi_l, \varphi_j$ are associated to two adjacent elements $K, K'$ and they have respectively $\kappa_1$ and $\kappa_2$ as wave numbers, meaning that
\begin{align*}
    & \varphi_l(\mathbf{x}) = \exp\{ i\kappa_1\mathbf{x}\cdot\mathbf{d}_l \}, & & \nabla\varphi_l(\mathbf{x}) = i\kappa_1\mathbf{d}_l \exp\{ ik_1\mathbf{x}\cdot\mathbf{d}_l \}=i\kappa_1\mathbf{d}_l \hspace{0.05cm} \varphi_l(\mathbf{x}), \\
    & \varphi_j(\mathbf{x}) = \exp\{ i\kappa_2\mathbf{x}\cdot\mathbf{d}_j \}, & & \nabla\varphi_j(\mathbf{x}) = i\kappa_2\mathbf{d}_j \exp\{ ik_2\mathbf{x}\cdot\mathbf{d}_j \}=i\kappa_2\mathbf{d}_j \hspace{0.05cm} \varphi_j(\mathbf{x}),
\end{align*} 
where $\varphi_l$ is defined $K$ and $\varphi_j$ on $K'$.

We have to integrate only over the face $F$ common to the two elements. We use the definition of the numerical fluxes in (\ref{flux_interno}). The matrix entry becomes:
\begin{flalign*}
    A_{j,l}=\int_{F} & \big[ \hspace{0.05cm} \{ \hspace{-0.1cm} \{ \varphi_l \} \hspace{-0.1cm} \} \llbracket \overline{\nabla \varphi_j} \rrbracket_N  - i\xi^{-1} \left( \beta \llbracket \nabla \varphi_l\rrbracket_N \llbracket \overline{\nabla \varphi_j} \rrbracket_N \right) \\ 
     & - \{ \hspace{-0.1cm} \{ \nabla_h \varphi_l \} \hspace{-0.1cm} \} \cdot \llbracket \overline{\varphi_j} \rrbracket_N  - i\xi \left( \alpha \llbracket  \varphi_l\rrbracket_N \cdot \llbracket \overline{\varphi_j} \rrbracket_N \right) \big] \hspace{0.1cm}dS,
\end{flalign*}
where $\xi$ is defined as in (\ref{xi}).

Computing the averages and the normal jumps, we get the following expression for $A_{j,l}$:
\begin{align*}
    A_{j,l} & = \left[\frac{1}{2}(i\kappa_2\mathbf{d}_j\cdot\mathbf{n}+i\kappa_1\mathbf{d}_l\cdot\mathbf{n}) - \beta i \xi^{-1} \kappa_1 \kappa_2\mathbf{d}_l\cdot\mathbf{n}\hspace{0.05cm}\mathbf{d}_j\cdot\mathbf{n} - \alpha i\xi \right] \int_{F} e^{ i\kappa_1\mathbf{x}\cdot\mathbf{d}_l - i\kappa_2\mathbf{x}\cdot\mathbf{d}_j}.
\end{align*}
Note that if $\kappa_1 = \kappa_2$ we get the same expression as before. The matrix entry can be computed with a closed formula similar to the one seen in Section \ref{matr_form}.

The right hand side of the linear system does not change from the case with $\kappa$ constant, since on the boundary faces we can always assume this to be true when we define the mesh. The only difference is that the integral term and the multiplied constant will change in relationship to the region of $\Omega$.

\subsection{Numerical Experiments}
\begin{figure}[h]
    \centering
    \includegraphics[width=.85\textwidth]{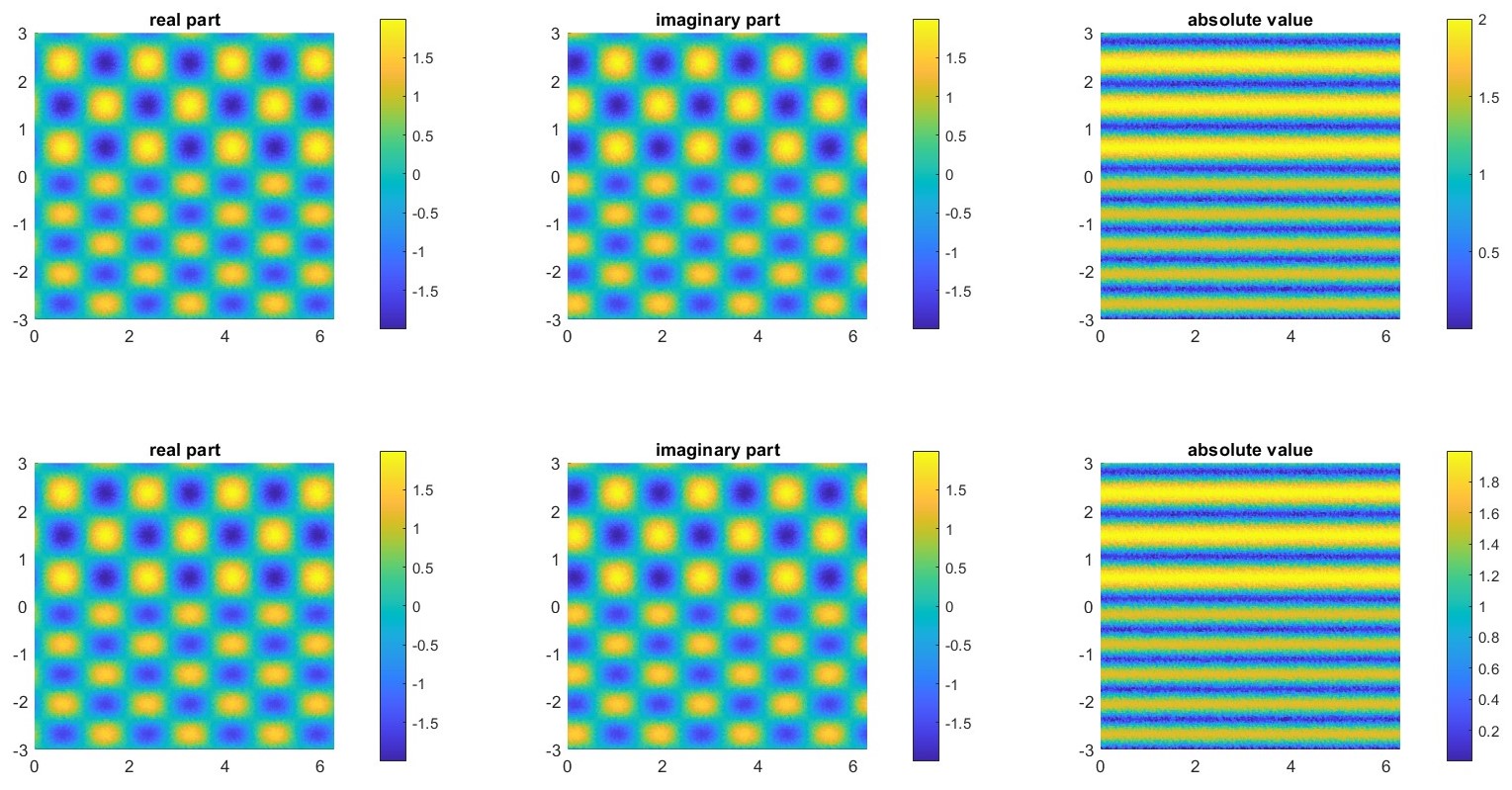}
    \caption{Plots of the numerical solution (top row) against the exact solution (bottom row)}
    \label{fig:PWDGimp}
\end{figure} 
\begin{figure}[h]
    \centering
    \includegraphics[width=.65\textwidth]{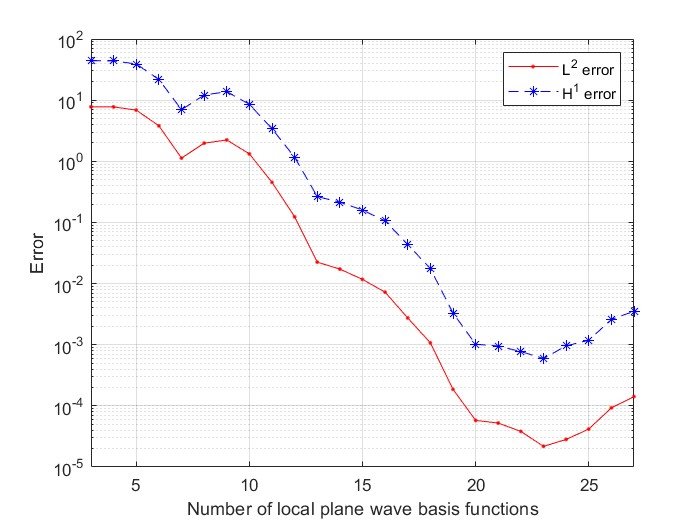}
    \caption{The errors in $L^2$ and $H^1$ norm for the different values of $p \in \{3,...,27\}$}
    \label{fig:PWDGimp_err}
\end{figure}
\begin{example} \label{Example_no_abs}
    To test our numerical method, we consider the simple case of $\Omega=[0,2\pi]\times [-H,H]$, divided in two different regions with different relative permittivity $\varepsilon_1, \varepsilon_2$ by the straight line $x_2=0$. We use (\ref{sol}) as exact solution to compute errors.
\end{example} \noindent
We first consider the case without absorption, with $\varepsilon_1=1$ and $\varepsilon_2=\frac{3}{2}$. The incident angle is $\theta = -\pi / 4$. In figures \ref{fig:PWDGimp} and \ref{fig:PWDGimp_err} we can see the results of the numerical experiments. We only vary the value of plane wave basis functions $p$ in every element $K$, while we keep the mesh size fixed. In this experiments we choose $h=1.5$, which gives us a triangulation of $56$ elements.

From the error plot we observe an exponential convergence in $p$, as we had in the case with $\varepsilon$ constant in $\Omega$.

\newpage
\begin{figure}[h] 
    \centering
    \includegraphics[width=.85\textwidth]{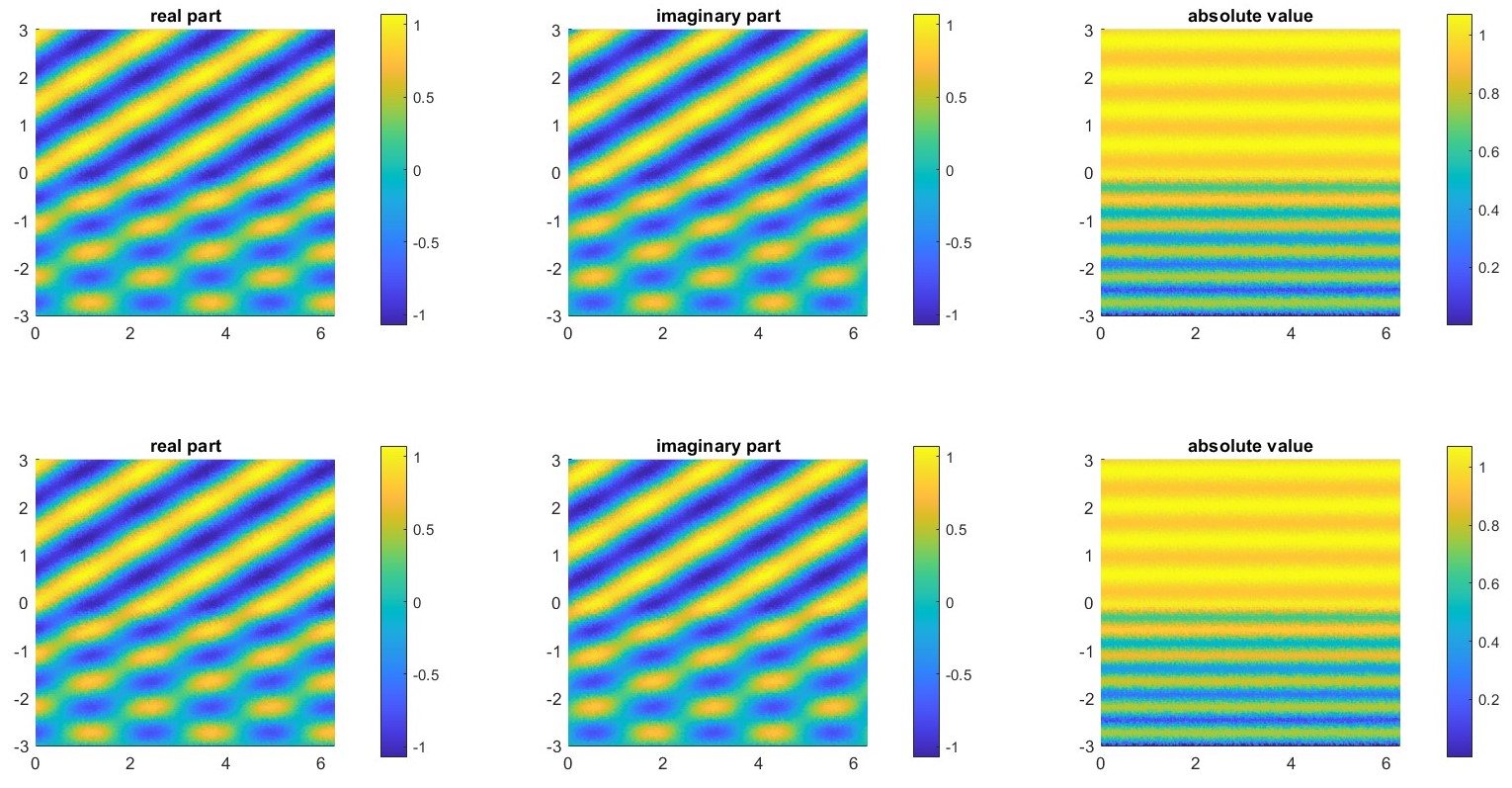}
    \caption{Plots of the numerical solution (top row) against the exact solution (bottom row)}
    \label{fig:PWDGimp_3}
\end{figure}
\begin{figure}[h] 
    \centering
    \includegraphics[width=.65\textwidth]{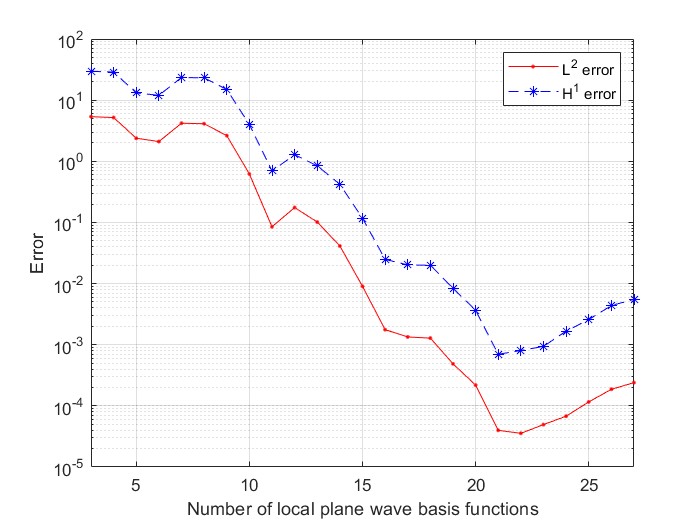}
    \caption{The errors in $L^2$ and $H^1$ norm for the different values of $p \in \{3,...,27\}$}
    \label{fig:PWDGimp_err_3}
\end{figure} 
\begin{example} \label{Esempio_cit}
    Another case we can consider is an incident plane wave with angle $\theta=- \pi / 3$ and a region with relative permittivities $\varepsilon_1 =1$ and $\varepsilon_2=(1.27+0.05i)^2$. \\
In this case we have absorption and the absolute value of the solution gradually decreases in the lower region, as we can see from figure \ref{fig:PWDGimp_3}. In figure $\ref{fig:PWDGimp_err_3}$ we can see the $L^2$ and $H^1$ error varying the number of plane waves $p$.
\end{example}  

\chapter{The DtN–PWDG Method} \label{ch4}

In this Chapter we present the DtN-PWDG method, which will be used to derive a numerical solution of the periodic grating problem. We refer to the formulation in \cite{CIVILETTI2020112478} and \cite{ZHENG2013550} for the model problem, as presented before in Chapter \ref{chap2}. We then derive the numerical method, following the idea in \cite{KAPITA2018208}. Last, we perform some numerical experiments to assert the accuracy of the DtN-PWDG scheme.

\section{Formulation of the numerical method} \label{sec_tdg}
Let $\Omega = \{x \in \mathbb{R}^2 : 0 < x_1 < 2\pi, -H < x_2 < H \},$ where $H > 0$ is an arbitrary positive constant, the domain considered before in Section \ref{sec2}. We denote with $k>0$ the wavenumber, which is assumed to be constant in $\Omega$, and with $\varepsilon \in \mathbb{C}$ the relative permittivity, which is supposed to be piecewise constant. We define $\kappa$ as $\kappa=k\sqrt{\varepsilon}$.

We want to find the numerical solution of the following boundary value problem: given an incident wave $u^i = e^{ik(x_1\cos \theta + x_2\sin \theta)} = e^{i\alpha_0 x_1 + i \beta_0 x_2}$, with $\theta \in \left(-\pi, 0 \right)$, find $u \in H^1_{\alpha_0}(\Omega)$ such that
\begin{equation} \label{nontrunc}
    \begin{cases}
        \Delta u + \kappa^2 u = 0 & \hspace{0.5cm} \text{in} \hspace{0.1cm} \Omega, \\
        u = 0 & \hspace{0.5cm} \text{on} \hspace{0.1cm} \Gamma_{D}, \\
        \frac{\partial(u-u^i)}{\partial \nu} - T(u-u^i) = 0 & \hspace{0.5cm} \text{on} \hspace{0.1cm} \Gamma_{H},
    \end{cases}
\end{equation}
where $\nu$ is the outer normal to $\Gamma_H$ and $\alpha_0 = k \cos \theta$ is the quasi-periodicity parameter. Here $T$ is the Dirichlet to Neumann operator defined in (\ref{DtN}).

In practical computations, we will solve the following truncated boundary value problem: find a quasi-periodic solution $u^M \in H^1_{\alpha_0}(\Omega)$ such that
\begin{equation} \label{truncated}
    \begin{cases}
        \Delta u^M + \kappa^2 u^M = 0 & \hspace{0.5cm} \text{in} \hspace{0.1cm} \Omega, \\
        u^M = 0 & \hspace{0.5cm} \text{on} \hspace{0.1cm} \Gamma_{D}, \\
        \frac{\partial(u^M-u^i)}{\partial \nu} - T_M(u^M-u^i) = 0 & \hspace{0.5cm} \text{on} \hspace{0.1cm} \Gamma_{H},
    \end{cases}
\end{equation}
where $T_M$ is the truncated DtN operator of order $M$ defined in (\ref{DtN_trunc}).
To better define the boundary conditions, we observe that, on $\Gamma_H$,
\begin{equation}
    \frac{\partial u}{\partial \nu} - T(u) = \frac{\partial(u^i)}{\partial \nu} - T(u^i) +\frac{\partial(u^s)}{\partial \nu} - T(u^s) = 2i\beta_0 u^i,
\end{equation}
since $\frac{\partial u^s}{\partial \nu} - T(u^s) = 0$. Here $\beta_0 = k \sin \theta$. We also have that
\begin{equation}
    \frac{\partial u}{\partial \nu} = T(u) + 2i\beta_0 u^i.
\end{equation}

We now present the formulation of the DtN-PWDG method, following the idea in \cite{KAPITA2018208}. Here, we refer to a boundary problem on an periodic scatterer with a DtN boundary condition on a segment, while in \cite{KAPITA2018208} the scatterer is bounded and the domain is a circle, so the DtN operator is different. Moreover, in our formulation we solve the problem for the total field and in \cite{KAPITA2018208} the problem is solved for the scattered one.

Let $\bigtau_h = \{ K \}$ a finite element partition of $\Omega$, suited to the conformation of the domain; we write $\mathbf{n}_K$ for the outward-pointing unit normal vector on $\partial K$, and $h$ for the mesh width of $\bigtau_h$. The parameter $h$ represents the diameter of the largest element in $\bigtau_h$, so that $h = \max_{K\in\bigtau_h}h_K$ where $h_K$ is the diameter of the smallest circumscribed circle containing $K$. We denote by $ F_h = \bigcup_{K \in \bigtau_h } \partial K $ and $F_h^I = F_h \setminus \partial \Omega $ the skeleton of the mesh and its inner part.

Given two elements of the partition $K_1, K_2$, a piecewise-smooth function $v$ and vector field $\tau$ on $\bigtau_h$, we introduce  on $\partial K_1 \cap \partial K_2$ the averages and the normal jumps as defined in (\ref{jumps}).\\
For $s>0$, we consider the usual broken Sobolev space $H^s(\bigtau_h)$ and the Trefftz space $T(\bigtau_h)$ as defined in (\ref{broken_sob}), with the only difference being the value of $\kappa= \kappa (\mathbf{x}) = k\sqrt{\varepsilon(\mathbf{x})}$ which is not constant.

As in Section \ref{dg}, we can derive the general Treffz DG formulation for problem (\ref{truncated}). Replacing $u_h$ and $\nabla u_h$ by the numerical fluxes $\hat{u}_h$ and $-\widehat{i\kappa\sigma}_h$, that are single valued approximations of $u_h$ and $\nabla u_h$ respectively on each edge, we get the following formulation.
\begin{equation} \label{TDG}
    \int_{\partial K} \hat{u}_h\overline{\nabla v_h \cdot \mathbf{n}} \hspace{0.1cm} dS + \int_{\partial K} \widehat{i\kappa\sigma}_h \cdot \mathbf{n} \overline{v_h}  \hspace{0.1cm} dS = 0
\end{equation}

We now define the PWDG fluxes. The definition of the fluxes on interior edges and edges on the scatterer are taken to be those of standard PWDG methods in \cite{Survey} and \cite{PWDG} that we presented in (\ref{flux_interno}) and (\ref{flux_dirichlet}). On the artificial boundary $\Gamma_H$ we propose new numerical fluxes, following the idea in \cite{KAPITA2018208}: in particular, we adapt the formulation in \cite{KAPITA2018208} from a DtN on a circle for a bounded scatterer to the formulation on $\Gamma_H$ for an infinite and periodic scatterer.

On interior edges in $F_h^I$ we choose
\begin{equation}
    \begin{cases}
        \hat{u}_h = \{ \hspace{-0.1cm} \{  u_h \} \hspace{-0.1cm} \}  + \beta \frac{1}{i\xi} \llbracket \nabla_h u_h  \rrbracket_N, & \\
        \widehat{i\kappa\sigma}_h = - \{ \hspace{-0.1cm} \{ \nabla_h u_h \} \hspace{-0.1cm} \}  - i \xi \alpha \hspace{0.05cm} \llbracket u_h  \rrbracket_N, & 
    \end{cases}
\end{equation}
where $\xi$ is defined as in (\ref{xi})
\begin{equation} \label{xi_2}
    \xi = \frac{1}{2} \biggl[ \Re (\kappa_1) + \Re (\kappa_2) \biggl].
\end{equation}

On the Dirichlet boundary we opt for
\begin{equation}
    \begin{cases}
        \hat{u}_h = 0, & \\
        \widehat{i\kappa \sigma}_h = - \nabla_h u_h - i \kappa \alpha \hspace{0.05cm} u_h \hspace{0.05cm} \mathbf{n}. &  
    \end{cases}
\end{equation}

On $\Gamma_H$ we propose
\begin{equation}
    \begin{cases}
        \hat{u}_h = u_h + \frac{\delta}{i\kappa} \left( \nabla_h  u_h \cdot \mathbf{n} - T_M u_h - 2i\beta_0 u^i \right), & \\
        \widehat{i \kappa \sigma}_h = - \hspace{0.05cm} T_M u_h \mathbf{n} - 2i\beta_0 u^i \mathbf{n} - \frac{\delta}{i \kappa} T^*_M \left( \nabla_h  u_h -T_M u_h \mathbf{n} - 2i\beta_0 u^i \mathbf{n} \right), &
    \end{cases}
\end{equation}
where $\alpha, \beta, \delta > 0$ are positive flux coefficients defined on the edges of the mesh, and $T^*_M$ is the $L^2(\Gamma_H)$-adjoint of $T_M$, defined as
\begin{equation}
    \int_{\Gamma_H} T^*_M v \hspace{0.05cm} \overline{w} \hspace{0.1cm} ds= \int_{\Gamma_H} v \hspace{0.05cm} \overline{T_Mw} \hspace{0.1cm} ds.
\end{equation}

We choose the fluxes in this way to get the consistency property of the method, as it is proven in Proposition \ref{prop:consistent}, and the coercivity of the sesquilinear form as we see in Proposition \ref{prop:norma}.

Now we add (\ref{TDG}) over all elements $K \in \bigtau_h$ and obtain:
\begin{equation*}
    \int_{F_h^I} \left( \hat{u}_h \llbracket \overline{\nabla_h v_h} \rrbracket_N +  \widehat{i\kappa\sigma}_h \cdot \llbracket \overline{v_h} \rrbracket_N \right) dS + \int_{F_h^B}  \left( \hat{u}_h\overline{\nabla_h v_h \cdot \mathbf{n}} + \widehat{i\kappa\sigma}_h \cdot \mathbf{n}\overline{v_h}\right) dS = 0.
\end{equation*}
Substituting the numerical fluxes, we get the following TDG scheme: Find $u_h \in V_h(\bigtau_h)$ such that for all $v_h \in V_h(\bigtau_h)$
\begin{equation}
    A^M_h(u_h, v_h) =L^M_h(v_h),
\end{equation}
where
\begin{align} \label{eqn:bilinear}
    & A^M_h(u, v) :=\\ \nonumber
    & \int_{F_h^I} \left( \{ \hspace{-0.1cm} \{ u \} \hspace{-0.1cm} \} \llbracket \overline{\nabla_h v} \rrbracket_N  - i\xi^{-1} \beta \hspace{0.05cm} \llbracket \nabla_h u\rrbracket_N  \hspace{0.05cm} \llbracket \overline{\nabla_h v} \rrbracket_N  -  \{ \hspace{-0.1cm} \{ \nabla_h u \} \hspace{-0.1cm} \} \cdot \llbracket \overline{v} \rrbracket_N  - i\xi \hspace{0.05cm}  \alpha \hspace{0.05cm} \llbracket  u\rrbracket_N \cdot \llbracket \overline{v} \rrbracket_N  \right) dS \\ \nonumber
    & + \int_{\Gamma_H} \biggl(  u \hspace{0.05cm} \overline{\nabla_h v \cdot \mathbf{n}}  - T_M u \hspace{0.05cm} \overline{v} -\delta \hspace{0.05cm} i\kappa^{-1} \left(\nabla_h u \cdot \mathbf{n} -T_M u\right) \hspace{0.05cm} \overline{\left(\nabla_h v \cdot \mathbf{n} -T_M v\right)} \biggl) dS \\ \nonumber
    & + \int_{\Gamma_D} \left( -\nabla_h u \cdot \mathbf{n} \hspace{0.05cm} \overline{v} - i\kappa \hspace{0.05cm}\alpha \hspace{0.05cm} u\hspace{0.05cm} \overline{v} \right) dS,
\end{align}
and
\begin{equation}
    L^M_h(v) :=  \int_{\Gamma_H} 2i\beta_0 \hspace{0.05cm} u^i \left( \overline{v} - \delta \hspace{0.05cm} i\kappa^{-1} \hspace{0.05cm} \overline{ (\nabla_h v \cdot \mathbf{n} -T_M v )} \hspace{0.05cm} \right) dS.
\end{equation} 

\subsection{Matrix formulation}
We already introduced the plane wave space on $K$
\begin{align*}
    V_p(K) & = \{ \hspace{0.1cm} v \in L^2(K) :  v(\mathbf{x}) = \sum_{j=1}^p \alpha_j \exp\{i\kappa \mathbf{d_j} \cdot \mathbf{x}\}, \hspace{0.3cm} \alpha_j \in \mathbb{C} \hspace{0.1cm} \},
\end{align*}
where $\{ \mathbf{d}_j \}_{j=1}^p$, with $|\mathbf{d_j}|=1$ are different directions and $\kappa=k\sqrt{\varepsilon}$ depends on the region where the element $K$ is located. Since we choose a mesh that is conforming to the domain profile, $\kappa$ is constant inside any element of the mesh.

We also introduce the global space
\begin{equation*}
    V_p(\bigtau_h)=\{ \hspace{0.1cm} v \in L^2(\Omega) : v_{|K} \in V_p(K), \hspace{0.2cm}\forall K \in \bigtau_h \hspace{0.1cm} \}.
\end{equation*}
The PWDG formulation is: find $u_h \in V_p(\bigtau_h)$ such that, for all $v_h \in V_p(\bigtau_h)$,
\begin{equation} \label{eqn:matrix}
    A_h^M(u_h, v_h) =L_h^M(v_h).
\end{equation}
Consider $p\in\mathbb{N}$ fixed; since $u_h,v_h \in V_p(\bigtau_h)$, on $K$ we can write 
\begin{align*}
   & u_h(\mathbf{x}) = \sum_{l=1}^p u_{l}^K \varphi_{l}^K(\mathbf{x}), \hspace{0.7cm} v_h(\mathbf{x}) = \sum_{j=1}^p v_{j}^K \varphi_{j}^K(\mathbf{x}),
\end{align*}
where $u_{j}^K,v_{j}^K \in \mathbb{C}$, $\hspace{0.2cm} \forall j=1,...,p$, $\hspace{0.2cm} \forall K\in \bigtau_h$ and 
\begin{equation*}
    \varphi_{j}^K(\mathbf{x})=\begin{cases}
    \exp\{i\kappa\mathbf{d_j} \cdot \mathbf{x}\}, & \mathbf{x} \in K \\
    0 & \mathbf{x} \not \in K
\end{cases}, \hspace{0.4cm} \forall j=1,...,p, \hspace{0.2cm} \forall K\in \bigtau_h.
\end{equation*}
If we write $u_p = \sum_{j=1}^N u_j\varphi_j$, where $N$ is the total number of degrees of freedom, the problem (\ref{eqn:matrix}) can be written as a linear system $A\mathbf{u}=\mathbf{L}$, where
\begin{align*}
    \mathbf{u}_l = u_l, & & A_{j,l} = A_h^M(\varphi_l, \varphi_j), & & \mathbf{L} = L_h^M(\varphi_l).
\end{align*}

As in Section \ref{matr_form}, we can distinguish different cases, the first being the case of $\varphi_l, \varphi_j$ on the same internal element $K$, which doesn't change from what we've seen before:
\begin{align*}
    A_{j,l} & = \sum_{F\in\mathcal{F}}  \left[ -\frac{1}{2}i\kappa(\mathbf{d}_j\cdot\mathbf{n}+\mathbf{d}_l\cdot\mathbf{n})-\beta i \xi^{-1} \kappa^2 \mathbf{d}_l\cdot\mathbf{n}\hspace{0.05cm}\mathbf{d}_j\cdot\mathbf{n} -\alpha i \xi \hspace{0.05cm} \right] \int_{F} e^{ i\kappa \mathbf{x}\cdot(\mathbf{d}_l-\mathbf{d}_j)} \hspace{0.05cm} dS,
\end{align*}
where $\mathcal{F}$ is the set of faces of $K$.

In this case, since we are referring to the same element, the value of $\varepsilon$ does not change from the two basis functions and neither does $\kappa$, so we can consider it constant. We only have to be careful when we identify the region where the element is located to choose the correct value of $\kappa$.

A second case is when $\varphi_l, \varphi_j$ are associated to the same element $K$, but some of its faces are part of Dirichlet boundary. In this case the matrix entry is:
\begin{align*}
    A_{j,l} &  = \sum_{F\in \mathcal{F}^I} \left[ -\frac{1}{2}i\kappa(\mathbf{d}_j\cdot\mathbf{n}+\mathbf{d}_l\cdot\mathbf{n})-\beta i \xi^{-1} \kappa^2 \mathbf{d}_l\cdot\mathbf{n}\hspace{0.05cm}\mathbf{d}_j\cdot\mathbf{n} -\alpha i \xi \hspace{0.05cm} \right] \int_{F} e^{ i\kappa\mathbf{x}\cdot(\mathbf{d}_l-\mathbf{d}_j)} \hspace{0.05cm} dS\\
    & + \sum_{F\in \mathcal{F}^D}  i\kappa\hspace{0.1cm}\left[ \hspace{0.05cm}- \alpha - \mathbf{d}_l \cdot \mathbf{n} \hspace{0.1cm}\right] \int_{F} e^{ i\kappa\mathbf{x}\cdot(\mathbf{d}_l-\mathbf{d}_j)} \hspace{0.05cm} dS,
\end{align*}
where $\mathcal{F}^I = \mathcal{F} \cap F_h^I$ and $\mathcal{F}^D = \mathcal{F} \cap \Gamma_D$.

Third case is when $\varphi_l, \varphi_j$ are associated to two adjacent elements $K, K'$. We have to integrate only over the face $F$ common to the two elements. When $\varphi_l, \varphi_j$ are associated to two adjacent elements $K, K'$ and they have respectively $\kappa_1=k\sqrt{\varepsilon_1}$ and $\kappa_2=k\sqrt{\varepsilon_2}$ as wave numbers, the matrix entry becomes
\begin{align*}
    A_{j,l} & = \left[\frac{1}{2}(i\kappa_2\mathbf{d}_j\cdot\mathbf{n}+ik_1\mathbf{d}_l\cdot\mathbf{n}) - \beta i \xi^{-1} \kappa_1 \kappa_2\mathbf{d}_l\cdot\mathbf{n}\hspace{0.05cm}\mathbf{d}_j\cdot\mathbf{n} -\alpha i\xi  \hspace{0.05cm}\right] \int_{F} e^{ i\kappa_1\mathbf{x}\cdot\mathbf{d}_l - i\kappa_2\mathbf{x}\cdot\mathbf{d}_j},
\end{align*}
with $\xi$ as in (\ref{xi_2}). In general we can also have $\kappa_1 = \kappa_2$. All the matrix entries can be computed with a closed formula similar to the one seen before.

\begin{figure}[]
    \centering
    \includegraphics[width=.75\textwidth]{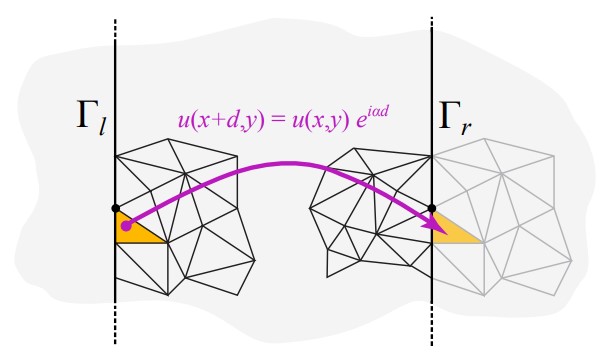}
    \caption{(Figure 5.2, \cite{ch5}) Quasi-periodicity of $u$ and periodicity of the mesh.}
    \label{fig:periodicity}
\end{figure}
We also have to include the quasi-periodicity in our linear system. The mesh is on the domain $\Omega = [0, 2\pi]\times[-H,H]$, which is obtained as the vertical strip of an unbounded domain which is $2\pi$-periodic. Since we have to include an interaction term between the left and the right boundary, we have to ensure that the mesh is compatible. To do so, we require that the elements lying on the left and the right boundary share the $x_2$-coordinates of the faces endpoints. The idea is that if we repeat the same mesh translated by $2\pi$ we get a conforming mesh. Moreover, since we have to imagine the domain as a cylinder, we will include all the faces $F$ on the left and right boundary of $\Omega$ among the internal faces, so we are going to sum the corresponding term to the matrix entries.\\
Moreover, we have to include the interaction terms between basis function associated to elements respectively on the right and left side of the domain, which are adjacent in our cylinder. In this interaction we have to consider the quasi-periodicity of the basis functions, in the sense that we have to multiply the function associated to the left side by the quasi-periodicity constant $e^{i\alpha_0 2\pi}$. In Figure \ref{fig:periodicity} we can see an example of the periodic mesh and how basis functions interact with each other.

For the faces located on $\Gamma_H$, we have two different types of contributions, the first being the ones which do not include terms with the DtN operator $T_M$, which are local contributions, while the terms including the DtN operator are global, since the function $T_M \varphi_l$ is defined on all $\Gamma_H$.

For an element with a face lying on $\Gamma_H$, we have to add to the matrix entry also the contribute derived by the integrals on $\Gamma_H$, which is the following.
\begin{align} \nonumber
& \int_{\Gamma_H} \biggl(  u \hspace{0.05cm} \overline{\nabla_h v \cdot \mathbf{n}}  - T_M u \hspace{0.05cm} \overline{v} - \delta \hspace{0.05cm} i\kappa^{-1} \hspace{0.05cm} \left(\nabla_h u \cdot \mathbf{n} -T_M u\right) \hspace{0.05cm} \overline{\left(\nabla_h v \cdot \mathbf{n} -T_M v\right)} \biggl) \hspace{0.05cm} dS \\ \label{local}
& = \int_{\Gamma_H}  \left( u -\delta \hspace{0.05cm} i\kappa^{-1} \hspace{0.05cm}\nabla_h u \cdot \mathbf{n} \right) \hspace{0.05cm} \overline{\nabla_h v \cdot \mathbf{n}} \hspace{0.1cm} dS \\ \label{global}
& - \int_{\Gamma_H} \biggl( T_M u \hspace{0.05cm} \left( \overline{v} - \delta \hspace{0.05cm} i\kappa^{-1} \hspace{0.05cm}\overline{\nabla_h v \cdot \mathbf{n}} \right) + \delta \hspace{0.05cm} i\kappa^{-1} \hspace{0.05cm} \left( T_Mu \hspace{0.05cm} \overline{T_Mv} - \nabla_h u \cdot \mathbf{n} \hspace{0.05cm}\overline{T_M v} \right) \biggl) \hspace{0.05cm} dS.
\end{align}

We can observe that the term (\ref{local}) is a local term, since it is non-zero only for basis functions associated to the same element, while the term (\ref{global}) is a global term, because it is non-zero for every couple of basis functions associated to any element with a face on $\Gamma_H$. 

The local term (\ref{local}) is non-zero only if the two basis functions $\varphi_l$ and $\varphi_j$ are associated to the same upper boundary element $K$. In this case the local term of the matrix entry is
\begin{align*}
    & \int_{F}  \left( \varphi_l - \delta \hspace{0.05cm} i\kappa^{-1} \hspace{0.05cm} \nabla_h \varphi_l \cdot \mathbf{n} \right) \hspace{0.05cm} \overline{\nabla_h \varphi_j \cdot \mathbf{n}} \hspace{0.1cm} dS\\ 
    & = \int_{F}  \left( e^{i\kappa\mathbf{x}\cdot\mathbf{d}_l}(1 + \delta \hspace{0.05cm} \mathbf{d}_l \cdot \mathbf{n}) \right) \hspace{0.05cm} \left( -i\kappa \hspace{0.05cm} \mathbf{d}_j \cdot \mathbf{n} \hspace{0.05cm} e^{-i\kappa\mathbf{x}\cdot\mathbf{d}_j} \right) \hspace{0.05cm} dS \\
    & = -i\kappa \hspace{0.05cm} \mathbf{d}_j \cdot \mathbf{n} \hspace{0.05cm} (1 + \delta \hspace{0.05cm} \mathbf{d}_l \cdot \mathbf{n}) \hspace{0.05cm}  \int_{F} e^{i\kappa\mathbf{x}\cdot(\mathbf{d}_l-\mathbf{d}_j)} \hspace{0.05cm} dS,
\end{align*}
where $F$ is the face of $K$ lying on $\Gamma_H$ and $\mathbf{n}$ is the outer normal of $K$.

The global term (\ref{global}) is non-zero for every couple of basis functions associated to any element $K$ with a face on the boundary $\Gamma_H$. This is because the function $T_M \varphi_l$ is defined on all $\Gamma_H$, while $\varphi_l$ is zero outside of $K$.

Given two basis function $\varphi_l$, $\varphi_j$ associated to any upper elements $K_1$, $K_2$, the global contribution is the sum of three terms (\ref{primo})-(\ref{terzo})
\begin{align} \nonumber
    & - \int_{\Gamma_H} T_M \varphi_l \hspace{0.05cm} \left( \overline{\varphi_j}-\delta \hspace{0.05cm} i\kappa^{-1} \hspace{0.05cm}\overline{\nabla_h \varphi_j \cdot \mathbf{n}} \right) dS  - \int_{\Gamma_H} \delta \hspace{0.05cm} i\kappa^{-1} \hspace{0.05cm} T_M\varphi_l \hspace{0.05cm} \overline{T_M\varphi_j} \hspace{0.05cm} dS  \\ \nonumber
    & + \int_{\Gamma_H} \delta \hspace{0.05cm} i\kappa^{-1} \hspace{0.05cm} \nabla_h \varphi_l \cdot \mathbf{n} \hspace{0.05cm} \overline{T_M \varphi_j} \hspace{0.05cm} dS\\ \label{primo}
    & =  -( 1 - \delta \hspace{0.05cm} \mathbf{d}_j \cdot \mathbf{n}) \int_{\Gamma_H} T_M \varphi_l \hspace{0.1cm}  \overline{\varphi_j} \hspace{0.05cm} dS \\ \label{secondo}
    & - \delta \hspace{0.05cm} i\kappa^{-1} \hspace{0.05cm} \int_{\Gamma_H} T_M\varphi_l \hspace{0.05cm} \overline{T_M\varphi_j} \hspace{0.05cm} dS \\ \label{terzo}
    & - \delta \hspace{0.05cm} \mathbf{d}_l \cdot \mathbf{n} \hspace{0.1cm} \int_{\Gamma_H} \varphi_l \hspace{0.05cm} \overline{T_M \varphi_j} \hspace{0.05cm} dS,
\end{align}
since $\nabla_h \varphi_l = i\kappa \mathbf{d}_l \hspace{0.05cm} \varphi_l \hspace{0.05cm}$.

In order to compute the three terms, we recall the definition of the truncated DtN operator (\ref{DtN_trunc})
\begin{align*} 
    & T_M:H^{1/2}_{\alpha_0}(\Gamma_H) \to H^{-1/2}_{\alpha_0}(\Gamma_H), \\
    & (T_M\phi)(x_1) = i \sum_{n=-M}^M \phi_n \beta_n e^{i\alpha_nx_1}, \hspace{0.5cm} \text{for} \hspace{0.2cm} \phi \in H^{1/2}(\Gamma_H),
\end{align*}
where, for $n=-M,...,M$, $\alpha_n = \alpha_0+n$,
\begin{equation*}
    \beta_n = \begin{cases} \sqrt{k^2\varepsilon^+-\alpha_n^2} & \alpha_n^2 < k^2\varepsilon^+, \\ i\sqrt{\alpha_n^2-k^2\varepsilon^+} & \alpha_n^2 > k^2\varepsilon^+ ,\end{cases} 
\end{equation*}
and 
\begin{equation*}
    \phi_n =\phi_n(H) = \frac{1}{2\pi} \int_0^{2 \pi} e^{-i\alpha_n x_1} \phi(x_1,H) \hspace{0.05cm} dx_1.
\end{equation*}
If we choose $\phi = \varphi_l = e^{i\kappa \mathbf{x}\cdot\mathbf{d}_l}$ on $K_1$, we have
\begin{align*}
    \varphi_l^n & = \frac{1}{2\pi} \int_0^{2 \pi} e^{-i\alpha_n x_1} \varphi_l(x_1,H) \hspace{0.05cm} dx_1 \\
    & = \frac{1}{2\pi} e^{i\kappa \hspace{0.05cm} (\mathbf{d}_l)_{_2} \hspace{0.05cm} H} \int_{p_1}^{p_2} e^{(i\kappa \hspace{0.05cm} (\mathbf{d}_l)_{_1} -i\alpha_n) \hspace{0.05cm} x_1} \hspace{0.05cm} dx_1,
\end{align*}
where $p_1$, $p_2$ are the $x_1$-coordinates of the endpoints of the face of $K_1$ which lies on $\Gamma_H$.

Once we have computed the Fourier coefficients $\varphi_l^n$, we can compute the integral in (\ref{primo})
\begin{align*}
    \int_{\Gamma_H} T_M \varphi_l \hspace{0.1cm}  \overline{\varphi_j} \hspace{0.05cm} dS & = \int_0^{2\pi} i \sum_{n=-M}^M \varphi_l^n \beta_n e^{i\alpha_nx_1} \hspace{0.1cm}  \overline{\varphi_j(x_1,H)} \hspace{0.05cm} dx_1 \\
    & = i \sum_{n=-M}^M \varphi_l^n \beta_n e^{-i\kappa \hspace{0.05cm} (\mathbf{d}_j)_{_2}\hspace{0.03cm} H} \int_{q_1}^{q_2} e^{ (i\alpha_n-i\kappa \hspace{0.05cm}(\mathbf{d}_j)_{_1}) \hspace{0.05cm} x_1} \hspace{0.05cm} dx_1,
\end{align*}
where $q_1$, $q_2$ are the $x_1$-coordinates of the endpoints of the face of $K_2$ which lies on $\Gamma_H$.

To compute the integral $\int_{\Gamma_H} \varphi_l \hspace{0.1cm}  \overline{T_M \varphi_j} \hspace{0.05cm} dS$ in (\ref{terzo}) we make use of the same definition of $T_M$
\begin{align*}
    \int_{\Gamma_H} \varphi_l \hspace{0.1cm}  \overline{T_M \varphi_j} \hspace{0.05cm} dS & = \int_0^{2\pi} \varphi_l(x_1,H) \hspace{0.1cm}  \overline{\left( i \sum_{n=-M}^M \varphi_j^n \beta_n e^{i\alpha_nx_1} \right) } \hspace{0.05cm} dx_1 \\
    & = -i \sum_{n=-M}^M \overline{\varphi_j^n \beta_n} e^{i\kappa \hspace{0.05cm} (\mathbf{d}_l)_{_2}\hspace{0.03cm} H} \int_{p_1}^{p_2} e^{ (i\kappa \hspace{0.05cm}(\mathbf{d}_l)_{_1}-i\alpha_n) \hspace{0.05cm} x_1} \hspace{0.05cm} dx_1,
\end{align*}
with $p_1$, $p_2$ defined as before but for $K_2$ instead of $K_1$.

Finally, we compute the $\int_{\Gamma_H} T_M\varphi_l \hspace{0.05cm} \overline{T_M\varphi_j} \hspace{0.05cm} dS$ in (\ref{secondo}).
\begin{align*}
    \int_{\Gamma_H} T_M\varphi_l \hspace{0.05cm} \overline{T_M\varphi_j} \hspace{0.05cm} dS & = \int_0^{2\pi} \left( i \sum_{n=-M}^M \varphi_l^n \beta_n e^{i\alpha_nx_1} \right) \hspace{0.1cm} \overline{\left( i \sum_{m=-M}^M \varphi_j^m \beta_m e^{i\alpha_m x_1} \right) } \hspace{0.05cm} dx_1 \\
    & = \sum_{n=-M}^M \sum_{m=-M}^M \varphi_l^n \beta_n \hspace{0.1cm} \overline{\varphi_j^m \beta_m} \int_0^{2\pi} e^{i(\alpha_n -\alpha_m) \hspace{0.03cm} x_1}  \hspace{0.05cm} dx_1 \\
    & = \sum_{n=-M}^M 2\pi \hspace{0.1cm} \varphi_l^n \beta_n \hspace{0.1cm} \overline{\varphi_j^n \beta_n} \\
    & = \sum_{n=-M}^M 2\pi \hspace{0.1cm} \varphi_l^n\hspace{0.1cm} \overline{\varphi_j^n} \hspace{0.1cm} | \beta_n |^2,
\end{align*}
since $\int_0^{2\pi} e^{i(\alpha_n -\alpha_m) \hspace{0.03cm} x_1}  \hspace{0.05cm} dx_1 = 2\pi \hspace{0.05cm} \delta_{n,m}$.\\

Now we focus on the right hand side of the linear system: we note that this term is nonzero only if the basis function is associated to an element $K$ with a face $F$ lying on $\Gamma_H$. We have
\begin{align} \nonumber
    \mathbf{L}_j & = L^M_h(\varphi_j) =  \int_{\Gamma_H} 2i\beta_0 \hspace{0.05cm} u^i \left( \overline{\varphi_j} -\delta \hspace{0.05cm} i\kappa^{-1} \hspace{0.05cm} \overline{\nabla_h \varphi_j \cdot \mathbf{n}} + \delta \hspace{0.05cm} i\kappa^{-1} \hspace{0.05cm} \overline{T_M \varphi_j} \right) \hspace{0.05cm} dS \\ \label{uno}
    & = 2i\beta_0 \hspace{0.05cm} (1 - \delta \hspace{0.05cm} \mathbf{d}_j \cdot \mathbf{n} ) \int_F u^i \hspace{0.05cm} \overline{\varphi_j} \hspace{0.1cm} dS \\ \label{due}
    & - 2 \hspace{0.05cm} \delta \hspace{0.05cm}\beta_0 \hspace{0.05cm} \kappa^{-1} \int_{\Gamma_H} u^i \hspace{0.05cm} \overline{T_M \varphi_j} \hspace{0.05cm} dS.
\end{align} 
Here we used that $\nabla_h \varphi_j = i\kappa \mathbf{d}_j \hspace{0.05cm} \varphi_j \hspace{0.05cm}$ and that the function $\varphi_j$ is zero outside of $K$, while $T_M \varphi_j$ is defined on all $\Gamma_H$.

The integral in (\ref{uno}) is easy to compute, as we can write $u^i(\mathbf{x})=e^{i\kappa^+ \mathbf{x} \cdot \mathbf{d}_i }$, with $\mathbf{d}_i = (\cos \theta, \sin \theta)^T$ and $\kappa^+=k\sqrt{\varepsilon^+}$. So we have
\begin{equation*}
    \int_F u^i \hspace{0.05cm} \overline{\varphi_j} \hspace{0.1cm} dS = \int_{F} e^{i\kappa^+\mathbf{x}\cdot\mathbf{d}_i-i\kappa \mathbf{x}\cdot\mathbf{d}_j} \hspace{0.05cm} dS.
\end{equation*}
Note that in general $\kappa \neq \kappa^+$, since the relative permittivity inside of $\Omega$ could be different from $\varepsilon^+$.

For the integral in (\ref{due}) we can write $u^{i} (x_1,x_2)=e^{i\alpha_0 x_1 + i \beta_0 x_2}$, so we have $u^i(x_1,H) = e^{i \beta_0 \hspace{0.03cm} H}e^{i\alpha_0 x_1}$. The integral $\int_{\Gamma_H} u^i \hspace{0.05cm} \overline{T_M \varphi_j} \hspace{0.05cm} dS$ becomes
\begin{align*}
    \int_{\Gamma_H} u^i \hspace{0.05cm} \overline{T_M \varphi_j} \hspace{0.05cm} dS & =  \int_{\Gamma_H} u^i(x_1,H) \hspace{0.05cm} \overline{\left( i \sum_{n=-M}^M \varphi_j^n \beta_n e^{i\alpha_n x_1} \right)} \hspace{0.05cm} dx_1 \\
    & = -i \sum_{n=-M}^M \overline{\varphi_j^n \beta_n} e^{i \beta_0 \hspace{0.03cm} H} \int_{\Gamma_H} e^{i(\alpha_0-\alpha_n) x_1} \hspace{0.05cm} dx_1 \\
    & = -i \hspace{0.05cm} \overline{\varphi_j^0 \beta_0} e^{i \beta_0 \hspace{0.03cm} H} \cdot 2\pi.
\end{align*}
Here we used that $\int_0^{2\pi} e^{i(\alpha_0 -\alpha_n) \hspace{0.03cm} x_1}  \hspace{0.05cm} dx_1 = 2\pi \hspace{0.05cm} \delta_{0,n}$.

Inserting these results into the expression of $\mathbf{L}_j$ we get
\begin{align*}
    \mathbf{L}_j & = 2i\beta_0 \hspace{0.05cm} (1-\delta \hspace{0.05cm} \mathbf{d}_j \cdot \mathbf{n} ) \int_{F} e^{ik^+\mathbf{x}\cdot\mathbf{d}_i-i\kappa \mathbf{x}\cdot\mathbf{d}_j} \hspace{0.05cm} dS + 4\pi \hspace{0.05cm} i \hspace{0.05cm} \delta \hspace{0.05cm}\beta_0 \hspace{0.05cm} \kappa^{-1} \hspace{0.05cm} \overline{\varphi_j^0 \beta_0} \hspace{0.1cm} e^{i \beta_0 \hspace{0.03cm} H} \\
    & = 2i\beta_0 \hspace{0.05cm} (1-\delta \hspace{0.05cm} \mathbf{d}_j \cdot \mathbf{n} ) \int_{F} e^{ik^+\mathbf{x}\cdot\mathbf{d}_i-i\kappa \mathbf{x}\cdot\mathbf{d}_j} \hspace{0.05cm} dS - 4\pi \hspace{0.05cm} i \hspace{0.05cm} \delta \hspace{0.05cm} |\beta_0|^2 \hspace{0.05cm} \kappa^{-1} \hspace{0.05cm} \overline{\varphi_j^0} \hspace{0.1cm} e^{i \beta_0 \hspace{0.03cm} H}.
\end{align*} 

\section{Error analysis}
We develop our analysis under the domain assumptions described in \cite{CIVILETTI2020112478} and the non trapping conditions (\ref{notrap}) for $\varepsilon$ as we did in Section \ref{regularity}. Under this assumptions, we can guarantee the existence and the uniqueness of the solution to the scattering problem (\ref{truncated}) and that this solution coincides with the one of (\ref{nontrunc}).

We define the bilinear form $A_h$, for which the exact DtN operator is used on $\Gamma_H$. This can be seen as a generalization of $A_h^M$ for the case $M=\infty$. 
\begin{align} \label{DtN_bilinear_exact}
    & A_h(u, v) :=\\ \nonumber
    & \int_{F_h^I} \left( \{ \hspace{-0.1cm} \{ u \} \hspace{-0.1cm} \} \llbracket \overline{\nabla_h v} \rrbracket_N  - i\xi^{-1} \beta \hspace{0.05cm} \llbracket \nabla_h u\rrbracket_N  \hspace{0.05cm} \llbracket \overline{\nabla_h v} \rrbracket_N  -  \{ \hspace{-0.1cm} \{ \nabla_h u \} \hspace{-0.1cm} \} \cdot \llbracket \overline{v} \rrbracket_N  - i\xi \hspace{0.05cm}  \alpha \hspace{0.05cm} \llbracket  u\rrbracket_N \cdot \llbracket \overline{v} \rrbracket_N  \right) dS \\ \nonumber
    & + \int_{\Gamma_H} \biggl(  u \hspace{0.05cm} \overline{\nabla_h v \cdot \mathbf{n}}  - T_M u \hspace{0.05cm} \overline{v} -\delta \hspace{0.05cm} i\kappa^{-1} \left(\nabla_h u \cdot \mathbf{n} -T u\right) \hspace{0.05cm} \overline{\left(\nabla_h v \cdot \mathbf{n} -T v\right)} \biggl) dS \\ \nonumber
    & + \int_{\Gamma_D} \left( -\nabla_h u \cdot \mathbf{n} \hspace{0.05cm} \overline{v} - i\kappa \hspace{0.05cm}\alpha \hspace{0.05cm} u\hspace{0.05cm} \overline{v} \right) dS,
\end{align}
while we define $L_h$ as
\begin{equation}
    L_h(v) :=  \int_{\Gamma_H} 2i\beta_0 \hspace{0.05cm} u^i \left( \overline{v} - \delta \hspace{0.05cm} i\kappa^{-1} \hspace{0.05cm} \overline{ (\nabla_h v \cdot \mathbf{n} -T v )} \hspace{0.05cm} \right) dS.
\end{equation} 
We refer as the non-truncated DtN-PWDG formulation as: Find $u_h \in V_h(\bigtau_h)$ such that for all $v_h \in V_h(\bigtau_h)$
\begin{equation}
    A_h(u_h, v_h) =L_h(v_h).
\end{equation} 
\begin{prop} \label{prop:consistent}
    The non-truncated DtN–PWDG method is consistent.
\end{prop}
\begin{proof}
    If $u$ solves the boundary value problem (\ref{nontrunc}), assumed $u \in H^2_{\alpha_0}(\Omega)$, we have that and on any interior edge  $\llbracket u \rrbracket_N = 0 $ and $\llbracket \nabla_h u \rrbracket_N = 0 $. Moreover, $\nabla_h u\cdot \mathbf{n} - T u = 2i\beta_0 u^i$ on $\Gamma_H$ and $u = 0$ on $\Gamma_D$. \\
    From the definition of the numerical fluxes $\hat{u}_h$ and $\widehat{i\kappa\sigma}_h$, we have that on all the faces, when we substitute $u$ in the expression of the fluxes, we have that $\hat{u}_h=u$ and $\widehat{i\kappa\sigma}_h=-\nabla_h u$. This gives us that, for every $v_h \in V_h(\bigtau_h)$,
    \begin{align*}
        A_h(u, v_h) = L_h(v_h),
    \end{align*}
    which is the consistency of the method.
\end{proof}

\begin{lem} \label{lemma:l2prod}
    We have that for every $w \in H^1_{\alpha_0} (\Omega)$, $w\neq 0$
    \begin{equation}
        \Im \int_{\Gamma_H} Tw \hspace{0.01cm} \overline{w} \hspace{0.1cm} dS \geq 0.
    \end{equation}
\end{lem}
\begin{proof}
    We recall that on $\Gamma_H$ we can write
    \begin{equation*}
        w = \sum_{n \in \mathbb{Z}} w_n e^{i \alpha_n x_1}, \hspace{0.7cm} Tw = i \sum_{n\in\mathbb{Z}} w_n \beta_n e^{i\alpha_n x_1},
    \end{equation*}
    so we have
    \begin{align*}
         \int_{\Gamma_H} Tw \hspace{0.01cm} \overline{w} \hspace{0.1cm} dS & = \int_{0}^{2\pi} i \sum_{n\in\mathbb{Z}} w_n \beta_n e^{i\alpha_nx_1} \overline{\left( \sum_{m \in \mathbb{Z}} w_m e^{i \alpha_m x_1} \right)} \hspace{0.1cm} dx_1 \\
         & = \sum_{n\in\mathbb{Z}} \sum_{m \in \mathbb{Z}} i w_n \beta_n \hspace{0.05cm} \overline{w_m} \int_{0}^{2\pi} e^{i(\alpha_n-\alpha_m) x_1} \hspace{0.1cm} dx_1 \\
         & = \sum_{n\in\mathbb{Z}} i |w_n|^2 \beta_n \hspace{0.05cm} 2\pi.
    \end{align*}
    Recalling that
    \begin{equation*}
        \beta_n = \begin{cases} \sqrt{k^2\varepsilon^+-\alpha_n^2} & \alpha_n^2 < k^2\varepsilon^+, \\ i\sqrt{\alpha_n^2-k^2\varepsilon^+} & \alpha_n^2 > k^2\varepsilon^+, \end{cases} 
    \end{equation*}
    we get
    \begin{equation*}
         \Im \int_{\Gamma_H} Tw \hspace{0.01cm} \overline{w} \hspace{0.1cm} dS = \sum_{\alpha^2_n < k^2\varepsilon^+} 2\pi \hspace{0.05cm} |w_n|^2 \sqrt{k^2\varepsilon^+-\alpha_n^2} \geq 0.
    \end{equation*}
\end{proof}
\begin{rem} \label{imag_rem}
    We proved this result for the full DtN operator $T$, while in the discrete problem we make use of the truncated operator $T_M$. Since the expression of $\Im \int_{\Gamma_H} Tw \hspace{0.01cm} \overline{w} \hspace{0.1cm} dS$ is a finite sum, by choosing $M$ sufficiently large, in particular $M$ such that $\alpha^2_M > k^2\varepsilon^+$, and assumed that $\Re (\varepsilon^+) >0$ and $\Im (\varepsilon^+)=0$, we can guarantee that $\Im \int_{\Gamma_H} Tw \hspace{0.01cm} \overline{w} \hspace{0.1cm} dS= \Im \int_{\Gamma_H} T_M w \hspace{0.01cm} \overline{w} \hspace{0.1cm} dS$. \\
\end{rem}

For the error analysis, it is useful to write the sesquilinear form (\ref{eqn:bilinear}) as
\begin{align} \label{eqn:bilinear_new}
    & A^M_h(u, v) = \int_\Omega \left( \nabla_h u \cdot \nabla_h \overline{v} -\kappa^2 u \overline{v} \right) \hspace{0.05cm} d\mathbf{x}  \\ \nonumber
    & \int_{F_h^I} \left( -i\xi \hspace{0.05cm}\alpha \hspace{0.05cm} \llbracket u \rrbracket_N \cdot \overline{\llbracket v \rrbracket}_N - \llbracket u \rrbracket_N \cdot \overline{\dlgraffa \nabla_h v \drgraffa} -\beta \hspace{0.05cm} i\xi^{-1} \hspace{0.05cm} \llbracket \nabla_h u \rrbracket_N \hspace{0.05cm} \overline{\llbracket \nabla_h v \rrbracket}_N - \dlgraffa \nabla_h u \drgraffa \cdot \overline{\llbracket v \rrbracket}_N \right) dS \\ \nonumber
    & + \int_{\Gamma_H} \biggl( - T_M u \hspace{0.05cm} \overline{v} -\delta \hspace{0.05cm} i\kappa^{-1} \hspace{0.05cm} \left(\nabla_h u \cdot \mathbf{n} -T_M u\right) \hspace{0.05cm} \overline{\left(\nabla_h v \cdot \mathbf{n} -T_M v\right)} \biggl) dS \\
    & + \int_{\Gamma_D} \left( -i\kappa \hspace{0.05cm}\alpha \hspace{0.05cm} u\hspace{0.05cm} \overline{v} + u \overline{\nabla_h v \cdot \mathbf{n}}  -\nabla_h u \cdot \mathbf{n} \hspace{0.05cm} \overline{v} \right) dS, \nonumber
\end{align}
and
\begin{align} \label{eqn:bilinear_exact}
    & A_h(u, v) = \int_\Omega \left( \nabla_h u \cdot \nabla_h \overline{v} -\kappa^2 u \overline{v} \right) \hspace{0.05cm} d\mathbf{x}  \\ \nonumber
    & \int_{F_h^I} \left( -i\xi \hspace{0.05cm}\alpha \hspace{0.05cm} \llbracket u \rrbracket_N \cdot \overline{\llbracket v \rrbracket}_N - \llbracket u \rrbracket_N \cdot \overline{\dlgraffa \nabla_h v \drgraffa} -\beta \hspace{0.05cm} i\xi^{-1} \hspace{0.05cm} \llbracket \nabla_h u \rrbracket_N \hspace{0.05cm} \overline{\llbracket \nabla_h v \rrbracket}_N - \dlgraffa \nabla_h u \drgraffa \cdot \overline{\llbracket v \rrbracket}_N \right) dS \\ \nonumber
    & + \int_{\Gamma_H} \biggl( - T u \hspace{0.05cm} \overline{v} -\delta \hspace{0.05cm} i\kappa^{-1} \hspace{0.05cm} \left(\nabla_h u \cdot \mathbf{n} -T u\right) \hspace{0.05cm} \overline{\left(\nabla_h v \cdot \mathbf{n} -T v\right)} \biggl) dS \\
    & + \int_{\Gamma_D} \left( -i\kappa \hspace{0.05cm}\alpha \hspace{0.05cm} u\hspace{0.05cm} \overline{v} + u \overline{\nabla_h v \cdot \mathbf{n}}  -\nabla_h u \cdot \mathbf{n} \hspace{0.05cm} \overline{v} \right) dS. \nonumber
\end{align}

We now define two mesh-dependent seminorms on the Trefftz space $T(\bigtau_h)$; we refer to (\ref{broken_sob}) for the definition of the space.
\begin{align} \nonumber
    \vertiii{w}_{\text{TDG}_T}^2 & := \left\| \xi^{-\frac{1}{2}} \beta^{\frac{1}{2}} \llbracket \nabla_h w \rrbracket_N \right\|^2_{L^2(F_h^I)} + \left\| \xi^{\frac{1}{2}} \alpha^{\frac{1}{2}} \llbracket w \rrbracket_N \right\|^2_{L^2(F_h^I)} \\
    & + \left\| \kappa^{\frac{1}{2}} \alpha^{\frac{1}{2}} w \right\|^2_{L^2(\Gamma_D)} + \left\| \kappa^{-\frac{1}{2}} \delta^{\frac{1}{2}} ( \nabla_h w \cdot \mathbf{n} - T w ) \right\|^2_{L^2(\Gamma_H)},
\end{align}
and
\begin{align} \nonumber
    \vertiii{w}_{\text{TDG}_T^+}^2 & := \vertiii{w}_{\text{TDG}_T}^2 + \left\| \xi^{\frac{1}{2}} \beta^{-\frac{1}{2}} \dlgraffa w \drgraffa \right\|^2_{L^2(F_h^I)} +\left\| \xi^{-\frac{1}{2}} \alpha^{-\frac{1}{2}} \dlgraffa \nabla_h w \drgraffa \right\|^2_{L^2(F_h^I)} \\
    & + \left\| \kappa^{-\frac{1}{2}} \alpha^{-\frac{1}{2}} \nabla_h w \cdot \mathbf{n} \right\|^2_{L^2(\Gamma_D)} +  \left\| \kappa^{\frac{1}{2}} \delta^{-\frac{1}{2}} w \right\|^2_{L^2(\Gamma_H)}.
\end{align}

\begin{prop} \label{prop:norma}
    The seminorm $\vertiii{\cdot}_{\text{TDG}_T}$ is actually a norm on $T(\bigtau_h)$ and, provided $M$ is sufficiently large and the non-trapping conditions (\ref{notrap}) on $\varepsilon$ are satisfied, we have
    \begin{equation} \label{result}
        -\Im A^M_h(w, w) \geq \vertiii{w}_{\text{TDG}_T}^2.
    \end{equation}
    Moreover,
    \begin{equation}
        A^M_h(v, w) \leq 2 \vertiii{v}_{\text{TDG}_T^+} \hspace{0.1cm} \vertiii{w}_{\text{TDG}_T}.
    \end{equation}
\end{prop}
\begin{proof}
    From the expression of (\ref{eqn:bilinear_exact}), we have
    \begin{align*}
        -\Im A_h(w, w) & = \left\| \xi^{-\frac{1}{2}} \beta^{\frac{1}{2}} \llbracket \nabla_h w \rrbracket_N \right\|^2_{L^2(F_h^I)} + \left\| \xi^{\frac{1}{2}}\alpha^{\frac{1}{2}} \llbracket w \rrbracket_N \right\|^2_{L^2(F_h^I)} +\left\| \kappa^{\frac{1}{2}} \alpha^{\frac{1}{2}} w \right\|^2_{L^2(\Gamma_D)}\\
        & + \Im \int_{\Gamma_H} T w \hspace{0.05cm} \overline{w} \hspace{0.1cm} dS + \left\| \kappa^{-\frac{1}{2}} \delta^{\frac{1}{2}} ( \nabla_h w \cdot \mathbf{n} - T w ) \right\|^2_{L^2(\Gamma_H)}\\
        & \geq \vertiii{w}^2_{\text{TDG}_T},
    \end{align*}
    since $\Im \int_{\Gamma_H} T w \hspace{0.05cm} \overline{w} \hspace{0.1cm} dS \geq 0$ from Lemma \ref{lemma:l2prod}. Moreover, if $\Im A_h(w, w)= 0$, then $w \in H^2(\Omega)$ satisfies the Helmholtz equation in $\Omega$, with $w = 0$ on $\Gamma_D$ and $\nabla_h w \cdot \mathbf{n} - T w = 0$ on $\Gamma_H$. This problem has only the trivial solution $w = 0$ since the non-trapping conditions on $\varepsilon$ are satisfied. This proves that $\vertiii{\cdot}_{\text{TDG}_T}$ is a norm.

    From Remark \ref{imag_rem} we get that if $M$ is sufficiently large,
    \begin{equation*}
        \Im \int_{\Gamma_H} Tw \hspace{0.01cm} \overline{w} \hspace{0.1cm} dS= \Im \int_{\Gamma_H} T_M w \hspace{0.01cm} \overline{w} \hspace{0.1cm} dS,
    \end{equation*}
    and so
    \begin{equation*}
        \Im A_h^M(w, w) = \Im A_h(w, w),
    \end{equation*}
    which proves (\ref{result}).
    
    To prove the last part of the Proposition, we apply the Cauchy-Schwarz inequality repeatedly to (\ref{eqn:bilinear_new}).
\end{proof}
\begin{prop}
    Provided $M$ is large enough, the discrete problem (\ref{eqn:matrix}) has a unique solution $u_h \in V_p(\bigtau_h)$.
\end{prop}
\begin{proof}
    Assume $A_h^M(u_h,v)=0$ for all $v\in V_p(\bigtau_h)$. Then we have in particular $A_h^M(u_h,u_h)=0$, and so $\Im \hspace{0.05cm} A_h^M(u_h,u_h)=0$. This gives us $\vertiii{u_h}_{\text{TDG}_T}=0$, which implies $u_h=0$, since $\vertiii{\cdot}_{\text{TDG}_T}$ is a norm on the space $V_p(\bigtau_h)$.
\end{proof}

\subsection{A mesh-dependent quasi-optimality estimate}
We now present a quasi-optimality estimate in the mesh-dependent DG norm.
\begin{prop}
    Assume $M$ is sufficiently large and the non trapping conditions (\ref{notrap}) are satisfied. Let $u^M$ be the unique solution of the truncated boundary value problem (\ref{truncated}), and $u_h \in V_p(\bigtau_h)$ the unique solution of the discrete problem (\ref{eqn:matrix}). Then
    \begin{equation} \label{quasiopt}
        \vertiii{u^M - u_h}_{\text{TDG}_T} \leq 2 \inf_{v_h \in V_p(\bigtau_h)} \vertiii{u^M - v_h}_{\text{TDG}_T^+}.
    \end{equation}
\end{prop}
\begin{proof}
    Let $v_h \in V_p(\bigtau_h)$. By Proposition \ref{prop:consistent} and \ref{prop:norma} we have
    \begin{align*}
        \vertiii{u^M - u_h}_{\text{TDG}_T}^2 & \leq -\Im \hspace{0.05cm} A_h^M (u^M - u_h, u^M - u_h) \\
        & \leq | A_h^M (u^M - u_h, u^M - u_h) | \\
        & = | A_h^M (u^M - u_h, u^M - v_h) | \\
        & \leq 2 \vertiii{u^M - u_h}_{\text{TDG}_T} \hspace{0.05cm} \vertiii{u^M - v_h}_{\text{TDG}_T^+}.
    \end{align*}
\end{proof}

\begin{rem}
    The right hand side error of (\ref{quasiopt}) can be controlled with plane waves approximation estimates as in \cite{ZAMP}.
\end{rem}

\newpage
\section{Convergence Experiments}
\subsection{Numerical tests for the DtN-PWDG method}
\begin{figure}[h] 
    \centering
    \includegraphics[width=.85\textwidth]{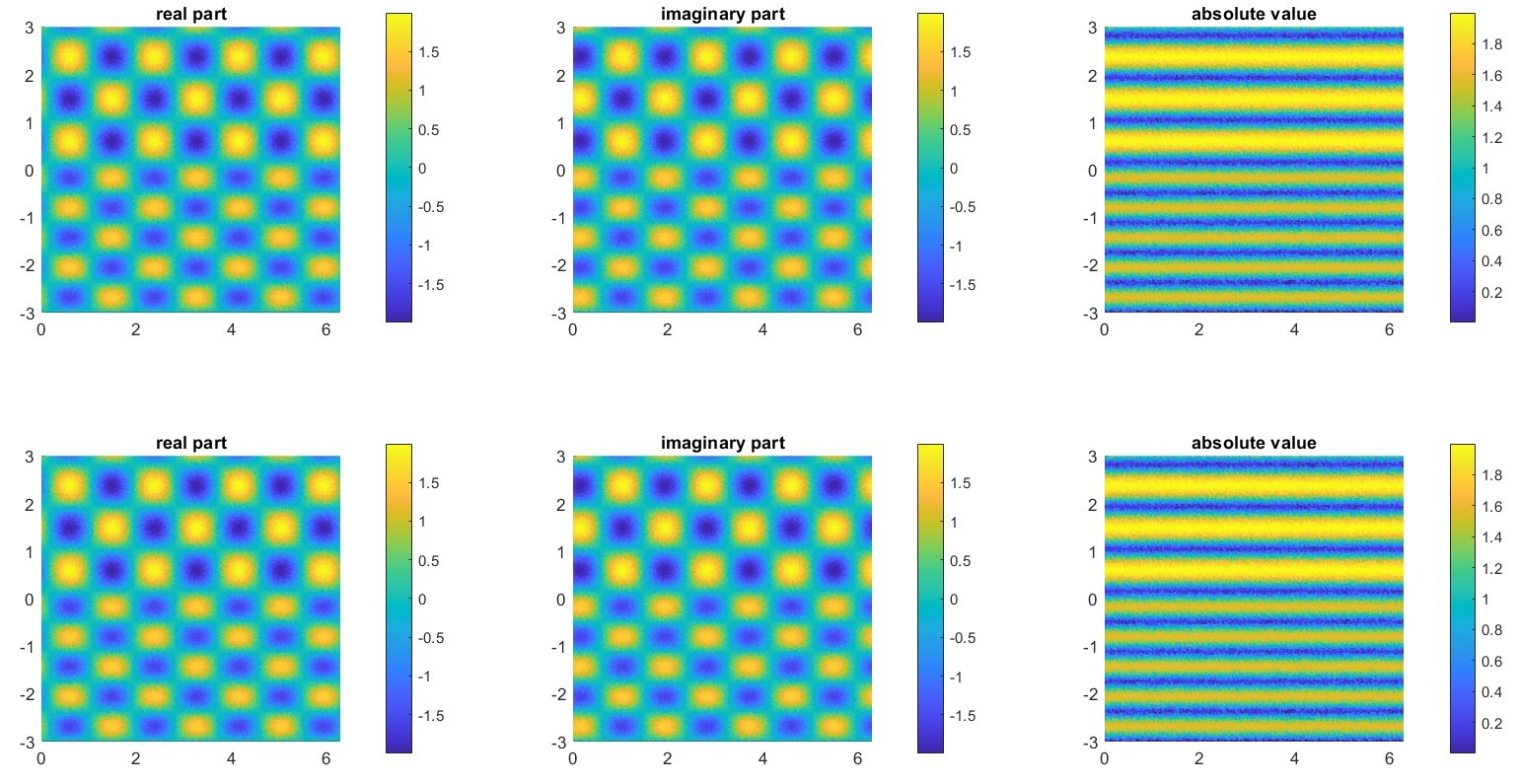}
    \caption{Plots of the numerical solution (top row) against the exact solution (bottom row)}
    \label{fig:DtN_PWDG_4}
\end{figure}
\begin{figure}[h] 
    \centering
    \includegraphics[width=.65\textwidth]{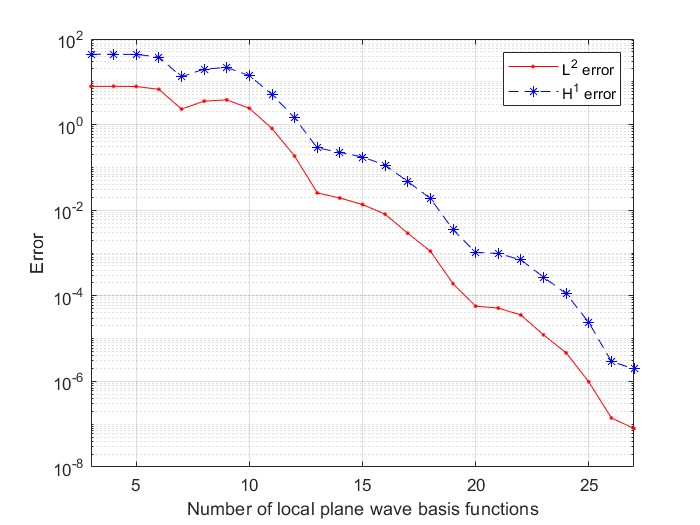}
    \caption{The errors in $L^2$ and $H^1$ norm for $h=1.5$ and $p \in \{3,...,27\}$}
    \label{fig:DtN_PWDG_err_4}
\end{figure} \noindent
To test our method, we consider the usual domain $\Omega=[0,2\pi]\times (-H,H)$, with $H=3$, divided in two regions with different relative permittivity $\varepsilon_1, \varepsilon_2$ by the straight line $x_2=0$. A plane wave is incident on the straight line with angle $\theta \in(-\pi, 0)$. As already seen the exact solution is given by (\ref{sol}). For all the experiments in this section, we consider $h=1.5$ as the mesh parameter, so that the mesh is composed of $56$ triangles.
\begin{example}
   We first consider the case with no absorption of Example \ref{Example_no_abs}: a plane wave is incident with an angle $\theta=- \pi / 4$ and the region $\Omega$ has relative permittivities $\varepsilon_1 =1$ and $\varepsilon_2=\frac{3}{2}$. In Figure \ref{fig:DtN_PWDG_4} and \ref{fig:DtN_PWDG_err_4} we display the plot of the numerical solution and the exact one and the $L^2$ and $H^1$ errors for increasing values of the number of local plane waves functions $p$.
\end{example}
\begin{example} \label{tezst}
   We consider a plane wave incident with an angle $\theta=- \pi / 3$ and a region with relative permittivities $\varepsilon_1 =1$ and $\varepsilon_2=(1.27+0.05i)^2$. This is the same domain of Example \ref{Esempio_cit}. In Figure \ref{fig:DtN_PWDG_3} and \ref{fig:DtN_PWDG_err_3} we display the plot of the numerical solution against the exact one and the $L^2$ and $H^1$ errors for increasing values of the number of local plane waves functions $p$.\\
   From the error plots, we observe an exponential convergence in the number of plane wave directions $p$.\\
\end{example}
\begin{figure}[] 
    \centering
    \includegraphics[width=.85\textwidth]{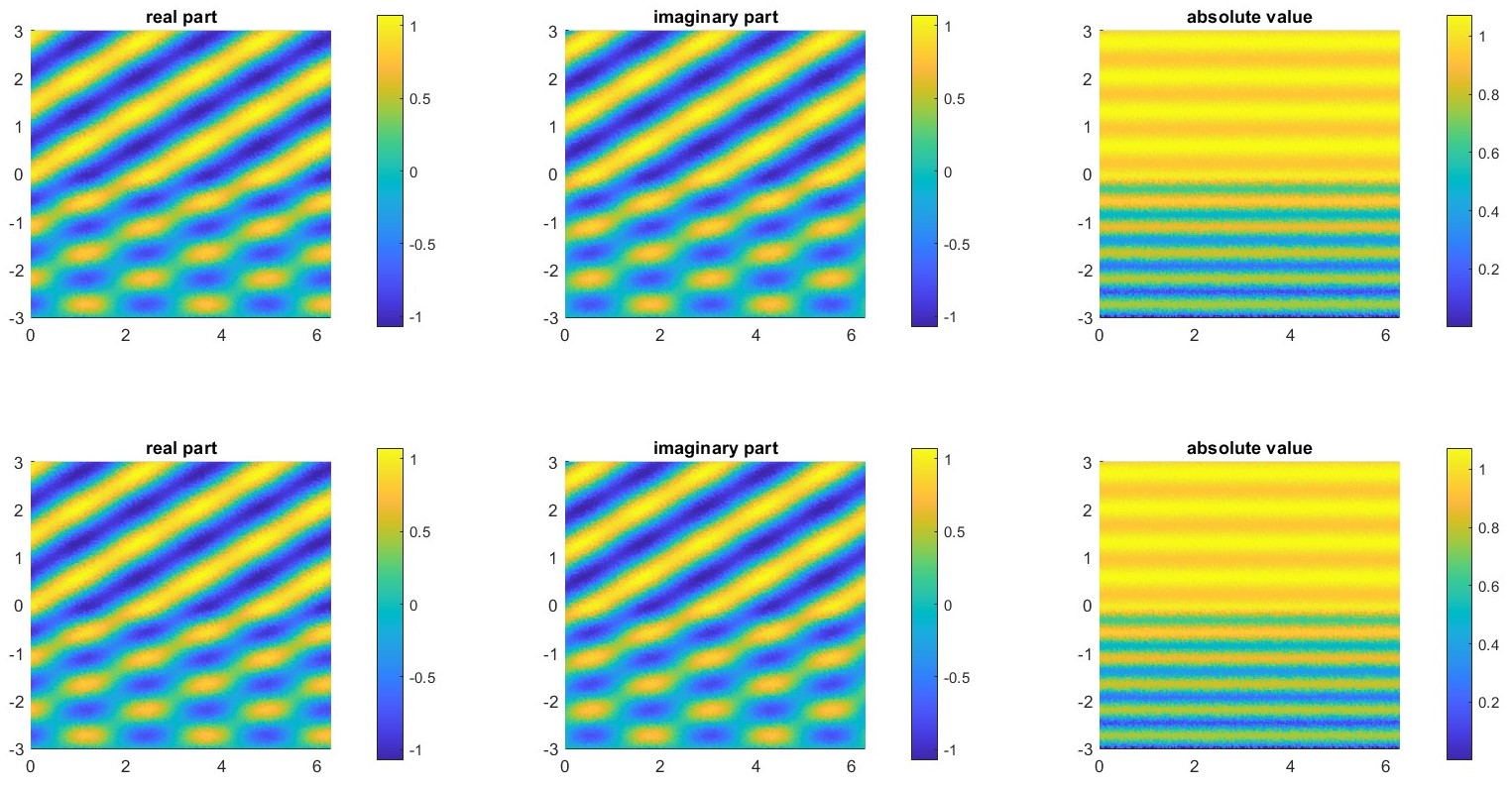}
    \caption{Plots of the numerical solution (top row) against the exact solution (bottom row)}
    \label{fig:DtN_PWDG_3}
\end{figure}
\begin{figure}[] 
    \centering
    \includegraphics[width=.65\textwidth]{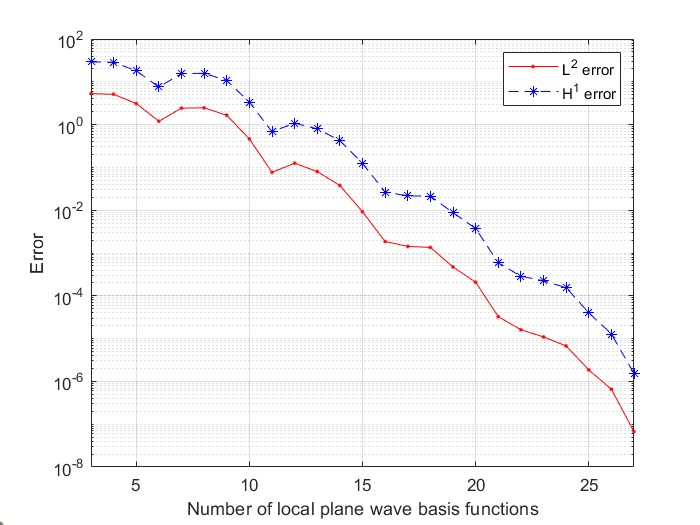}
    \caption{The errors in $L^2$ and $H^1$ norm for $h=1.5$ and $p \in \{3,...,27\}$}
    \label{fig:DtN_PWDG_err_3}
\end{figure}

\subsection{The $h$-version}
We consider the same problem of Example \ref{tezst} and consider the $h$-convergence instead of the $p$-convergence. In this test we fix the number of plane wave basis function inside any element as $p=3$ and the number of Fourier modes $M=100$ and we vary the mesh width $h$. 

In Figure \ref{fig:h_version} we display the results of the numerical experiments: as we can observe, the error decreases with a rate of $h^2$, which is the usual convergence rate for internal boundary value problems with plane waves methods and $p=3$, as in \cite{Gittelson}. This convergence rate is comparable to the one obtained with the P1 version of the DtN-FEM method.
\begin{figure}[] 
    \centering
    \includegraphics[width=.75\textwidth]{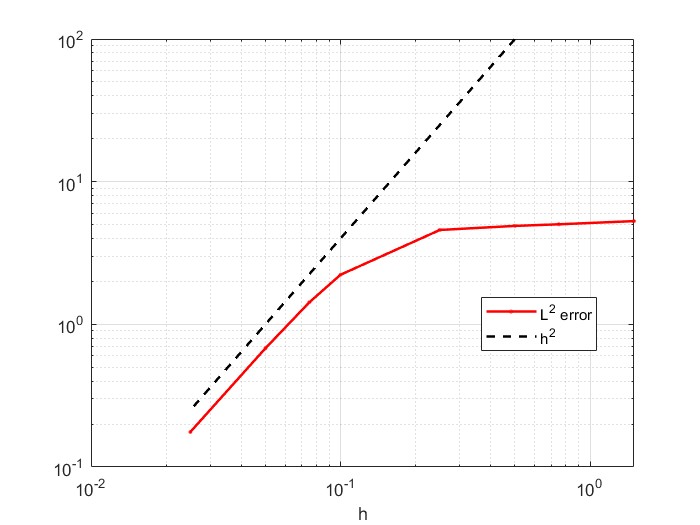}
    \caption{The errors in $L^2(\Omega)$ in the $h$-version}
    \label{fig:h_version}
\end{figure}

\subsection{Convergence in the number of Fourier modes $M$}
In this section, we fix the number of plane wave basis functions $p$ and the mesh width $h$ and vary the number of Fourier modes $M$ for the truncated Dirichlet-to-Neumann operator. We expect the error curve to flatten after a certain value of $M$. We consider the same problem of Example \ref{tezst}, so we have $\theta=-\pi/3$ and $\varepsilon_1 =1$, $\varepsilon_2=(1.27+0.05i)^2$. We choose $p=15$ and $h=1.5$ and vary $M$ from $5$ to $100$.

In Figure \ref{fig:M_conv} we display the results of the numerical experiments: we can observe that initially the error decreases, but after the value $M=50$ the curve flattens and we do not observe a substantial improvement. This is due the fact that for large values of $M$ the contribute given by the new elements of the Fourier expansion is very small, so the operator $T_M$ is converges to the full operator $T$. We also note that the error order remains the same for all the values of $M$, since the error in $M$ gets dominated by the TDG error. 
\begin{figure}[] 
    \centering
    \includegraphics[width=.75\textwidth]{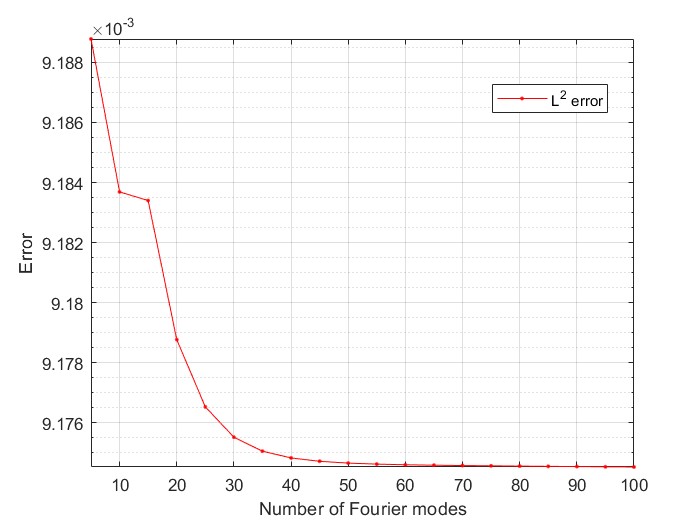}
    \caption{The errors in $L^2(\Omega)$ for different numbers of Fourier modes $M$}
    \label{fig:M_conv}
\end{figure}

\subsection{Comparison between the two PWDG methods}
\begin{figure}[] 
    \centering
    \includegraphics[width=.75\textwidth]{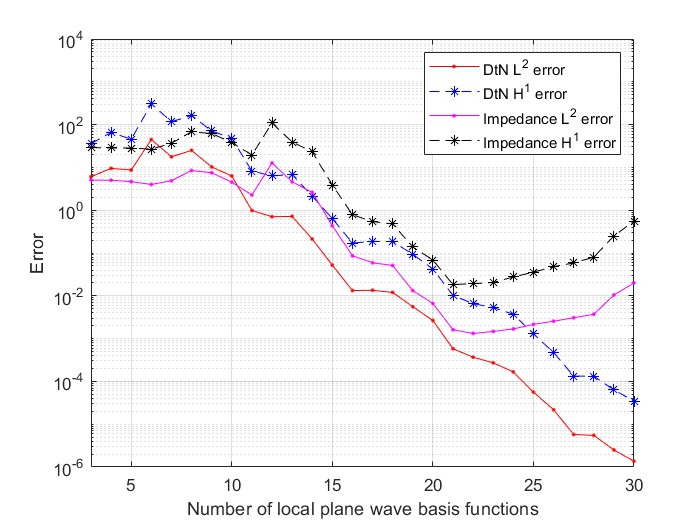}
    \caption{The errors in $L^2$ and $H^1$ norm of the two numerical methods for the different values of $p \in \{3,...,30\}$}
    \label{fig:Confronto}
\end{figure}
Here we compare the two PWDG methods we presented. We consider the usual domain $\Omega=[0,2\pi]\times (-H,H)$, with $H=3$, divided in two different regions with different relative permittivity $\varepsilon_1= 1$ and $\varepsilon_2=(1.8+0.15i)^2$ by the straight line $x_2=0$. A plane wave is incident on the straight line with angle $\theta =-\pi/4$. Note that we are not comparing two methods for the same boundary value problem, but we are comparing the same field $u$ as solution of two different BVPs.

In Figure \ref{fig:Confronto} we display the errors of the DtN-PWDG and Impedance-PWDG methods for increasing values of $p$ and fixed $h=1.5$. 

As we can observe, the DtN-PWDG method performs better than the Impedance-PWDG one and we do not observe the effect of round-off for high values of $p$.

\section{Numerical experiments with different obstacles}
In this section we perform numerical experiments where the usual domain $\Omega$ is divided in different ways. In this experiments we don't know the analytical solution but we can compare our results with others obtained with different numerical methods, such as DtN-FEM or least-squares.
\begin{figure}[] 
    \centering
    \includegraphics[width=.95\textwidth]{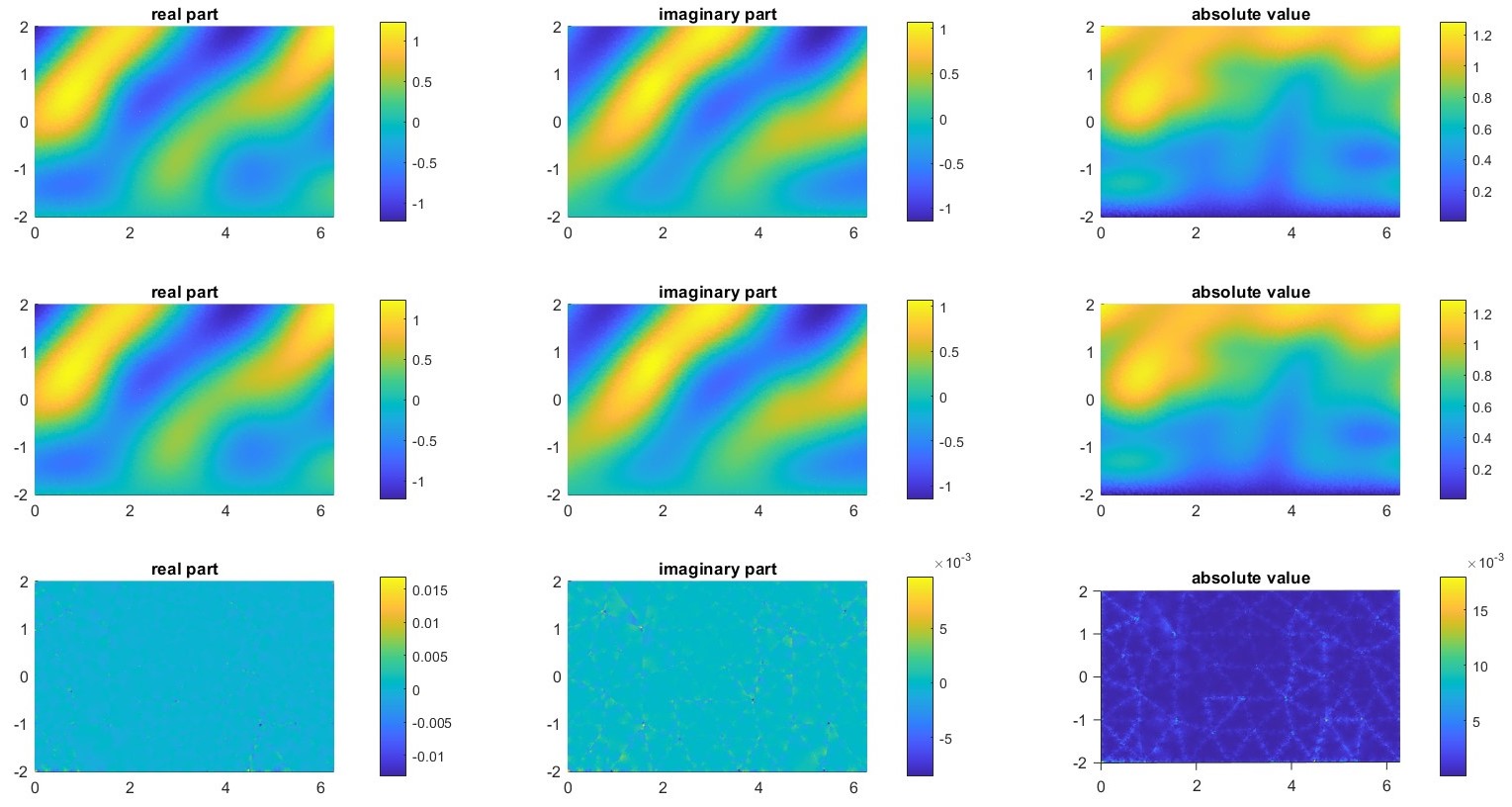}
    \caption{The numerical solution of Example \ref{ex:PWDG_quad} for $h=1.5$, $p=10$ (first row) and $p=20$ (second row) and the difference between the two numerical solutions (bottom row)}
    \label{fig:DtN_PWDG_quad}
\end{figure}
\begin{figure}[]
\centering
\subfloat[][\emph{Real part}]
{\includegraphics[width=.9\textwidth]{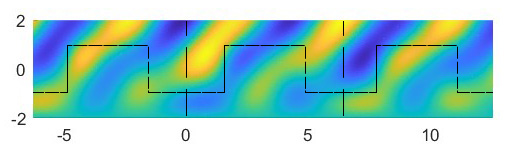}} \quad
\subfloat[][\emph{Imaginary part}]
{\includegraphics[width=.9\textwidth]{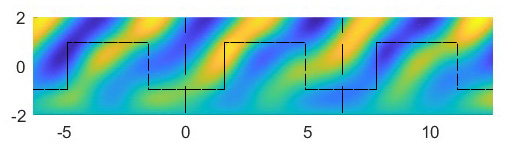}} \quad
\subfloat[][\emph{Absolute value}]
{\includegraphics[width=.9\textwidth]{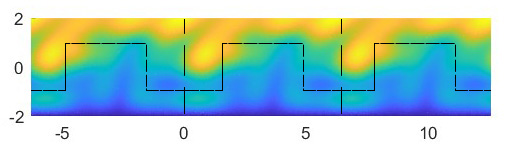}} \quad
\caption{Plots of the approximate solution of Example \ref{ex:PWDG_quad} on the extended domain $[-2\pi,4\pi]\times[-2,2]$. The scatterer profile is plotted as a dotted line}
\label{fig:plots_qp}
\end{figure}
\begin{example} \label{ex:PWDG_quad}
    We consider the domain $\Omega= [0,2\pi] \times [-2,2]$ displayed in Figure \ref{fig:plots_qp}; this is the same domain we used for Example \ref{ex:quad}, divided by a piecewise constant line in two regions with different relative permittivity $\varepsilon=1$ and $\varepsilon_2 = (1.27+0.25i)^2$. A plane wave is incident with angle $\theta= -\pi / 4$ and wavenumber $k=2$. We can compare our numerical solution with the one obtained in \cite{Bao} and the DtN-FEM solution displayed in Figure \ref{fig:quad_plot}. \\
    In Figure \ref{fig:DtN_PWDG_quad} we display the numerical solution for $h=1$ and $p=10$ and $20$. We also plot the difference between the two numerical solution and observe that the error is uniform inside the domain and is mostly distributed on the triangle interfaces, as we can clearly see from the plot of the absolute value.
    
    In Figure \ref{fig:plots_qp} we display the numerical solution obtained with the same parameters and $p=20$ on the extended domain $[-2\pi,4\pi]\times[-2,2]$. We solve the boundary problem on $\Omega= [0,2\pi] \times [-2,2]$ and then we extend the local solution to a global solution using the property of quasi-periodicity. As we can observe from the figure, the solution extends smoothly both on the left and the right boundary of $\Omega$. In particular, from the plot of the absolute value of the solution, we can see the profile of the scatterer.
\end{example} \noindent
\begin{figure}[h]
    \centering
    \includegraphics[width=.5\textwidth]{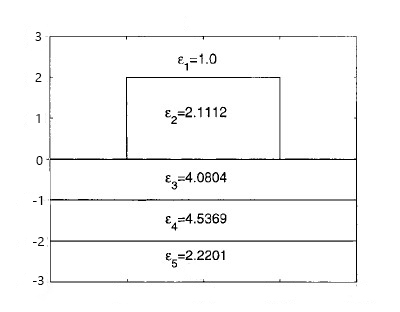}
    \caption{The domain of Example \ref{ex:bao}}
    \label{fig:domain_2}
\end{figure}
\begin{figure}[h]
    \centering
    \includegraphics[width=.99\textwidth]{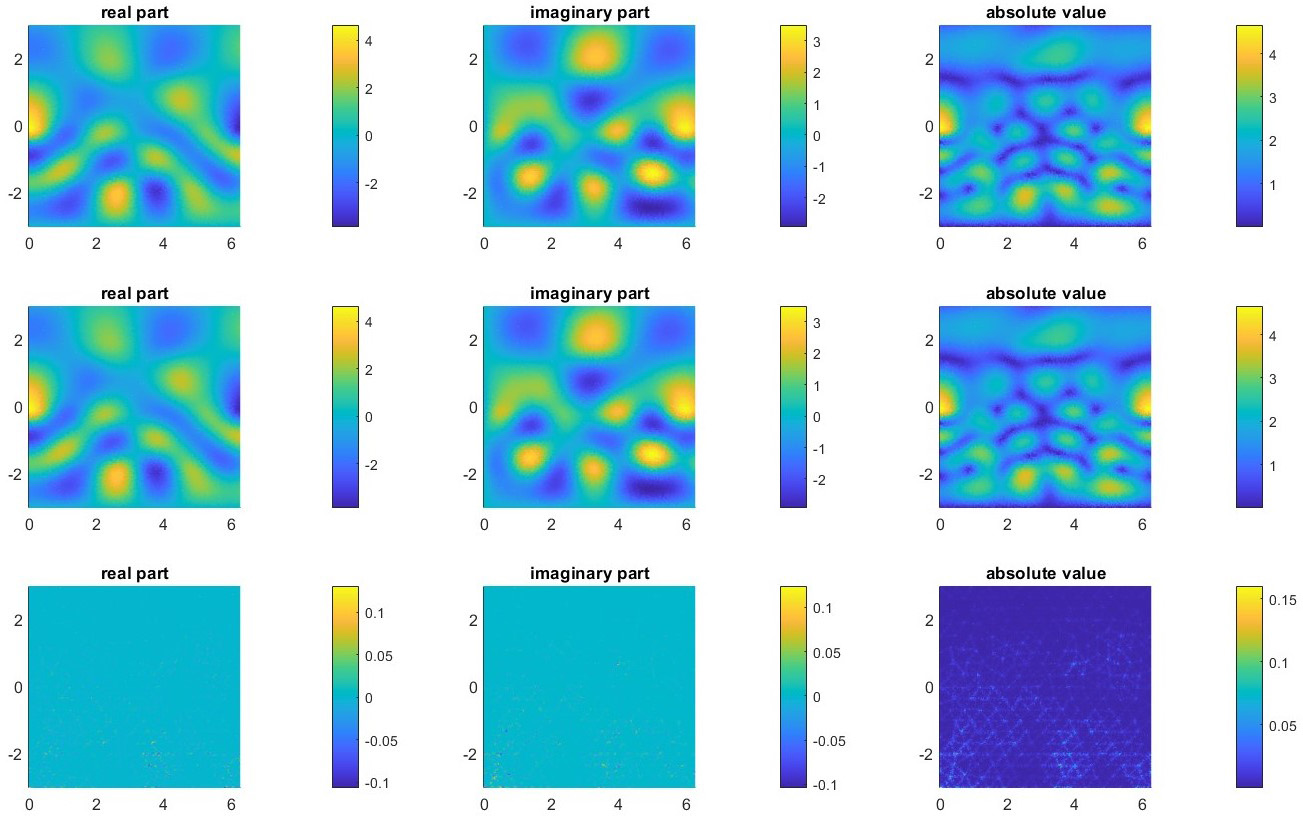}
    \caption{The numerical solution of Example \ref{ex:bao} for $h=0.5$, $p=10$ (first row) and $p=20$ (second row) and the difference between the two numerical solutions (bottom row)}
    \label{fig:bao_plot}
\end{figure}
\begin{example} \label{ex:bao}
    We consider $\Omega=[0,2\pi]\times[-3,3]$ as shown in Figure \ref{fig:domain_2}. This type of grating structure has applications particularly in optical filters and guided mode resonance devices and is also analyzed in \cite{Bao}. The parameters are given as follows: $\varepsilon_1 = 1$, $\varepsilon_2 = 1.49^2$, $\varepsilon_3 = 2.13^2$, $\varepsilon_4 = 2.02^2$, $\varepsilon_2 = 1.453^2$, $\theta=-\pi/4$, $k=2$. 
    
    Figure \ref{fig:bao_plot} shows the plots of the approximate solution for $h=0.5$ and $p=10$ and $20$. We also plot the difference between the two approximate solution and as before we can observe that the relative error is distributed uniformly in $\Omega$ and it is higher on the skeleton of the mesh.
\end{example}

\chapter{Conclusion}
We proved that Plane Wave Discontinuous Galerkin methods are well-suited for the problem of scattering by periodic structures. The DtN boundary condition and the quasi-periodicity can be adapted to the PWDG formulation and provide good numerical results. The convergence tests highlight a behaviour comparable to the one of usual plane wave based methods on bounded domains, which leads to similar estimates for the numerical solution. 

The DtN-PWDG formulation is continue and quasi-optimal for the boundary value problem and it could be used to simulate the grating problem on different domains,

Many extensions are possible, such as a more precise analytical study of the scattering problem, which could lead to explicit stability estimates, as well as a study of the dependence of the variational solution on the scatterer geometry. 

\printbibliography

\end{document}